\documentclass[10pt]{amsart}
\usepackage{amssymb}
\usepackage{amsmath}
\usepackage{mathrsfs}
\usepackage{verbatim}
\usepackage{amsthm}
\usepackage{wasysym}
\usepackage{upgreek}
\usepackage{color}
\usepackage{accents}
\usepackage{fancyhdr}
\usepackage{hyperref}

\numberwithin{equation}{section}

\newtheorem{proposition}{Proposition}[section]
\newtheorem{lemma}[proposition]{Lemma}
\newtheorem{corollary}[proposition]{Corollary}
\newtheorem{theorem}{Theorem}[section]

\theoremstyle{definition}
\newtheorem{definition}{Definition}[section]
\newtheorem{remark}{Remark}[section]

\newcommand{\eqdef}{\overset{\mbox{\tiny{def}}}{=}}
\newcommand{\Rlog}{L}

\newcommand{\fluidnorm}[1]{\mathcal{S}_{#1}}

\newcommand{\fluidenergy}[1]{\mathcal{E}_{#1}}

\newcommand{\speed}{c_s}
\newcommand{\underpartial}{\underline{\partial}}

\newcommand{\DecayFunction}{\mathscr{B}}

\begin{document}

\title{The Stabilizing Effect of Spacetime Expansion on Relativistic Fluids With Sharp Results for the Radiation Equation of State}

\email{jspeck@math.mit.edu}
\author{Jared Speck$^*$}

\begin{abstract}
In this article, we study the $1 + 3$ dimensional relativistic Euler equations on a pre-specified
conformally flat expanding spacetime background with spatial slices that are diffeomorphic to $\mathbb{R}^3.$ We assume that the fluid verifies the equation of state $p = \speed^2 \rho,$ where $0 \leq \speed \leq \sqrt{1/3}$ is the speed of sound. We also assume that the inverse of the scale factor associated to the expanding spacetime metric verifies a $\speed-$dependent
time-integrability condition. Under these assumptions, we use the vectorfield energy method to prove that an explicit family of physically motivated, spatially homogeneous, and spatially isotropic fluid solutions is globally future-stable under small perturbations of their initial conditions. The explicit solutions corresponding to each scale factor are analogs of the well-known spatially flat Friedmann-Lema\^{\i}tre-Robertson-Walker family. Our nonlinear analysis, which exploits dissipative terms generated by the expansion, shows that the perturbed solutions exist for all future times and remain close to the explicit solutions. This work is an extension of previous results, which showed that an analogous stability result holds when the spacetime is exponentially expanding. In the case of the radiation equation of state $p = (1/3)\rho,$ we also show that if the time-integrability condition for the inverse of the scale factor fails to hold, then the explicit fluid solutions are unstable. More precisely, we show that arbitrarily small smooth perturbations of the explicit solutions' initial conditions can launch perturbed solutions that form shocks in finite time. The shock formation proof is based on the conformal invariance of the relativistic Euler equations when $\speed^2 = 1/3,$ which allows for a reduction to a well-known result of Christodoulou.
\end{abstract}

\thanks{$^*$Massachusetts Institute of Technology, Department of Mathematics, Cambridge, MA, USA}

\thanks{The author was supported in part by a Solomon Buchsbaum grant administered by the Massachusetts Institute of Technology, and by an NSF All-Institutes Postdoctoral Fellowship administered by the Mathematical Sciences Research Institute through its core grant DMS-0441170.}



\date{Version of \today}
\maketitle
\setcounter{tocdepth}{2}

\pagenumbering{roman} \tableofcontents \newpage \pagenumbering{arabic}


\section{Introduction and Summary of Main Results} \label{S:INTRO}

In this article, we study the future-global behavior of small perturbations of a family of physically motivated explicit solutions to the $1 + 3$ dimensional relativistic Euler equations. We assume throughout that the fluid \emph{equation of state} is 

\begin{align}
	p = \speed^2 \rho, 
\end{align}
where $p \geq 0$ denotes the \emph{pressure}, $\rho \geq 0$ denotes the \emph{energy density}, and the constant $\speed \geq 0$ is known as the \emph{speed of sound}. The explicit solutions model a spatially homogeneous isotropic fluid of strictly positive energy density evolving in a pre-specified spacetime-with-boundary $([1,\infty) \times \mathbb{R}^3,g)$ which is undergoing expansion. Here, $g$ is an expanding Lorentzian metric of the form \eqref{E:METRICFORM}. As is discussed in Section \ref{SS:PREVIOUSWORK} in more detail, such solutions play a central role in cosmology, where much of the ``normal matter''\footnote{In cosmological models, ``normal matter'' is distinguished from ``dark matter;'' the latter interacts with normal matter indirectly, i.e., only through its gravitational influence.} content of the universe is often assumed to be effectively modeled by a fluid with the aforementioned properties.

Although all speeds verifying $0 \leq \speed^2 \leq 1$ are studied in the cosmology literature, the cases $\speed^2 = 0$ and $\speed^2 = 1/3$ are of special significance. The former is known as the case of the ``pressureless dust,'' and is often assumed to be a good model for the normal matter in the present-day universe. The latter is known as the case of ``pure radiation,'' and is often assumed to be a good model of matter in the early universe. For an introductory discussion of these issues, see \cite[Section 5.2]{rW1984}. As we shall see, in order to prove our main results, we will assume that $\speed^2 \in [0, 1/3].$ We simultaneously analyze the cases $\speed^2 \in (0, 1/3)$ using the same techniques, while the endpoint cases $\speed^2 = 0$ and $\speed^2 = 1/3$ require some modifications; when $\speed^2 =0,$ $\rho$ loses one degree of differentiability relative to the other cases, while the analysis of the case $\speed^2 = 1/3$ is based on the fact that the relativistic Euler equations are conformally invariant in this case (see Section \ref{S:PURERADIATION}). We will also provide some heuristic evidence of fluid instability when $\speed^2 > 1/3;$ (see Section \ref{SSS:EXPANSIONRESTRICTIONROLE}). 

We assume throughout that the pre-specified \emph{spacetime metric} $g$ is of the form

\begin{align} \label{E:METRICFORM}
	g & = -dt^2 + e^{2\Omega(t)} \sum_{j=1}^3 (dx^j)^2, && \Omega(1) = 0,
\end{align}
where $e^{\Omega(t)} > 0$ is known as the \emph{scale factor}, $t \in [1,\infty)$ is a time coordinate, and $(x^1,x^2,x^3)$ are standard  coordinates on $\mathbb{R}^3.$ A major goal of this article is to study the influence of the scale factor on
the long-time behavior of the fluid. Our analysis will heavily depend on the assumptions we make on $e^{\Omega(t)},$ which are stated later in this section. Metrics of the form \eqref{E:METRICFORM} are of particular importance because experimental evidence indicates that universe is expanding and approximately spatially flat \cite{gEjBsW1992}. According to Steven Weinberg \cite[pg. 1]{sW2008}, ``Almost all of modern cosmology is based on [metrics of the form \eqref{E:METRICFORM}].'' The physical significance of these metrics is discussed more fully in Section \ref{SS:PREVIOUSWORK}. 

Under equations of state of the form $p = \speed^2 \rho,$ the relativistic Euler equations in the spacetime-with-boundary $([1,\infty) \times \mathbb{R}^3, g)$ are the four equations

\begin{align} \label{E:DIVT0}
	D_{\alpha} T^{\alpha \mu} & = 0, && (\mu = 0,1,2,3),
\end{align}   
where

\begin{align} \label{E:EULEREMT}
	T^{\mu \nu} & = (\rho + p)u^{\mu} u^{\nu} + p (g^{-1})^{\mu \nu}
\end{align}
is the energy-momentum tensor of a perfect fluid, $D$ is the Levi-Civita connection of $g,$
and $u^{\mu}$ is the \emph{four-velocity}, a future-directed ($u^0 > 0$) vectorfield normalized by 

\begin{align} \label{E:UNORMALIZED}
	g_{\alpha \beta}u^{\alpha} u^{\beta} & = - 1.
\end{align}
Equivalently, the following equation for $u^0$ holds in our coordinate system:

\begin{align} \label{E:U0UPPERISOLATED}
	u^0 & = (1 + g_{ab}u^a u^b)^{1/2}.
\end{align}

Under the above assumptions, it is well-known (see e.g. \cite{jS2011}) that when $\rho > 0,$ the relativistic Euler equations \eqref{E:DIVT0} can be written in the equivalent form

\begin{subequations}
\begin{align}
	u^{\alpha} D_{\alpha} \ln \rho + (1 + \speed^2) D_{\alpha} u^{\alpha} & = 0, && \label{E:EULERINTROP} \\
 	u^{\alpha} D_{\alpha} u^{\mu} + \frac{\speed^2}{1 + \speed^2} \Pi^{\alpha \mu} D_{\alpha} \ln \rho & = 0, && (\mu= 0,1,2,3), 
		\label{E:EULERINTROU} \\
	\Pi^{\mu \nu} & \eqdef u^{\mu} u^{\nu} + (g^{-1})^{\mu \nu}, && (\mu, \nu = 0,1,2,3).  \label{E:PIIINTRO}
\end{align}
\end{subequations}
The tensorfield $\Pi^{\mu \nu}$ projects onto the $g-$ orthogonal complement of $u^{\mu}.$ We remark that the system 
\eqref{E:EULERINTROP} - \eqref{E:EULERINTROU} is consistently overdetermined in the sense that the $0$ component of \eqref{E:EULERINTROU} is a consequence of the remaining (including \eqref{E:U0UPPERISOLATED}) equations.

Other important experimentally determined facts are: on large scales, the matter content of the universe appears to be approximately spatially homogeneous (see e.g. \cite{jYsBbPtS2005}) and isotropic (see e.g. \cite{kKsWoLmR1999}). These conditions can be modeled by the following explicit homogeneous/isotropic fluid ``background'' solution:

\begin{align}
	\widetilde{\rho} \eqdef \bar{\rho} e^{-3(1 + \speed^2)\Omega(t)}, \qquad \widetilde{u}^{\mu} & \eqdef \delta_0^{\mu}, 
		\label{E:BACKGROUNDU}
\end{align}
where $\bar{\rho} > 0$ is a constant. It follows easily from the discussion in Section \ref{SS:ALTFORM} that \eqref{E:BACKGROUNDU} are solutions to \eqref{E:EULERINTROP} - \eqref{E:EULERINTROU}
(where $(t,x^1,x^2,x^3)$ are coordinates such that the metric $g$ has the form \eqref{E:METRICFORM}); see Remark \ref{R:BACKGROUNDSOLUTION}. 

The main goal of this article is to understand the future-global dynamics of solutions to \eqref{E:EULERINTROP} - \eqref{E:EULERINTROU} that are launched by small perturbations of the initial data corresponding the explicit solutions \eqref{E:BACKGROUNDU}. In order to prove our main stability results, we make the following assumptions on the scale factor. 

\begin{center}
	{\Large \textbf{Assumptions on} $e^{\Omega(t)}$}
\end{center}

\begin{align}
	& \mbox{$e^{\Omega(\cdot)} \in C^1\big([1,\infty),[1,\infty)\big)$
		and increases without bound as $t \to \infty$} \tag*{\textbf{A1}} \label{A:A1} \\
	& \mbox{$\left\lbrace \begin{array}{ll} 
	\int_{s = 1}^{\infty} e^{- 2 \Omega(s)} \, d s < \infty, & \speed^2 = 0, \\
	 \int_{s = 1}^{\infty} e^{- \Omega(s)} \DecayFunction(\Omega(s)) \, d s < \infty, & 0 < \speed^2 < 1/3, \\
	 \int_{s = 1}^{\infty} e^{- \Omega(s)} \, ds  < \infty, & \speed^2 = 1/3 
	 \end{array} \right.$} \tag*{\textbf{A2}} \label{A:A2} \\
	& \mbox{$\left\lbrace \begin{array}{ll} 
	 \mbox{$\DecayFunction(\cdot) \in C^1\big([1,\infty),(0,\infty)\big)$ and increases without bound}, & 0 < 	
	 	\speed^2 < 1/3,  \\
	 \frac{\DecayFunction'(\Omega)}{\DecayFunction(\Omega)} \leq (1 - 3 \speed^2) \ \mbox{for all large} \ \Omega, & 0 < 	
	 	\speed^2 < 1/3, \\
	 \int_{\Omega = 1}^{\infty} \frac{d \Omega}{\DecayFunction(\Omega)}  < \infty, & 0 < \speed^2 < 1/3.
	 \end{array} \right.$} \tag*{\textbf{A3}} \label{A:A3}
\end{align}


We now roughly summarize our main results. See Theorem \ref{T:GLOBALEXISTENCE} and Corollary \ref{C:SHOCKSCANFORM} for precise statements.

\begin{changemargin}{.25in}{.25in} 
\textbf{Main Results.} \
Assume that the scale factor of $g$ verifies assumptions \ref{A:A1} - \ref{A:A3} and that the equation of state
$p = \speed^2 \rho$ holds, where $0 \leq \speed^2 \leq 1/3.$ Then the explicit solution $(\widetilde{\rho},\widetilde{u}^{\mu})$ defined in \eqref{E:BACKGROUNDU} is globally future-stable. More precisely, if the initial data corresponding to the explicit solution are perturbed by a sufficiently small element of a suitable Sobolev space, then the corresponding solution of \eqref{E:EULERINTROP} - \eqref{E:EULERINTROU} exists classically for 
$(t,x^1,x^2,x^3) \in [1,\infty) \times \mathbb{R}^3$ and remains globally close to the explicit solution. In particular, no future-shocks form in such solutions.

In contrast, if $\speed^2 = 1/3,$ assumption \ref{A:A1} holds, but $\int_{s = 1}^{\infty} e^{- \Omega(s)} \, d s = \infty,$ then the solution $(\widetilde{\rho},\widetilde{u}^{\mu})$ is nonlinearly unstable. In particular, there exists an open family of initial data containing arbitrarily small smooth perturbations of the explicit data whose corresponding perturbed solutions develop shock singularities in finite time.
\end{changemargin}

\begin{remark}
	Since we prove our stability results through the use of energy methods (which are stable), the precise form of the metric 
	\eqref{E:METRICFORM} is not essential and was chosen only out of convenience; our stability results would also hold 
	under suitably small perturbations of $g$ (belonging e.g. to a Sobolev-type space of sufficiently high order). However, it is 
	not clear whether or not the shock formation result is stable under perturbations of the metric; see Remark \ref{R:UNSTABLEPROOF}.
\end{remark}

\begin{remark}
As we will see in Proposition \ref{P:DECOMPOSITION} and Corollary \ref{C:ISOLATEPARTIALTPUJ} (see also Section \ref{SSS:EXPANSIONRESTRICTIONROLE}), spacetime expansion generates \emph{dissipative} terms in the relativistic Euler equations. This dissipative mechanism is the main reason we are able to prove our future stability results. A related fact is that our proof of future stability also applies to the spacetime-with-boundary $(g,[1,\infty) \times \mathbb{T}^3).$ We remark that a framework for proving small-data global existence for quasilinear hyperbolic PDEs in the presence of strong dissipation has been developed in e.g. \cite{lHs1997} (see also \cite{tN1978} for the case of quasilinear wave equations), but its assumptions roughly coincide with the dissipative effect generated by $e^{\Omega(t)} = e^{Ht}$ (where $H>0$ is a constant) and
are therefore much stronger than \ref{A:A1} - \ref{A:A3}. 
\end{remark}

Some important examples of scale factors that appear in the cosmology literature and that verify \ref{A:A1} - \ref{A:A3} include \textbf{i)} the case of exponential expansion, in which $e^{\Omega(t)} \sim e^{Ht}$ for some ``Hubble'' constant $H > 0$ (choose $\DecayFunction(\Omega) = e^{q \Omega}$ for some small constant $q > 0$ when $0 < \speed^2 < 1/3$). Future stability in this case is implied by the results proved in \cite{iRjS2009}, \cite{jS2011} under the assumption $0 < \speed^2 < 1/3.$ \textbf{ii)} the case of accelerated power-law expansion, in which $e^{\Omega(t)} \sim t^{Q}$ for some constant $Q > 1$ (choose $\DecayFunction(\Omega) = e^{q \Omega}$ for some small constant $q> 0$ when $0 < \speed^2 < 1/3$). The significance of these expansion rates will be discussed in more detail in Section \ref{SS:PREVIOUSWORK}.

\subsection{Motivation and connections to previous work} \label{SS:PREVIOUSWORK}

We are interested in the system \eqref{E:EULERINTROP} - \eqref{E:EULERINTROU} because it plays a fundamental role in the standard model of cosmology. More precisely, physicists/cosmologists often couple the Einstein equations of general relativity to the relativistic Euler equations and sometimes also to additional matter models. They then study explicit solutions of the resulting coupled system and make predictions about the long-time behavior of the universe based on the properties of the explicit solutions. Now a basic requirement for the explicit solutions to have any long-time predictive value \emph{is that they be future-stable under small perturbations of their data}; this basic requirement is the primary motivation behind our investigation.

In the coupled problem mentioned above, the spacetime metric $g$ is not pre-specified, but is instead one of the dynamic variables. In this article, we don't address the coupled problem, but rather the behavior of the fluid matter under the influence of a large class of pre-specified spacetime metrics $g.$ Our hypotheses \ref{A:A1} - \ref{A:A3} on $g$ are meant to roughly capture the main features of the kinds of metrics that may arise in a large class of physically motivated expanding solutions to the coupled problem. Readers may consult e.g. \cite{aR2005b}, \cite{aR2006b} for background information on the coupled Einstein-matter problem in the presence of expansion. The physical relevance of these expanding metrics is the following: experimental evidence suggests that our own universe seems to be undergoing \emph{accelerated expansion}. There are many experimental references available on the issue of accelerated spacetime expansion; \cite{aRaF1998}, \cite{aRsP1999} are two examples. However, as noted in e.g. \cite{aR2005a}, \cite{aR2006b}, \cite{hR2009}, \textbf{the precise rate and mechanism of the expansion of our universe have yet to be determined}. This uncertainty has motivated us to make the very general hypotheses \ref{A:A1} - \ref{A:A3}. We expect that these hypotheses will be roughly consistent with the behavior of most of the physically relevant metrics $g$ that arise in the study of accelerated expansion in the coupled problem. Now for metrics of the form \eqref{E:METRICFORM}, accelerated expansion (i.e. $\frac{d}{dt} e^{\Omega(t)}, \frac{d^2}{dt^2} e^{\Omega(t)} > 0$) is equivalent to (where $\omega(t) \eqdef \frac{d}{dt} \Omega(t)$)

\begin{enumerate}
	\item $\omega > 0$
	\item  $\frac{d}{dt} \omega + \omega^2 > 0.$
\end{enumerate}
We note that these conditions are not equivalent to \ref{A:A1} - \ref{A:A3}; it is not difficult to see that $(1)- (2)$ may hold even if \ref{A:A2} - \ref{A:A3} fail (for all functions $\DecayFunction(\cdot)$), and conversely \ref{A:A1} - \ref{A:A3} may hold without the expansion being accelerated for all time (we don't even need $\omega$ to be differentiable in order to prove our results). Nevertheless, many of the scale factors appearing in the cosmology literature verify $(1) - (2)$ and also \ref{A:A1} - \ref{A:A3}. Later in this section we will further discuss the two previously mentioned examples of exponential expansion and accelerated power-law expansion.

Our main results are further motivated by the following: based on our previous experience \cite{iRjS2009}, \cite{jS2011},
in certain accelerated expanding regimes, we expect that a nearly quiet fluid's influence on $g$ in the coupled problem is often lower-order in comparison to the other influences (including the primary influences driving the expansion). In other words, as long as the fluid behavior is tame, we expect that the coupled problem effectively behaves as a partially decoupled system in that $g$ can highly influence the fluid, but not the other way around\footnote{This expectation is false for the Friedmann-Lema\^{\i}tre-Robertson-Walker family of solutions to the Euler-Einstein system with no cosmological constant or additional matter model. In this case, the fluid itself is the primary influence driving the expansion, which is not accelerated.}. Thus, we have the following moral consequence of our \textbf{Main Results}: we expect that the addition of a perfect fluid (verifying $p= \speed^2 \rho,$ $0 \leq \speed^2 \leq 1/3)$ would not destroy the stability of a stable solution to a coupled Einstein-matter system, as long as the stable metrics are expanding at rates in the spirit of \ref{A:A1} - \ref{A:A3}. As we will soon discuss, this expectation has previously been confirmed when the expansion is generated by the inclusion of a positive cosmological constant in the Einstein equations; in this case the expansion is very rapid and verifies
$e^{\Omega(t)} \sim e^{Ht},$ where $H > 0$ is a constant.

The key question from the previous paragraph (which is addressed by our main results) is \emph{whether or not the fluid behavior remains tame if it is initially tame}. This is a highly non-trivial question whose answer depends on the expansion properties of $g.$ In the case of Minkowski spacetime (where $e^{\Omega(t)} \equiv 1$), Christodoulou's monograph \cite{dC2007} demonstrated that under a general physical equation of state, the constant solutions (i.e., $\rho \equiv \bar{\rho} > 0,$ $u^{\mu} \equiv \delta_0^{\mu}$) to the relativistic Euler equations are unstable\footnote{There was one exceptional equation of state which features small-data global existence when the fluid is irrotational; this equation of state leads (in the irrotational case) to the quasilinear wave equation for minimal surface graphs.}. More precisely, there exists an open family of initial data whose corresponding solutions develop shock singularities in finite time. Furthermore, the family contains arbitrarily small (non-zero) perturbations of the constant state, including irrotational data. We state some of Christodoulou's main results in more detail in Theorem \ref{T:CHRISTODOULOUSHOCK} below. One very important aspect of his monograph is that it provides a complete description of the nature of the shock (we do not recall the fully detailed picture in Theorem \ref{T:CHRISTODOULOUSHOCK}). His work can be therefore be viewed as a major extension of the well-known article \cite{tS1985}, in which Sideris showed that singularities can form in the non-relativistic Euler equations under the polytropic equations of state for arbitrarily small perturbations of constant background solutions; in \cite{tS1985}, a detailed picture of the singularity formation was not obtained because the argument involved the analysis of averaged quantities. 

In contrast to Christodoulou's shock formation result, the works \cite{iRjS2009}, \cite{jS2011} of Rodnianski and the author
showed the following: when the scale factor verifies $e^{\Omega(t)} \sim e^{Ht}$ for some ``Hubble constant'' $H > 0,$ the equation of state is $p= \speed^2 \rho,$ and $0 < \speed^2 < 1/3,$ then $(g,\widetilde{\rho},\widetilde{u})$ (defined in \eqref{E:METRICFORM} and \eqref{E:BACKGROUNDU}) is a future-stable solution to the coupled Euler-Einstein equations. The exponential expansion $e^{\Omega(t)} \sim e^{Ht}$ of $g$ in these solutions is driven by the inclusion of a \emph{positive cosmological constant} $\Lambda$ in the Einstein equations\footnote{In the coupled problem, the scale factor corresponding to the explicit solution metric $g$ verifies \emph{Friedmann's ODE}, which leads to the asymptotic behavior $e^{\Omega(t)} \sim e^{Ht}.$}, where $H = \sqrt{\Lambda/3}.$ These special ``Friedmann-Lema\^{\i}tre-Robertson-Walker'' solutions $(g,\widetilde{\rho},\widetilde{u})$ lie at the heart of many cosmological predictions. Furthermore, the case $\speed^2 = 1/3$ was recently addressed in \cite{cLjK2011} using an extension of Friedrich's \emph{conformal method} (see e.g. \cite{hF2002}), which is closely related to the methods we use in Section \ref{S:PURERADIATION}. Roughly speaking, the conformal method is a collection of techniques for transforming the question of global-in-time existence into a question of local-in-time existence for the \emph{conformal field equations} through changes of variables. This method often works when the energy-momentum tensor of the matter model is trace-free (as is the case of the fluid energy-momentum tensor \eqref{E:EULEREMT} when $\speed^2 = 1/3.$) In summary, in the mild initial condition regime, previous results show that \emph{exponential expansion suppresses the formation of fluid shocks} when $0 < \speed^2 \leq 1/3.$ 

We now mention another well-studied mechanism for generating accelerated expansion in solutions to the coupled problem: the inclusion of a scalar field matter model in the Einstein-matter equations. In the simplest cases, the scalar field $\Phi$ is
postulated to verify a wave equation of the form $(g^{-1})^{\alpha \beta} D_{\alpha} D_{\beta} \Phi = V'(\Phi),$ where $V(\Phi)$ is a nonlinearity. As discussed in e.g. \cite{aR2004a}, \cite{aR2005a}, by adjusting $V,$ one is able to exert control over the scale factor $e^{\Omega},$ at least for special solutions. Let us briefly discuss two interesting examples that lead to expansion rates verifying our hypotheses \ref{A:A1} - \ref{A:A3}. First, in \cite{hR2008}, Ringstr\"{o}m studied nonlinearities that verify $V(0) > 0,$ $V'(0) = 0,$ $V''(0) > 0.$ For small $\Phi,$ the conditions on $V$ generate a positive cosmological-constant-like effect, and the coupled Einstein-scalar field system effectively behaves (for $\Phi \sim 0$)
as if one had introduced the cosmological constant $\Lambda = V(0)$ into the Einstein equations. As discussed in the previous paragraph, this generates exponential expansion. The main result of \cite{hR2008} was a proof of the future stability of a large family of background solutions to the coupled Einstein-scalar field system (with no fluid present). As a second example, in \cite{hR2009}, Ringstr\"{o}m studied the Einstein-scalar field system under the choice $V(\Phi) = V_0 e^{- \uplambda \Phi},$ where $V_0$ and $\uplambda$ are positive constants. For certain explicit solutions, this choice of $V$ generates power-law expansion $e^{\Omega(t)} \sim t^Q,$ where $Q = \sqrt{\uplambda},$ and the main result of \cite{hR2009} was a proof of the future stability of these explicit solutions when $Q > 1$ (i.e., when the expansion is accelerated).

\subsection{Outline of the analysis}

\subsubsection{The continuation principle}
Our proof of the future stability of the explicit solutions \eqref{E:BACKGROUNDU} 
is based on a continuation principle (Proposition \ref{P:CONTINUATION}) together with a standard bootstrap argument. With the aid of Sobolev embedding, the content of the continuation principle can be roughly summarized as follows: if the solution forms a singularity at time $T_{max},$ then a certain order-$N$ (throughout the article, $N \geq 3$ is a fixed integer) Sobolev norm $\fluidnorm{N}$ (see Definition \ref{D:NORMS}) of the solution necessarily blows-up at time $T_{max}.$ Therefore, if one can derive a priori estimates guaranteeing that $\fluidnorm{N}(t)$ remains finite for all $t \geq 1,$ then one has demonstrated future-global existence. To show that $\fluidnorm{N}(t)$ remains finite, we will assume that

\begin{align} \label{E:BOOTSTRAPINTRO}
	\fluidnorm{N}(t) \leq \epsilon
\end{align}
holds on a time interval $[1,T)$ of existence, where $\epsilon$ is a sufficiently small positive number; by standard local existence theory, this assumption holds for short times when the data are sufficiently small. Based on this assumption, we will then derive inequalities for $\fluidnorm{N}(t)$ leading to an \emph{improvement} on the assumption \eqref{E:BOOTSTRAPINTRO}, as long as the data and $\epsilon$ are sufficiently small. Consequently, since $\fluidnorm{N}(t)$ is continuous, it can never exceed a certain size, and future stability follows from the continuation principle. The complete argument is provided in
the proof of Theorem \ref{T:GLOBALEXISTENCE}.

\subsubsection{Energy currents} \label{SSS:CURRENTS}
The only general method we know of for estimating solutions to quasilinear hyperbolic PDEs (such as the relativistic Euler equations) involves the derivation of coercive integral identities. This framework is often referred to as the \emph{energy method}, and its raison d'\^{e}tre is connected to the following basic difficulty: the time derivative of $\fluidnorm{N}(t)$ cannot be controlled in terms of the solution itself, for it is a crudely defined Sobolev norm which has no special structure tying it to the evolution of the fluid solution. To estimate $\fluidnorm{N}(t),$ we will construct coercive quantities $\fluidenergy{N}(t)$ called \emph{energies}. $\fluidenergy{N}(t)$ will be used to control the $L^2$ norms of the up-to-top-order 
(i.e., up-to-$N^{th}-$order) spatial derivatives of the solution via the divergence theorem (see \eqref{E:DIV}). We remark that the coercive nature of $\fluidenergy{N}(t)$ is established in Proposition \ref{P:ENERGYNORMCOMPARISON}. We also remark that in the case $\speed^2 = 0,$ we will define separate energies $\mathcal{E}_{N;velocity}$ and $\mathcal{E}_{N-1;density}$ for the four-velocity and the density. The reason is that in this case, the evolution of the four velocity decouples from that of the density, and furthermore, the density is necessarily one degree less differentiable than the four-velocity.

To construct the energies, we will use a version of the \emph{method of vectorfield multipliers}. Over the last few decades, this fundamental method has been applied in many different contexts to a large variety of hyperbolic PDEs. All applications of this method are essentially extensions of Noether's theorem \cite{eN1918}. In particular, we make use of the multiplier method framework developed in \cite{dC2000}, a work that can be viewed as a geometric extension of Noether's theorem to handle \emph{regularly hyperbolic} (see \cite{dC2000} for Christodoulou's definition of regular hyperbolicity) systems/solutions that are derivable from a Lagrangian but perhaps lack symmetry (and hence the energy estimates contain error terms). The multiplier method allows us to construct \emph{compatible ``energy'' currents} with the help of an auxiliary \emph{multiplier vectorfield}; we will further discuss the specific multiplier used in this article later in this section. The energy currents are special solution-dependent vectorfields $\dot{J}^{\mu}[\partial_{\vec{\alpha}} \mathbf{W}, \partial_{\vec{\alpha}} \mathbf{W}],$ and the energies are 

\begin{align} \label{E:ENERGYINTRO}
	\fluidenergy{N}(t) \eqdef \left(\sum_{|\vec{\alpha}| \leq N} \int_{\mathbb{R}^3} \dot{J}^0[\partial_{\vec{\alpha}} \mathbf{W}, 	\partial_{\vec{\alpha}} \mathbf{W}] \, d^3 x \right)^{1/2}. 
\end{align}	
Here, $\partial_{\vec{\alpha}}$ is an order $\leq N$ spatial differential operator, $\mathbf{W} = (\ln(e^{3(1 + \speed^2)\Omega} \rho/\bar{\rho}),u^1,u^2,u^3)$ is an array of fluid variables, and the precise definition of $\dot{J}^{\mu}$ is given in Definition \ref{D:ENERGYCURRENT}. As is suggested by the notation, $\dot{J}^{\mu}[\partial_{\vec{\alpha}} \mathbf{W}, \partial_{\vec{\alpha}} \mathbf{W}]$ depends quadratically on $\partial_{\vec{\alpha}} \mathbf{W}.$ Furthermore, the coefficients of the quadratic terms depend on the solution $\mathbf{W}$ itself. 


In this article, we will not fully explain our construction of the currents $\dot{J}^{\mu}.$ Roughly speaking, these currents exist because the relativistic Euler equations are the Euler-Lagrange equations corresponding to a Lagrangian. For such PDEs, Christodoulou's aforementioned framework \cite{dC2000} explains how one can construct suitable currents $\dot{J}^{\mu}$ using a version of the multiplier method. In fact, energy currents for the relativistic Euler equations were first derived by Christodoulou in the Lagrangian-coordinate framework\footnote{Interestingly, the relativistic Euler equations fall just outside of the scope of the methods of \cite{dC2007}; in the Lagrangian-coordinate framework, the methods of \cite{dC2007} lead to semi-coercive estimates rather than fully coercive estimates. Nevertheless, we are able to derive fully coercive estimates in the Eulerian-coordinate framework.} in \cite{dC2000}, and later in the Eulerian-coordinate framework in \cite{dC2007}. The Eulerian-coordinate fluid energy current framework has since been applied by the author and his collaborator in various contexts; see \cite{jS2008b}, \cite{jS2008a}, \cite{jS2011}, and \cite{jSrS2011}. 

Assuming that $0 < \speed^2 \leq 1/3$ and that \eqref{E:BOOTSTRAPINTRO} holds for sufficiently small $\epsilon,$ 
we will exploit the following two fundamental properties of $\dot{J}^{\mu}:$ 

\begin{enumerate}
	\item $\dot{J}^0[\partial_{\vec{\alpha}} \mathbf{W}, \partial_{\vec{\alpha}} \mathbf{W}]$ 
	 $\approx \Big( \partial_{\vec{\alpha}} \ln \left(e^{3(1 + \speed^2)\Omega} \rho/\bar{\rho} \right)\Big)^2$
	 	$+ e^{2\Omega} \sum_{j=1}^3 (\partial_{\vec{\alpha}} u^j)^2$ 
		(see the proof of Proposition \ref{P:ENERGYNORMCOMPARISON}).
	\item For a solution $\mathbf{W},$ it can be shown that \\
		$\partial_{\mu} (\dot{J}^{\mu}[\partial_{\vec{\alpha}} \mathbf{W}, \partial_{\vec{\alpha}} \mathbf{W}])
		= F_{\vec{\alpha}}(\Omega(t),\omega(t);\mathbf{W}, \partial \mathbf{W}, \underpartial^{(2)} \mathbf{W}, \underpartial^{(3)} 		\mathbf{W},\cdots, \underpartial^{(|\vec{\alpha}| - 1)} \mathbf{W}, \partial_{\vec{\alpha}} \mathbf{W}).$ 
		The important point is that the smooth function $F_{\vec{\alpha}}$ 
		does \emph{not} depend on $\underpartial^{(|\vec{\alpha}| + 1)} \mathbf{W},$ 
		where $\underpartial$ denotes the spatial derivative gradient; see \eqref{E:DIVDOTJ}.
\end{enumerate}

These properties allow us to exploit the following basic integral identity, which is a version of the divergence theorem (see also \eqref{E:ENTIMEDERIVATIVE} and \eqref{E:FLUIDENERGYTIMEDERIVATIVE0SPEEDVELOCITY} - 
\eqref{E:FLUIDENERGYTIMEDERIVATIVE0SPEEDDENSITY}):

\begin{align} \label{E:DIV}
	\frac{d}{dt} \mathcal{E}_N^2(t) 
	& \eqdef \sum_{|\vec{\alpha}| \leq N}
		\int_{\mathbb{R}^3} \partial_{\mu} \big(\dot{J}^{\mu}[\partial_{\vec{\alpha}} \mathbf{W}, \partial_{\vec{\alpha}} 
		\mathbf{W}] \big) \, d^3 x \\
	& = \int_{\mathbb{R}^3}  
		F_N(\Omega(t),\omega(t);\mathbf{W}, \partial \mathbf{W}, \underpartial^{(2)} \mathbf{W}, \underpartial^{(3)} \mathbf{W},
		\cdots, \underpartial^{(N)} \mathbf{W}) \, d^3 x. \notag
\end{align}
Above, $F_N$ is a smooth function of its arguments. By $(1),$ the right-hand side of \eqref{E:DIV} is controllable in terms of $\mathcal{E}_N(t).$ Thus, $\mathcal{E}_N(t)$ is a norm-like quantity whose time evolution can be estimated. Roughly speaking, in order to apply the continuation principle, we have to show that the identity \eqref{E:DIV} prevents $\mathcal{E}_N(t)$ and $\fluidnorm{N}(t)$ from blowing up in finite time. In particular, we have to carefully investigate the structure of $F_N(\cdots);$ see the extended discussion in Section \ref{SSS:DIVERGENCECOMMENTS}. We also remark that in the case $\speed^2 = 0,$ we will utilize two distinct currents $\dot{J}_{velocity}^{\mu}$ and $\dot{J}_{density}^{\mu}$ in order to handle the partially decoupled nature of the relativistic Euler equations under the dust equation of state; this is of course connected to our previously mentioned use of separate energies $\mathcal{E}_{N;velocity}$ and $\mathcal{E}_{N-1;density}$ when $\speed^2 = 0.$

It is interesting to note that Christodoulou showed \cite{dC2000} that the Lagrangian for relativistic fluid mechanics is \emph{generic}, which means that up to null currents\footnote{A null current is a solution-dependent vectorfield whose divergence can be shown to be independent of the solution's derivatives without having to use the equations of motion (which are the relativistic Euler equations in the present article). Such null currents have a special algebraic structure and have been completely classified \cite{dC2000}.} (which by themselves lead to non-coercive integral identities) the method of vectorfield multipliers generates \emph{all possible integral identities} that a solution can verify. Furthermore, he showed that a certain widely used subclass of compatible currents, namely those generated by \emph{domain multiplier vectorfields}, have the following property: the only useful members of this class (i.e., generating an $L^2-$ coercive compatible ``energy'' current) are generated by domain multiplier vectorfields that belong to the \emph{inner core} of the characteristic subset of the tangent space. Now for the relativistic Euler equations, the inner core is the degenerate one-dimensional set $\mbox{span}\lbrace u \rbrace$ (see \cite{jS2008b}); this is connected to the fact that the fluid vorticity two-form verifies a transport equation in the direction of $u.$ Thus, it can be shown that up to multiplication by a scalar function, \emph{the only coercive compatible energy current generated by the domain multiplier vectorfields is the one defined below in Definition \ref{D:ENERGYCURRENT}.}

\subsubsection{The analytic effect of expansion and the restriction $0 \leq \speed^2 \leq 1/3$} \label{SSS:EXPANSIONRESTRICTIONROLE}

From the analytic standpoint, the stabilizing effect of spacetime expansion can be summarized as follows: 
the scale factor $e^{\Omega(t)}$ generates the dissipative $\omega(t) (3 \speed^2 - 2)u^j$ term on the right-hand side of \eqref{E:PARTIALTUJ}. The sources of these dissipative terms are the Christoffel symbols of $g$ 
relative to the coordinate system $(t,x^1,x^2,x^3),$ which are provided in Lemma \ref{L:BACKGROUNDCHRISTOFFEL}. In
the cases $0 < \speed^2 < 1/3,$ equation \eqref{E:PARTIALTUJ} can be rewritten as

\begin{align} \label{E:UJAPPROX}
	\partial_t \big[e^{\Omega} \DecayFunction(\Omega) u^j \big] 
		= \omega \Big\lbrace 3 \speed^2 - 1 + \frac{\DecayFunction'(\Omega)}{\DecayFunction(\Omega)}\Big\rbrace
			\big[e^{\Omega} \DecayFunction(\Omega) u^j \big]
		+ e^{\Omega} \DecayFunction(\Omega) \triangle'^j,
\end{align}
where $\DecayFunction$ is the function from the hypotheses \ref{A:A2} - \ref{A:A3} and $\omega(t) \eqdef \frac{d}{dt} \Omega(t).$ If we assume that the right-hand side of \eqref{E:UJAPPROX} is negligible (we will in fact establish this assumption when $0 < \speed^2 < 1/3$ and \eqref{E:BOOTSTRAPINTRO} holds), then by integrating the ``ODE'' \eqref{E:UJAPPROX}, we deduce that $|u^j| \lesssim \epsilon e^{-\Omega(t)}\DecayFunction^{-1}(\Omega(t)).$ Therefore, the size of the spatial part of $u$ as measured by $g$ is $|g_{ab} u^a u^b| = e^{2 \Omega} \delta_{ab} u^a u^b \lesssim \epsilon^2 \DecayFunction^{-2}(\Omega(t)).$ Using hypotheses \ref{A:A1} - \ref{A:A3}, it follows that when $0 < \speed^2 < 1/3,$ the quantity $g_{ab} u^a u^b,$ which appears e.g. as an inhomogeneous term in the Euler equation \eqref{E:FINALEULERP}, is expected to decay towards its background solution value of $0.$ Furthermore, $u^0$ (see \eqref{E:U0UPPERISOLATED}) is expected to decays towards its background value of $1.$ In the case $\speed^2 = 0,$ our assumption \eqref{E:BOOTSTRAPINTRO} will directly imply that $|g_{ab} u^a u^b| \lesssim \epsilon^2 e^{- 2 \Omega(t)},$ and we again expect decay towards $0.$ In the case $\speed^2 = 1/3$ (in which $\DecayFunction \equiv 1$), $|g_{ab} u^a u^b|$ is expected to remain at size $\epsilon^2.$ In contrast, if $\speed^2 > 1/3,$ this heuristic analysis suggests that there is nothing preventing the quantity $|g_{ab}u^a u^b|$ from growing unabatedly; we thus anticipate that there may be instability in these cases. We remark that in \cite{aR2004b}, Rendall also detected evidence of instability for $\speed^2 > 1/3$ through the use of formal power series expansions.

\subsubsection{Comments on the analysis of \eqref{E:DIV}} \label{SSS:DIVERGENCECOMMENTS}

Much of our work goes into estimating the $\int_{\mathbb{R}^3} F_N(\Omega(t), \omega(t);\mathbf{W}, \partial \mathbf{W}, \underpartial^{(2)} \mathbf{W}, \underpartial^{(3)} \mathbf{W}, \cdots, \underpartial^{(N)} \mathbf{W}) \, d^3 x$ term on the right-hand side of \eqref{E:DIV} in terms of the Sobolev norm $\fluidnorm{N}$ under the bootstrap assumption \eqref{E:BOOTSTRAPINTRO}. The end result of these estimates in the cases $0 < \speed^2 \leq 1/3$ is inequality \eqref{E:DIVJDOTL1}, and in the case $\speed^2 = 0$ is the inequalities \eqref{E:DIVJDOTL1DUSTVELOCITY} - \eqref{E:DIVJDOTL1DUSTDENSITY}. Our analysis is based on the identity \eqref{E:DIVDOTJ} for $\partial_{\mu} \dot{J}^{\mu}$ (in the cases $0 < \speed^2 \leq 1/3$), the identities \eqref{E:DIVDOTJ0SPEEDVELOCITY} - \eqref{E:DIVDOTJ0SPEEDDENSITY} for $\partial_{\mu} \dot{J}_{velocity}^{\mu}$ and $\partial_{\mu} \dot{J}_{density}^{\mu}$ (in the case $\speed^2 = 0$), on standard Sobolev-Moser type estimates (see the Appendix), and on Lemma \ref{L:EOVINHOMOGENEOUS}. Lemma \ref{L:EOVINHOMOGENEOUS} is a particularly important ingredient in this analysis, for this lemma dissects the algebraic structure of the inhomogeneous terms in the equations verified by the quantities $\partial_{\vec{\alpha}} \mathbf{W}.$ Understanding this structure is essential since these inhomogeneous terms appear in the expression \eqref{E:DIVDOTJ} for $\partial_{\mu} \big(\dot{J}^{\mu} [\partial_{\vec{\alpha}} \mathbf{W}, \partial_{\vec{\alpha}} \mathbf{W}] \big)$ (and hence in $F_N(\cdots)$), and also in the analogous expressions \eqref{E:DIVDOTJ0SPEEDVELOCITY} - \eqref{E:DIVDOTJ0SPEEDDENSITY} in the case $\speed^2 = 0.$ We remark that the quantities $\partial_{\vec{\alpha}} \mathbf{W}$ for $|\vec{\alpha}| \leq N$ verify linear (in $\partial_{\vec{\alpha}} \mathbf{W}$) PDEs whose principal coefficients depend on the solution $\mathbf{W}$ and whose inhomogeneous terms depend on the derivatives $\partial_{\vec{\beta}} \mathbf{W}$ for $|\vec{\beta}| \leq |\vec{\alpha}|.$ These PDEs are known as the \emph{equations of variation}; we describe them in detail in Section \ref{SS:ENERGY}.

We now discuss some of the subtleties in our derivation of the estimates \eqref{E:DIVJDOTL1DUSTVELOCITY} - \eqref{E:DIVJDOTL1}
which are the main estimates used in our derivation of the key energy inequalities \eqref{E:ENINTEGRALDUSTVELOCITY} - \eqref{E:UNMINUSONEINTEGRAL} of Proposition \ref{P:INTEGRALENERGYINEQUALITIES}. For simplicity, we will focus the present discussion on the cases $0 < \speed^2 < 1/3$ and the estimate \eqref{E:DIVJDOTL1}, which provides an upper bound for $\sum_{|\vec{\alpha}| \leq N} \int_{\mathbb{R}^3} \partial_{\mu} (\dot{J}^{\mu}[\partial_{\vec{\alpha}} \mathbf{W}, \partial_{\vec{\alpha}} \mathbf{W}]) \, d^3 x$ in terms of the norm $\fluidnorm{N}.$ A very important aspect of our argument in these cases is the following: in order to derive \eqref{E:DIVJDOTL1} and to close our global existence argument, \emph{we need to prove that the lower-order spatial derivatives of $u^j$ decay faster} than its top-order derivatives. To see why this is the case, we will now show that without this improved decay, there would be an obstruction to global existence. We illustrate this obstruction by closely investigating a specific term from the right-hand side of \eqref{E:DIV}. In particular, the $\frac{2 \speed^2 \mathfrak{F}}{(1 + \speed^2)} \Rlog$ term from the identity \eqref{E:DIVDOTJ} below, where $\mathfrak{F} = - \omega \frac{(1 + \speed^2)}{u^0} g_{ab}u^a u^b$ is the term on the right-hand side of \eqref{E:FINALEULERP}, leads to the identity

\begin{align} \label{E:HEURISTICENERGYIDENTITY}
	\frac{d}{dt} \mathcal{E}_N^2(t) = - 2 \speed^2 \int_{\mathbb{R}^3} \omega \frac{1}{u^0} 
	\overbrace{g_{ab}}^{e^{2\Omega}\delta_{ab}} u^a u^b \Rlog  \, d^3 x + \cdots,
\end{align}
where $\Rlog \eqdef \ln \left(e^{3(1 + \speed^2)\Omega} \rho/\bar{\rho} \right),$ and
$\mathcal{E}_N^2(t) \approx \| \Rlog \|_{H^N}^2 + e^{2\Omega} \sum_{j=1}^3 \| u^j \|_{H^N}^2.$ Using \eqref{E:BOOTSTRAPINTRO}, \eqref{E:HEURISTICENERGYIDENTITY}, the Sobolev embedding result $\| u^j \|_{L^{\infty}} \lesssim \| u^j \|_{H^2} \lesssim e^{- \Omega} \mathcal{E}_N,$ and the estimate $u^0 \geq 1,$ we deduce that

\begin{align} \label{E:USELESSENERGYIDENTITY}
	\frac{d}{dt} \mathcal{E}_N^2(t) \lesssim \omega \mathcal{E}_N^3(t) + \cdots.
\end{align}
Inequality \eqref{E:USELESSENERGYIDENTITY} allows for the dangerous possibility that $\mathcal{E}_N^2(t)$ grows unabatedly, and is therefore not sufficient to prove our main stability theorem.

To remedy this difficulty, we introduce the lower-order Sobolev norm $\mathcal{U}_{N-1} \eqdef e^{\Omega} \DecayFunction(\Omega) \Big(\sum_{j=1}^3 \| u^j \|_{H^{N-1}}^2 \Big)^{1/2}$ (see Definition \ref{D:NORMS}). Here, $\DecayFunction$ is the function
from the hypotheses \ref{A:A2} - \ref{A:A3}. The main point is the following: if we knew that 
$\mathcal{U}_{N-1} \leq \epsilon,$ then the Sobolev embedding estimate $\| u^j \|_{L^{\infty}} \lesssim 
\| u^j \|_{H^2} \lesssim e^{- \Omega} \DecayFunction^{-1}(\Omega) \mathcal{U}_{N-1} \lesssim \epsilon e^{-\Omega}\DecayFunction^{-1}(\Omega)$ would
allow us to upgrade \eqref{E:USELESSENERGYIDENTITY} to

\begin{align} \label{E:USEFULENERGYIDENTITY}
	\frac{d}{dt} \mathcal{E}_N^2(t) \lesssim \epsilon \omega(t) \DecayFunction^{-1}(\Omega(t)) \mathcal{E}_N^2(t) + \cdots.
\end{align}
Because of our assumptions \ref{A:A1} - \ref{A:A3} on the scale factor, we have 
that $\omega(t) \DecayFunction^{-1}(\Omega(t))$ $= \frac{d}{dt} \Big\lbrace \int_{1}^{\Omega(t)} \DecayFunction^{-1}(\widetilde{\Omega}) \, d \widetilde{\Omega} \Big\rbrace \in L_t^1([1,\infty)).$ Therefore, the first term on the right-hand side of \eqref{E:USEFULENERGYIDENTITY} is amenable to Gronwall's inequality. Assuming that the remaining terms $\cdots$ could be treated similarly, we would therefore be able to derive an a priori estimate for $\mathcal{E}_N^2(t)$ guaranteeing that it remains uniformly small for all time; thanks to the continuation principle, this is the main step in proving global existence. At its core, the proof of our main stability theorem is essentially a more elaborate version of the estimate \eqref{E:USEFULENERGYIDENTITY}.

It remains to discuss how we obtain the desired control over the norm $\mathcal{U}_{N-1}.$ Roughly speaking, we will
use the improved decay for the $u^j$ suggested by equation \eqref{E:UJAPPROX}. It turns out that the error terms corresponding to this equation are small enough that we can effectively control the time derivative of $\mathcal{U}_{N-1}(t)$ even if we only have control over the weaker norm $\sum_{j=1}^3 \sum_{|\vec{\alpha}|=N} e^{\Omega} \| \partial_{\vec{\alpha}} u^j \|_{L^2}$ of the top-order derivatives (which will follow from the bootstrap assumption \eqref{E:BOOTSTRAPINTRO}). The precise estimates corresponding to this fact are provided in inequalities \eqref{E:UNORMTIMEDERIVATIVE} and \eqref{E:TRIANGLEPRIMEJHNMINUSONE} below. We emphasize that in order to derive the improved decay (improved relative to the rate predicted by the Sobolev norm bootstrap assumption \eqref{E:BOOTSTRAPINTRO}), what matters is the interaction between the hypotheses \ref{A:A1} - \ref{A:A3} on the scale factor and the structure of the error terms. Note that \eqref{E:UJAPPROX} is \emph{not literally an ODE in} $u^j$ since spatial derivatives of the solution are present on the right-hand side as ``error terms.'' Therefore, the approach we have described only allows us to derive improved decay estimates for the below-top-order derivatives of the solution. We remark that we do not use the norm $\mathcal{U}_{N-1}$ in our analysis of the cases $\speed^2 = 0, 1/3.$

\subsubsection{Sharp results in case $\speed^2 = 1/3$} \label{SSS:PURERADIATION}

Our sharp analysis in the case $\speed^2 = 1/3$ is based on the well-known \emph{conformal invariance} of the relativistic Euler equations in this case; see Section \ref{S:PURERADIATION} for more details. Roughly speaking, when $\speed^2 = 1/3,$ one can perform a change-of-time-variable $\frac{d \tau}{d t} = e^{- \Omega(t)},$ $\tau(1) = 1,$ and also a rescaling of the fluid variables in order to translate the problem of interest into an equivalent problem in Minkowski spacetime. Now if assumption \ref{A:A1} holds but $\int_{s = 1}^{\infty} e^{- \Omega(s)} \, d s = \infty,$ then the analysis of Section \ref{S:PURERADIATION} shows that the instability of the background fluid solutions follows easily from the instability results derived in \cite{dC2007}; see Corollary \ref{C:SHOCKSCANFORM}. Furthermore, the arguments given in Section \ref{S:PURERADIATION} could easily have been extended to show the future stability of the background fluid solutions when $\speed^2 = 1/3,$ hypothesis \ref{A:A1} holds, and  $\int_{s = 1}^{\infty} e^{- \Omega(s)} \, d s < \infty.$ However, we instead chose to use the energy-method/continuation principle framework to prove future stability (see Theorem \ref{T:GLOBALEXISTENCE}) since this framework is stable and since there is no known alternative to this framework in the cases $\speed^2 \in [0, 1/3).$

\subsection{Outline of the article}

The remainder of the article can be summarized as follows.

\begin{itemize}
	\item In Section \ref{S:Notation} we introduce some standard notation and conventions.
	\item In Section \ref{S:EULERDECOMP}, we derive an equivalent version of the relativistic Euler equations 
		that is useful for our ensuing analysis. We also introduce some standard PDE matrix-vector notation.
	\item In Section \ref{S:NORMSANDENERGIES} we introduce the Sobolev norms and the related energies that we use
		to analyze solutions.
	\item In Section \ref{S:SOBOLEV}, we derive the Sobolev estimates that play a major role in our derivation
		of differential inequalities for the fluid norms and energies.
	\item In Section \ref{S:NORMVSENERGY} we prove a simple comparison proposition, which shows that the energies we have
		defined can be used to control the norms.
	\item In Section \ref{S:INTEGRALINEQUALITY}, we use the estimates of Section \ref{S:SOBOLEV} to derive differential
		inequalities for the norms and energies. 
	\item In Section \ref{S:GLOBALEXISTENCE}, we use the differential inequalities to prove our main future stability theorem.
	\item In Section \ref{S:PURERADIATION}, we use a well-known result of Christodoulou to show that in the case
		$\speed^2 = 1/3,$ the non-integrability condition $\int_{s = 1}^{\infty} e^{- \Omega(s)} \, d s = 
			\infty$ leads to the nonlinear instability of the explicit fluid solutions.
\end{itemize}

\section{Notation} \label{S:Notation}

\subsection{Index conventions}
Greek ``spacetime'' indices $\alpha, \beta, \cdots$ take on the values $0,1,2,3,$ while Latin ``spatial'' indices $a,b,\cdots$ 
take on the values $1,2,3.$ Pairs of repeated indices are summed over their respective ranges. We lower spacetime indices with
$g_{\mu \nu}$ and raise them with $(g^{-1})^{\mu \nu}.$ 

\subsection{Coordinate systems and differential operators} \label{SS:COORDINATESYSTEMS}
We perform most of our computations relative to the standard rectangular coordinate system $(x^0,x^1,x^2,x^3)$ on $\mathbb{R}^4.$ We often write $t$ instead of $x^0.$ In this coordinate system, the metric $g$ is of the form \eqref{E:METRICFORM}. We often use the symbol $\partial_{\mu}$ to denote the coordinate derivative $\frac{\partial}{\partial x^{\mu}},$ and we often write $\partial_t$ instead of $\partial_0.$ Occasionally, (see e.g. the derivation of \eqref{E:RESCALEDESTIMATE}) we also use the rescaled time variable $\tau$ introduced in Section \ref{S:PURERADIATION}. The corresponding rescaled partial time derivative is denoted by $\partial_{\tau}.$ All spatial derivatives are always taken with respect to the coordinates $(x^1,x^2,x^3).$

If $\vec{\alpha} = (n_1,n_2,n_3)$ is a triplet of non-negative integers, then we define the spatial multi-index coordinate differential operator $\partial_{\vec{\alpha}}$ by $\partial_{\vec{\alpha}} \eqdef \partial_1^{n_1} \partial_2^{n_2} \partial_3^{n_3}.$ We use the notation $|\vec{\alpha}| \eqdef n_1 + n_2 + n_3$ to denote the order of $\vec{\alpha}.$

\begin{align}
	D_{\mu} T_{\mu_1 \cdots \mu_s}^{\nu_1 \cdots \nu_r} = 
		\partial_{\mu} T_{\mu_1 \cdots \mu_s}^{\nu_1 \cdots \nu_r} + 
		\sum_{a=1}^r \Gamma_{\mu \ \alpha}^{\ \nu_{a}} T_{\mu_1 \cdots \mu_s}^{\nu_1 \cdots \nu_{a-1} \alpha \nu_{a+1} \nu_r} - 
		\sum_{a=1}^s \Gamma_{\mu \ \mu_{a}}^{\ \alpha} T_{\mu_1 \cdots \mu_{a-1} \alpha \mu_{a+1} \mu_s}^{\nu_1 \cdots \nu_r}
\end{align} 
denotes the components of the covariant derivative of a spacetime tensorfield $T.$ The Christoffel symbols
$\Gamma_{\mu \ \nu}^{\ \alpha}$ are defined in \eqref{E:CHRISTOFFELDEF}.

$\partial^{(N)} T_{\mu_1 \cdots \mu_s}^{\nu_1 \cdots \nu_r}$ denotes the array containing of all of the $N^{th}$ order \emph{spacetime} coordinate derivatives (including time derivatives) of the component $T_{\mu_1 \cdots \mu_s}^{\nu_1 \cdots \nu_r}.$ $\underpartial^{(N)} T_{\mu_1 \cdots \mu_s}^{\nu_1 \cdots \nu_r}$ denotes the array of all $N^{th}$ order \emph{spatial coordinate} derivatives of the component $T_{\mu_1 \cdots \mu_s}^{\nu_1 \cdots \nu_r}.$ We omit the superscript when $N=1.$

\subsection{Norms} \label{SS:NORMS}
We define the standard Sobolev norm of a function $\big\| f \big\|_{H^N}$ as follows:

\begin{align} \label{E:SobolevNormDef}
	\big\| f \big\|_{H^N} \eqdef
		\bigg( \sum_{|\vec{\alpha}| \leq N} 
		\int_{\mathbb{R}^3} \big|\partial_{\vec{\alpha}} f(t,x^1,x^2,x^3) \big|^2 \,
		d^3x \bigg)^{1/2}.
\end{align}
The above volume form $d^3x$ corresponds to the standard flat metric on $\mathbb{R}^3.$

We denote the $N^{th}$ order homogeneous Sobolev norm of $f$ by

\begin{align}
	\big\| \underpartial^{(N)} f \big\|_{L^2} \eqdef
		\sum_{|\vec{\alpha}| = N} \big\| \partial_{\vec{\alpha}} f \big\|_{L^2}.
\end{align}

If $\mathfrak{K} \subset \mathbb{R}^n$ or $\mathfrak{K} \subset \mathbb{T}^n,$ then $C^N_b(\mathfrak{K})$ denotes the set of $N-$times continuously differentiable functions (either scalar or array-valued, depending on context) on the interior of
$\mathfrak{K}$ with bounded derivatives up to order $N$ that extend continuously to the closure of $\mathfrak{K}.$ 
We define the norm corresponding to this function space by
    
    \begin{equation} \label{E:CbMNormDef}
        |F|_{N,\mathfrak{K}} \eqdef \sum_{|\vec{I}|\leq N} \mbox{ess} \sup_{\cdot \in \mathfrak{K}}
        |\partial_{\vec{I}}F(\cdot)|,
    \end{equation}
    where $\partial_{\vec{I}}$ is a multi-indexed differential operator representing repeated partial differentiation
    with respect to the arguments $\cdot$ of $F,$ which may be either spacetime coordinates or metric/fluid components
    depending on context. When $N=0,$ we use the slightly more streamlined notation
    
    \begin{align} 
    	|F|_{\mathfrak{K}} \eqdef \mbox{ess} \sup_{\cdot \in \mathfrak{K}} |F(\cdot)|.
    \end{align}
    Furthermore, we define
    \begin{align} 
    	|F^{(N)}|_{\mathfrak{K}} \eqdef \sum_{|\vec{I}| = N}|\partial_{\vec{I}} F|_{\mathfrak{K}}.
    \end{align}
    When $\mathfrak{K} = \mathbb{R}^3,$ we sometimes use the more familiar notation 
    
    \begin{align}
    	\| F \|_{L^{\infty}} & \eqdef \mbox{ess} \sup_{x \in \mathbb{R}^3} |F(x)|, \\
    	\| F \|_{C_b^N} & \eqdef \sum_{|\vec{\alpha}| \leq N} \big\| \partial_{\vec{\alpha}} F \|_{L^{\infty}}.
    \end{align}
    
	 	If $A$ is an $m \times n$ (often $4 \times 4$ in this article) array-valued function with entries $A_{jk},$ 
	 	$(1 \leq j \leq m, 1\leq k \leq n),$ then in Section \ref{S:SOBOLEV}, we write e.g.
    $\| A \|_{H^N}$ to denote the $m \times n$ \emph{array} whose entries are $\|A_{jk} \|_{H^N}.$ We use similar notation for 
    other norms of $A.$ 
		
		If $I \subset \mathbb{R}$ is an interval and $X$ is a normed function space, then 
    $C^N(I,X)$ denotes the set of $N$-times continuously differentiable maps from $I$ into $X.$

\subsection{Numerical constants} \label{SS:runningconstants}
$C$ denotes a numerical constant that is free to vary from line to line. We sometimes write e.g. $C(N)$ when we want to explicitly indicate the dependence of $C$ on quantities. We use symbols such as $c,$ $C_*,$ etc., to denote constants that play a distinguished role in the discussion. We write $X \lesssim Y$ when there exists a constant $C > 0$ such that $X \leq CY.$
We write $X \approx Y$ when $X \lesssim Y$ and $Y \lesssim X.$

\section{Alternative Formulation of the Relativistic Euler Equations} \label{S:EULERDECOMP}

In this section, we derive an equivalent version the relativistic Euler equations \eqref{E:EULERINTROP} - \eqref{E:EULERINTROU} under the equation of state $p = \speed^2 \rho;$ we will work with this version for most of the remainder of the article. We then briefly discuss a classical local existence result and a continuation principle. The latter provides sufficient criteria for the solution to avoid forming a singularity.

\subsection{Alternative formulation of the relativistic Euler equations} \label{SS:ALTFORM}

Our alternative formulation of the relativistic Euler equations is captured it the next proposition.

\begin{proposition}[\textbf{Alternative formulation of the Euler equations}] \label{P:DECOMPOSITION}
	Assume that the fluid verifies the equation of state $p = \speed^2 \rho.$ Let 
	
	\begin{align} \label{E:PRESCALED}
		\Rlog \eqdef \ln \left(e^{3(1 + \speed^2)\Omega} \rho/\bar{\rho} \right)
	\end{align}
	be a normalized energy density variable, where $\bar{\rho} > 0$ is the constant corresponding to the background density
	variable $\widetilde{\rho} = \bar{\rho} e^{-3(1 + \speed^2)\Omega}.$ Then in the coordinate system 
	$(t,x^1,x^2,x^3),$ the relativistic Euler equations \eqref{E:EULERINTROP} - \eqref{E:EULERINTROU} are equivalent to the 
	following system of equations in the unknowns $(\Rlog,u^1,u^2,u^3)$ (for $j = 1,2,3$):
	
\begin{subequations}
\begin{align}
	u^{\alpha} \partial_{\alpha} \Rlog + (1 + \speed^2) \Big(\frac{1}{u^0}\Big) u_a \partial_t u^a + (1 + 
		\speed^2) \partial_{a}u^{a} & = - \omega \frac{(1 + \speed^2)}{u^0} g_{ab}u^a u^b, \label{E:FINALEULERP} \\
	u^{\alpha} \partial_{\alpha} u^{j} + \frac{\speed^2}{(1 + \speed^2)}\Pi^{j \alpha} \partial_{\alpha} \Rlog 
		& = \omega (3 \speed^2 - 2) u^0 u^{j}. \label{E:FINALEULERUJ}
\end{align}
\end{subequations}
Above, 

\begin{align*}
u^0 \eqdef (1 + g_{ab}u^a u^b)^{1/2}, \qquad \Pi^{\mu \nu} \eqdef u^{\mu} u^{\nu} + (g^{-1})^{\mu \nu},
\qquad \omega(t) \eqdef \frac{d}{dt} \Omega(t).
\end{align*}

Furthermore, $u^0$ is a solution to the following equation:
	
	\begin{align}
		u^{\alpha} \partial_{\alpha} u^{0}  + \frac{\speed^2}{(1 + \speed^2)}\Pi^{0 \alpha} \partial_{\alpha} \Rlog 
			& =	\omega (3 \speed^2 - 1) g_{ab} u^a u^b. \label{E:U0EQUATION}
	\end{align}

\end{proposition}

\begin{proof}
		To obtain \eqref{E:FINALEULERP}, we first expand the covariant differentiation in \eqref{E:EULERINTROP}
		to deduce the following equation:
		
		\begin{align} \label{E:FIRSTEULERCHRISTOFFEL}
			u^{\alpha} \partial_{\alpha} \ln \rho + (1+ \speed^2) \partial_{\alpha} u^{\alpha} =
				- (1+ \speed^2) \Gamma_{\alpha \ \beta}^{\ \alpha} u^{\beta}.
		\end{align}
		Lemma \ref{L:BACKGROUNDCHRISTOFFEL} implies that 
		$\Gamma_{\alpha \ \beta}^{\ \alpha} u^{\beta} = 3 \omega u^0,$
		while equation \eqref{E:U0UPPERISOLATED} implies that $\partial_t u^0 = \frac{1}{u^0} \big\lbrace u_a \partial_t u^a 
		+ \omega g_{ab}u^a u^b \big\rbrace.$ Equation \eqref{E:FINALEULERP} now follows from these
		identities and the identity $\partial_t \ln \rho = \partial_t \ln \Rlog - 3(1 + \speed^2) \omega.$
		
		Similarly, to obtain \eqref{E:FINALEULERUJ}, we first expand the covariant differentiation in \eqref{E:EULERINTROU} 
		to deduce the following equation: 		
		
		\begin{align}  \label{E:FINALEULERCHRISTOFFEL}
			u^{\alpha} \partial_{\alpha} u^j + \frac{\speed^2}{(1 + \speed^2)} \Pi^{j \alpha} \partial_{\alpha} \ln \rho
				& = - \Gamma_{\alpha \ \beta}^{\ j} u^{\alpha} u^{\beta}.
		\end{align}
		Equation \eqref{E:FINALEULERUJ} now follows from \eqref{E:FINALEULERCHRISTOFFEL}, Lemma \ref{L:BACKGROUNDCHRISTOFFEL},
		and the identity $\frac{\speed^2}{(1 + \speed^2) } \Pi^{j \alpha} \partial_{\alpha} \ln \rho 
		= \frac{\speed^2}{(1 + \speed^2)} \Pi^{j \alpha} \partial_{\alpha} \Rlog 
		- 3 \speed^2 \omega u^0 u^j.$ The proof of \eqref{E:U0EQUATION} is similar, and we omit the details.
		
\end{proof}

\begin{remark} \label{R:BACKGROUNDSOLUTION}
	Note that in terms of the variables $(L,u^1,u^2,u^3),$ the background solution \eqref{E:BACKGROUNDU} takes the form
	
	\begin{align}
		(\widetilde{L},\widetilde{u}^1,\widetilde{u}^2,\widetilde{u}^3) = (0,0,0,0). 
	\end{align}
\end{remark}

Part of our analysis involves solving for $\partial_t \Rlog$ and $\partial_t u^{\mu}$ and treating spatial derivatives as inhomogeneous error terms. This approach can also be used to analyze $\partial_{\vec{\alpha}} \Rlog$ and $\partial_{\vec{\alpha}} u^{\mu},$ for $|\vec{\alpha}| \leq N - 1.$ As a preliminary step in this analysis, we solve for $\partial_t \Rlog$ and $\partial_t u^{\mu}.$

\begin{corollary} [\textbf{Isolated Time Derivatives}] \label{C:ISOLATEPARTIALTPUJ}
Let $(\Rlog,u^1,u^2,u^3)$ be a solution to the relativistic Euler equations \eqref{E:FINALEULERP} - \eqref{E:FINALEULERUJ}, where $u^0$ is defined by \eqref{E:U0UPPERISOLATED}. Then the time derivatives of the fluid variables can be expressed as follows:

\begin{subequations}
\begin{align}
	\partial_t \Rlog  & = \triangle', \label{E:PARTIALTRLOG} \\  
	\partial_t u^0 & = \triangle'^0, \label{E:PARTIALTU0} \\
	\partial_t u^j & = \omega (3 \speed^2 - 2)u^j + \triangle'^j, \label{E:PARTIALTUJ} 
\end{align}
\end{subequations}
where the error terms $\triangle',$ $ \triangle'^0,$ and $ \triangle'^j$ are defined by

\begin{subequations}
\begin{align}
	\triangle' & = \omega (1 + \speed^2)(1 - 3 \speed^2) \bigg\lbrace \frac{\frac{g_{ab}u^a u^b}{(u^0)^2}}
			{1 - \speed^2 \frac{g_{ab}u^a u^b}{(u^0)^2}} \bigg\rbrace \label{E:TRIANGLEPRIMEDEF} \\
	& \ \ + \bigg\lbrace 1 - \speed^2 \frac{g_{ab}u^a u^b}{(u^0)^2} \bigg\rbrace^{-1} 
		\Big\lbrace (\speed^2-1) \frac{u^a}{u^0} \partial_a \Rlog - (1 + \speed^2) \frac{\partial_a u^a}{u^0}  
		+ \frac{(1 + \speed^2)}{(u^0)^3} g_{ab} u^a u^k \partial_k u^b \Big\rbrace, \notag \\
	\triangle'^0 & =  \omega (3 \speed^2 - 1) 
		\bigg\lbrace\frac{\frac{g_{ab}u^a u^b}{u^0}}{1 - \speed^2 \frac{ g_{ab}u^a u^b}{(u^0)^2}}\bigg\rbrace 
		- \frac{\speed^2}{(1 + \speed^2)}\bigg\lbrace \frac{1 - \frac{g_{ab}u^au^b}{(u^0)^2}}{1 - \speed^2 
		\frac{g_{ab}u^au^b}{(u^0)^2}} \bigg\rbrace u^a \partial_a \Rlog \label{E:TRIANGLEPRIME0DEF} \\
		& \ \ + \bigg\lbrace \frac{\speed^2 \frac{g_{ab}u^a u^b}{(u^0)^2}}{1 - \speed^2\frac{g_{ab}u^au^b}{(u^0)^2}} \bigg\rbrace
			\partial_a u^a  
				- \frac{\frac{g_{ab} u^a u^k \partial_k u^b}{(u^0)^2}}{1 - \speed^2\frac{g_{ab}u^au^b}{(u^0)^2}} \ , \notag \\
	\triangle'^j & = \omega \speed^2 (3\speed^2 - 1) u^j 
		\bigg\lbrace\frac{\frac{g_{ab}u^au^b}{(u^0)^2}}{1 - \speed^2 \frac{g_{ab}u^a u^b}{(u^0)^2}} \bigg\rbrace 
		+ \speed^2 u^j \bigg\lbrace \frac{\frac{\partial_a u^a}{u^0} - \frac{g_{ab} u^a u^k \partial_k u^b}{(u^0)^3}}{1 - \speed^2 
		\frac{g_{ab}u^a u^b}{(u^0)^2}} \bigg\rbrace \label{E:TRIANGLEPRIMEJDEF} \\
	& \ \ - \frac{\speed^4}{(1 + \speed^2)} u^j 
			\bigg\lbrace \frac{1 - \frac{g_{ab}u^a u^b}{(u^0)^2}}{1 - \speed^2 \frac{g_{ab}u^a u^b}{(u^0)^2}} \bigg\rbrace
			\frac{u^a}{u^0} \partial_a \Rlog
		- \frac{u^a}{u^0} \partial_a u^j 
		- \frac{\speed^2}{(1 + \speed^2)} \frac{(g^{-1})^{aj} \partial_a \Rlog}{u^0}. \notag
\end{align}
\end{subequations}

\end{corollary}

\begin{proof}
	The proof consists of tedious but simple calculations; we omit the details.
\end{proof}

\begin{remark}
	Note that certain terms on the right-hand side of \eqref{E:TRIANGLEPRIMEDEF} - \eqref{E:TRIANGLEPRIMEJDEF} vanish when $\speed^2 = 0$ or $\speed^2 = 1/3.$ This vanishing is important in our analysis of these cases.
\end{remark}

In the next lemma, we provide the Christoffel symbols of the metric $g$ relative to our coordinate system. The lemma was used in the proof of Proposition \ref{P:DECOMPOSITION}.

\begin{lemma} \label{L:BACKGROUNDCHRISTOFFEL}
	The non-zero Christoffel symbols of the metric \\ 
	$g = -dt^2 + e^{2 \Omega(t)} \sum_{j=1}^3 (dx^j)^2$ are

	\begin{align} \label{E:BACKGROUNDCHRISTOFFEL}
		\Gamma_{j \ k}^{\ 0} = \Gamma_{k \ j}^{\ 0} = \omega g_{jk}, 
			&& \Gamma_{k \ 0}^{\ j} = \Gamma_{0 \ k}^{\ j}  = \omega \delta_k^j, && (j,k=1,2,3),
	\end{align}
	where
	
	\begin{align}
		\omega \eqdef \frac{d}{dt} \Omega.
	\end{align}
\end{lemma}

\begin{proof}
	The lemma follows via simple computations from the definition
	
	\begin{align} \label{E:CHRISTOFFELDEF}
		\Gamma_{\mu \ \nu}^{\ \alpha} & \eqdef \frac{1}{2} (g^{-1})^{\alpha \lambda}(\partial_{\mu} g_{\lambda \nu} 
		+ \partial_{\nu} g_{\mu \lambda} - \partial_{\lambda} g_{\mu \nu}). 
	\end{align}
\end{proof}

\subsection{Local existence and the continuation principle} \label{SS:LOCALEXISTENCE}

In this section, we state a standard local existence result for the system
\eqref{E:FINALEULERP} - \eqref{E:FINALEULERUJ}.

\begin{theorem}[\textbf{Local Existence}] \label{T:LOCAL}
	Let $N \geq 3$ be an integer. Let $\mathring{\Rlog} = \Rlog|_{t=1} = \ln \big(\rho/\bar{\rho}\big)|_{t=1},$ 
	$\mathring{u}^{j} = u^{j}|_{t=1},$ $(j=1,2,3),$ be initial data for the relativistic Euler equations 
	\eqref{E:FINALEULERP} - \eqref{E:FINALEULERUJ} satisfying 
	
\begin{subequations}	
\begin{align}
		\mathring{\Rlog}  & \in \left \lbrace \begin{array}{l}
	    H^{N-1}, \qquad \speed^2 = 0, \\
	    H^N, \qquad 0 < \speed^2 \leq 1/3, 
	    \end{array}
	    \right. \\
	    \mathring{u}^j & \in H^N,
\end{align}
\end{subequations}
	where $\bar{\rho} > 0$ is a constant. Assume that $\sup_{x \in \mathbb{R}^3} |\mathring{\Rlog}| < \infty.$ Then there exists 
	a real number $T^+ > 1$ such that these data launch a unique classical 
	solution $(\Rlog, u^1, u^2, u^3)$ to the relativistic Euler equations  \eqref{E:FINALEULERP} - \eqref{E:FINALEULERUJ} 
	existing on the spacetime slab $[1, T^+) \times \mathbb{R}^3.$ Relative to the coordinate system
	$(t,x^1,x^2,x^3),$ the solution has the following regularity properties:

\begin{subequations}	
\begin{align}
		\mathring{\Rlog} & \in \left \lbrace \begin{array}{l}
	    C^0([1, T_+) \times \mathbb{R}^3), \qquad \speed^2 = 0, \\
	    C^1([1, T_+) \times \mathbb{R}^3), \qquad 0 < \speed^2 \leq 1/3, 
	   \end{array}
	    \right. \\
	    \mathring{u}^j & \in C^1([1, T_+) \times \mathbb{R}^3),
\end{align}
\end{subequations}

\begin{subequations}
	\begin{align}
		\mathring{\Rlog} & \in \left \lbrace \begin{array}{l}
	    C^0([1, T_+),H^{N-1}), \qquad \speed^2 = 0, \\   
	    C^0([1, T_+),H^{N}), \qquad 0 < \speed^2 \leq 1/3, 
	    \end{array}
	    \right. \\
	    u^0 - 1, \ u^j & \in C^0([1, T_+),H^{N}),
\end{align}	
\end{subequations}
		
	In addition, there exists an open neighborhood $\mathcal{O}$ of
	$(\mathring{\Rlog}, \mathring{u}^{j})$ such that all data belonging to $\mathcal{O}$ launch solutions that also exist on the 
	slab $[1, T_+) \times \mathbb{R}^3$ and that have the same regularity properties as $(\Rlog, u^{\mu}).$ Furthermore, on 
	$\mathcal{O},$ the map $\mbox{data} \rightarrow \mbox{solution}$ is continuous.
	
\end{theorem}

\begin{proof}
	Theorem \ref{T:LOCAL} can be proved using a standard iteration or contraction mapping argument based on
	energy estimates for linearized equations that are in the spirit of the estimates derived in Sections \ref{S:NORMVSENERGY}
	and \ref{S:INTEGRALINEQUALITY}. See e.g. \cite[Ch. VI]{lH1997}, \cite{jS2008a} for details on how to use such
	energy estimates to deduce local existence.
\end{proof}

In our proof of Theorem \ref{T:GLOBALEXISTENCE} we invoke the following continuation principle, which provides
standard criteria that are sufficient to ensure that a solution to the relativistic Euler equations exists globally in time.

\begin{proposition}[\textbf{Continuation Principle}] \label{P:CONTINUATION}
	Assume the hypotheses of Theorem \ref{T:LOCAL}. Let $T_{max}$ be the supremum over all times $T_+$ such that the 
	solution $(\Rlog,u^{\mu})$ exists on the interval $[1,T_+)$ and has the properties stated in the conclusions of 
	Theorem \ref{T:LOCAL}. Then in the case $\speed^2 = 0,$ if $T_{max} < \infty,$ we have 
	
	\begin{align}
		\lim_{t \to T_{max}^-} \sup_{0 \leq s \leq t} \Bigg\lbrace \| \Rlog(s,\cdot) \|_{L^{\infty}} 
		+ \sum_{j=1}^3 \Big( \| u^j(s,\cdot) \|_{C_b^1} \Big) \Bigg\rbrace 
		= \infty.
	\end{align}
	
	Furthermore, in the cases $0 < \speed \leq 1/3,$ if $T_{max} < \infty,$ we have 
	
	\begin{align}
		\lim_{t \to T_{max}^-} \sup_{0 \leq s \leq t} \Bigg\lbrace 
		\| \Rlog(s,\cdot) \|_{C_b^1} 
		+ \sum_{j=1}^3 \Big( \| u^j(s,\cdot) \|_{C_b^1} \Big) \Bigg\rbrace 
		= \infty.
	\end{align}
	
\end{proposition}

\begin{proof}
	See e.g. \cite{lH1997}, \cite{jS2008b} for the ideas behind a proof. The case $\speed^2 = 0$ is special
	because in this case, the evolution of the $u^j$ decouples from that of $\Rlog,$ and furthermore,
	the evolution equation \eqref{E:FINALEULERP} is \emph{linear} in $\Rlog.$
\end{proof}

\subsection{Matrix-vector notation}

It is convenient to write the system \eqref{E:FINALEULERP} - \eqref{E:FINALEULERUJ} in abbreviated form using standard PDE matrix-vector notation:

\begin{align} \label{E:MATRIXVECTOREULER}
	A^{\beta} \partial_{\beta} \mathbf{W} = \mathbf{b}.
\end{align}
Here, $\mathbf{W}$ is a fluid variable column array, and $\mathbf{b}$ is a column array of inhomogeneous terms:

\begin{align} \label{E:arraydefs}
	\mathbf{W} \eqdef 
		\left( \begin{array}{c}
			\Rlog \\
			u^1 \\
			u^2 \\
			u^3
		\end{array} \right), \qquad 
		\mathbf{b} \eqdef 	\left( \begin{array}{c}
			- \omega \frac{(1 + \speed^2)}{u^0} g_{ab}u^a u^b \\
			\omega (3 \speed^2 - 2) u^0 u^1 \\
			\omega (3 \speed^2 - 2) u^0 u^2 \\
			\omega (3 \speed^2 - 2) u^0 u^3
		\end{array} \right).
\end{align}	
The $A^{\mu}$ are $4 \times 4$ matrices defined by

\begin{subequations}
\begin{align}
	A^0 & =   \begin{pmatrix}
                        u^0 & (1 + \speed^2) \frac{u_1}{u^0}  & (1 + \speed^2) \frac{u_2}{u^0}  
                        	& (1 + \speed^2) \frac{u_3}{u^0}  \\
                       	\Big( \frac{\speed^2}{(1 + \speed^2)} \Big) \Pi^{10} & u^0 & 0 & 0 \\
                        \Big( \frac{\speed^2}{(1 + \speed^2)} \Big) \Pi^{20} & 0 & u^0 & 0\\
                        \Big( \frac{\speed^2}{(1 + \speed^2)} \Big) \Pi^{30} & 0 & 0 & u^0 \\
                    \end{pmatrix}, \label{E:A0def} \\
 A^1 & =   \begin{pmatrix} \label{E:A1def}
                        u^1 & (1 + \speed^2) & 0  & 0  \\
                       	\frac{\speed^2}{(1 + \speed^2)} \Pi^{11} & u^1 & 0 & 0 \\
                        \frac{\speed^2}{(1 + \speed^2)} \Pi^{21} & 0 & u^1 & 0\\
                        \frac{\speed^2}{(1 + \speed^2)} \Pi^{31} & 0 & 0 & u^1 \\
                    \end{pmatrix},                   
\end{align}
\end{subequations}
and analogously for $A^2,$ $A^3.$ Furthermore, for later use, we calculate that $\mbox{det}(A^0) = (u^0)^2 \Big\lbrace \overbrace{(u^0)^2(1-\speed^2) + \speed^2}^{(u^0)^2 - \speed^2 \Pi^{00}} \Big\rbrace,$

\begin{align} \label{E:A0inverse}
	(&A^0)^{-1} = \big\lbrace (u^0)^2 - \speed^2 \Pi^{00} \big\rbrace^{-1} \\ 
										& \times \begin{pmatrix}
                       	u^0 & -(1 + \speed^2) \frac{u_1}{u^0} & (1 + \speed^2) \frac{u_2}{u^0} 
                       		& (1 + \speed^2) \frac{u_3}{u^0}  \\
                       	- \frac{\speed^2}{(1 + \speed^2)} \Pi^{10} & 
                       		u^0 - d_1 
                       		& \frac{\speed^2}{(u^0)^2} \Pi^{10}u_2  & \frac{\speed^2}{(u^0)^2} \Pi^{10}u_3 \\
                        - \frac{\speed^2}{(1 + \speed^2)} \Pi^{20} & \frac{\speed^2}{(u^0)^2} \Pi^{20}u_1 
                        	& u^0 - d_2 & 
                        		\frac{\speed^2}{(u^0)^2} \Pi^{20}u_3 \\
                        - \frac{\speed^2}{(1 + \speed^2)} \Pi^{30} & 
                        	\frac{\speed^2}{(u^0)^2} \Pi^{30}u_1 & \frac{\speed^2}{(u^0)^2} \Pi^{30}u_2 
                        	& u^0 - d_3 \\
                    \end{pmatrix}, \notag
\end{align}
and $d_1 \eqdef \frac{\speed^2}{(u^0)^2}\Big(\Pi^{20}u_2 + \Pi^{30}u_3\Big),$
$d_2 \eqdef \frac{\speed^2}{(u^0)^2}\Big(\Pi^{10}u_1 + \Pi^{30}u_3\Big),$
$d_3 \eqdef \frac{\speed^2}{(u^0)^2}\Big(\Pi^{10}u_1 + \Pi^{20}u_2 \Big).$

\section{Norms, Energies, and the Equations of Variation} \label{S:NORMSANDENERGIES}
 
In this section, we define some Sobolev norms and related energies, all of which are used in the proof of our main stability theorem. Our choice of norms is motivated in part by the continuation principle (Proposition \ref{P:CONTINUATION}); with the help of Sobolev embedding, our norms will control the quantities appearing in the statement of the continuation principle. The energies are introduced in order to control the up-to-top-order spatial derivatives of the solution. Unlike the Sobolev norm $\fluidnorm{N},$ the energy's time derivative can be controlled in terms of the energy itself with the help of the divergence theorem. This is because our energy is defined with the help of \emph{energy currents} $\dot{J}^{\mu},$ which are  solution-dependent, coercive vectorfields whose divergence can be controlled; see the discussion in Section \ref{SSS:CURRENTS}. We also introduce the equations of variation, which are the PDEs verified by the spatial derivatives of the solution. The structure of the equations of variation plays an important role in our derivation of estimates for the spatial derivatives.

\subsection{Norms}

\begin{definition}  [\textbf{Norms}] \label{D:NORMS}
	Let $N$ be a positive integer, let $\mathbf{W} \eqdef (\Rlog,u^1,u^2,u^3)^T$ be the array of fluid 
	variables. In the cases $0 < \speed^2 < 1/3,$ we define the lower-order fluid velocity norm $\mathcal{U}_{N-1}(t) \geq 0$ by
	
	\begin{align}
		\mathcal{U}_{N-1} & \eqdef e^{\Omega}\DecayFunction(\Omega) 
			\Big(\sum_{j=1}^3 \| u^j \|_{H^{N-1}}^2 \Big)^{1/2} =  e^{\Omega}\DecayFunction(\Omega)
		\Big(\sum_{j=1}^3 \sum_{|\vec{\alpha}| \leq N - 1} \int_{\mathbb{R}^3} (\partial_{\vec{\alpha}} u^j)^2 \, d^3 x
			\Big)^{1/2}, \label{E:UNMINUSONEDEF} 
	\end{align}
	
	In the cases $0 \leq \speed^2 \leq 1/3,$ we define the fluid norm $\fluidnorm{N}(t) \geq 0$ by
	
	\begin{align}
	\fluidnorm{N} & \eqdef \left\lbrace \begin{array}{ll} 
		\| \Rlog \|_{H^{N-1}} + e^{2\Omega} \sum_{j=1}^3 \| u^j \|_{H^N}, & \speed^2 = 0, \\
	   \|\Rlog \|_{H^N} + e^{\Omega} \sum_{j=1}^3 \| u^j \|_{H^N}
	    + \mathcal{U}_{N-1}, & 0 < \speed^2 < 1/3, \\
	    \|\Rlog \|_{H^N} + e^{\Omega} \sum_{j=1}^3 \| u^j \|_{H^N}, & \speed^2 = 1/3.
		\end{array} \right. \label{E:FLUIDNORMDEF}
	  \end{align}
	Above, $\DecayFunction$ is the function from the hypotheses \ref{A:A2} - \ref{A:A3} of Section \ref{S:INTRO}.
	
	In the case $\speed^2 = 0,$ we also define the fluid velocity norm $\mathcal{S}_{N;velocity}(t) \geq 0$ by
	
	\begin{align}
		\mathcal{S}_{N;velocity} & \eqdef e^{2\Omega} \sum_{j=1}^3 \| u^j \|_{H^N}.  \label{E:FLUIDNORMDEFVELOCITYDUST} 
	\end{align}
\end{definition}

\begin{remark} \label{R:QREMARK}
	The proof of our main stability theorem will show that all of the norms remain uniformly small for all future times if they 
	are initially small. Furthermore, in the cases $0 < \speed^2 < 1/3,$ the boundedness of $\mathcal{U}_{N-1}$ 
	shows that the lower-order derivatives of $u^j$ have an $L^2$ norm that decays at least as fast as $e^{- \Omega} 
	\DecayFunction^{-1}(\Omega),$ while  the boundedness of $\fluidnorm{N}$ shows only that the top-order derivatives (i.e. 
	$|\vec{\alpha}| = N$) of $u^j$ decay at least as fast as $e^{- \Omega};$ see Section \ref{SSS:DIVERGENCECOMMENTS} for a 
	discussion of the importance of the improved decay for the lower-order derivatives of the $u^j.$ We remark that in the 
	special case $\speed^2 = 0,$ the boundedness of $\fluidnorm{N}$ shows that all spatial derivatives of $u^j$ decay in $L^2$ at 
	least as fast as $e^{-2\Omega}.$ Note also that in the case $\speed^2 = 1/3,$ we do not prove an improved decay rate for the 
	lower-order derivatives of $u^j$
	(and hence the $\mathcal{U}_{N-1}$ norm is not needed for the analysis in this case).
	Our inability to show improved decay in this case is intimately connected to the fact that when $\speed^2 = 1/3,$ the 
	$-(3 \speed^2 - 2) \omega u^j$ term on the right-hand side of equation \eqref{E:PARTIALTUJ} suggests that we can only prove 
	that the lower-order derivatives of $u^j$ decay like $e^{-\Omega}$ (i.e., the same decay as for the top-order derivatives).
\end{remark}

\subsection{The equations of variation and energies for the fluid variables}\label{SS:ENERGY}
In this section, we define the fluid energy. This energy, which depends on the up-to-top-order spatial derivatives $(\partial_{\vec{\alpha}}\Rlog,\partial_{\vec{\alpha}} u^1, \partial_{\vec{\alpha}} u^2, \partial_{\vec{\alpha}} u^3)$ $|\vec{\alpha}| \leq N,$ plays a central role in the proof of our main stability theorem. We note the following previously mentioned exception to the preceding sentence: when $\speed^2 = 0$ we will only be able to control $\partial_{\vec{\alpha}}\Rlog$ for $|\vec{\alpha}| \leq N - 1.$ In order to study the evolution of the solution's derivatives, we will have to commute the equations \eqref{E:FINALEULERP} - \eqref{E:FINALEULERUJ} (or equivalently equation \eqref{E:MATRIXVECTOREULER}) with the operator $\partial_{\vec{\alpha}}.$ The quantities $(\partial_{\vec{\alpha}}\Rlog,\partial_{\vec{\alpha}} u^1, \partial_{\vec{\alpha}} u^2, \partial_{\vec{\alpha}} u^3)$ verify linear (in $(\partial_{\vec{\alpha}}\Rlog,\partial_{\vec{\alpha}} u^1, \partial_{\vec{\alpha}} u^2, \partial_{\vec{\alpha}} u^3)$) PDEs with principal coefficients that depend on the solution and inhomogeneous terms that depend on $(\partial_{\vec{\beta}}\Rlog,\partial_{\vec{\beta}} u^1, \partial_{\vec{\beta}} u^2, \partial_{\vec{\beta}} u^3),$ $|\vec{\beta}| \leq |\vec{\alpha}|;$ see Lemma \ref{L:EOVINHOMOGENEOUS} for the details. We refer to this system of PDEs as the \emph{equations of variation}, while the unknowns $(\dot{\Rlog}, \dot{u}^1, \dot{u}^2, \dot{u}^3) \eqdef (\partial_{\vec{\alpha}}\Rlog,\partial_{\vec{\alpha}} u^1, \partial_{\vec{\alpha}} u^2, \partial_{\vec{\alpha}} u^3)$ are called the \emph{variations}. More specifically, we define the equations of variation in the unknowns 
$(\dot{\Rlog}, \dot{u}^1, \dot{u}^2, \dot{u}^3)$ corresponding to $(\Rlog,u^1,u^2,u^3)$ as follows:

\begin{center}
	{\large \textbf{Equations of Variation}}
\end{center}

\begin{subequations}
\begin{align}
	u^{\alpha} \partial_{\alpha} \dot{\Rlog} 
		+ (1 + \speed^2) \Big(\frac{1}{u^0}\Big) u_a \partial_t \dot{u}^a + (1 + \speed^2) 
		\partial_a \dot{u}^a & = \mathfrak{F}, \label{E:EOV1} \\
	u^{\alpha} \partial_{\alpha} \dot{u}^j + \frac{\speed^2}{(1 + \speed^2)}\Pi^{j \alpha} \partial_{\alpha} \dot{\Rlog}
		& = \omega (3 \speed^2 - 2)u^0 \dot{u}^j + \mathfrak{G}^j. \label{E:EOV2}
\end{align}
\end{subequations}
The terms $\mathfrak{F},$ $\omega (3 \speed^2 - 2)u^0 \dot{u}^j,$ and $\mathfrak{G}^j$ denote the inhomogeneous terms that arise from commuting \eqref{E:FINALEULERP} - \eqref{E:FINALEULERUJ} with $\partial_{\vec{\alpha}}.$ Note that we have split the inhomogeneous term in \eqref{E:EOV2} into two parts; the $\omega (3 \speed^2 - 2)u^0 \dot{u}^j$ term is primarily responsible for creating decay in $\dot{u}^j,$ while $\mathfrak{G}^j$ will be shown to be an error term.

Using matrix-vector notation, we can rewrite \eqref{E:EOV1}- \eqref{E:EOV2} as  

\begin{align} \label{E:EOVMATRIXVECTOR}
	A^{\beta} \partial_{\beta} \dot{\mathbf{W}} = \mathbf{I},
\end{align}
where them matrices $A^{\mu}$ are defined in \eqref{E:A0def} - \eqref{E:A1def}, and

\begin{align}
	\dot{\mathbf{W}} & \eqdef (\dot{\Rlog}, \dot{u}^1, \dot{u}^2, \dot{u}^3)^T, \\
	\mathbf{I} & \eqdef \omega (3 \speed^2 - 2)u^0 \big(0, \dot{u}^1, \dot{u}^2, \dot{u}^3 \big)^T
		+ (\mathfrak{F}, \mathfrak{G}^1, \mathfrak{G}^2, \mathfrak{G}^3)^T.
\end{align}

To each variation $(\dot{u}^1, \dot{u}^2, \dot{u}^3),$ we associate a quantity $\dot{u}^0$ defined by

\begin{align} \label{E:DOT0INTERMSOFDOTJ}
	\dot{u}^0 \eqdef  \frac{1}{u^0} g_{ab} u^a \dot{u}^b.
\end{align}
This quantity appears below in the expression \eqref{E:ENERGYCURRENT}, which defines our fluid energy current. 
The importance of the definition \eqref{E:DOT0INTERMSOFDOTJ} is that it leads to the identity $g_{\alpha \beta}u^{\alpha} \dot{u}^{\beta}=0,$ which is essential for the derivation of the divergence identity \eqref{E:DIVDOTJ} below. In our analysis, we will need the following lemma, which essentially states that $\dot{u}^0$ is a solution to a linearization of \eqref{E:U0EQUATION} around $(\Rlog,u^1,u^2,u^3).$

\begin{lemma} \label{L:dot0equation}
Assume that $(\dot{\Rlog},\dot{u}^1, \dot{u}^2, \dot{u}^3)$ is a solution to the equations of variation \eqref{E:EOV1} - \eqref{E:EOV2} corresponding to $(\Rlog,u^1,u^2,u^3),$ and let $\dot{u}^0 \eqdef \frac{1}{u^0} u_a \dot{u}^a$ be as defined in \eqref{E:DOT0INTERMSOFDOTJ}. Then $\dot{u}^0$ verifies the following equation:

\begin{align}
	u^{\alpha} \partial_{\alpha} \dot{u}^0 + \frac{\speed^2}{(1 + \speed^2)}\Pi^{0 \alpha} \partial_{\alpha} \dot{\Rlog}
		& = \mathfrak{G}^0, \label{E:DOTU0EQUATION} 
\end{align}
where

\begin{align}
	\mathfrak{G}^0 & = g_{ab}\bigg[u^{\nu} \partial_{\nu} \Big(\frac{u^a}{u^0}\Big)\bigg] \dot{u}^b 
		+ 3 \speed^2 \omega \dot{u}^0
		+ 2 \omega (u^0 - 1)\dot{u}^0
		+ \Big(\frac{1}{u^0}\Big) g_{ab} u^a \mathfrak{G}^b.
		\label{E:dotu0G0def}
\end{align}

\end{lemma}

\begin{proof}
	By the definition of $\dot{u}^0,$ the left-hand side of \eqref{E:DOTU0EQUATION} is equal to
	
	\begin{align} \label{E:dotu0equationidentity1}
		\frac{u_a}{u^0} u^{\alpha} \partial_{\alpha} \dot{u}^a + \frac{\speed^2}{(1 + \speed^2)}\Pi^{0 \alpha} \partial_{\alpha} 
			\dot{\Rlog}
		+ g_{ab} \bigg[u^{\alpha} \partial_{\alpha} \Big(\frac{u^a}{u^0}\Big)\bigg] \dot{u}^b 
		+ 2 \omega g_{ab} u^a \dot{u}^b.
	\end{align}
	
	On the other hand, contracting equation \eqref{E:EOV2} against $u_j$ and using the identity 
	$u_a \Pi^{a \alpha} = - u_0 \Pi^{0 \alpha} = u^0 \Pi^{0 \alpha},$ we conclude that
	
	\begin{align} \label{E:dotu0equationidentity2}
		u_a u^{\alpha} \partial_{\alpha} \dot{u}^a  + u^0 \frac{\speed^2}{(1 + \speed^2)}\Pi^{0 \alpha} \partial_{\alpha} 
		\dot{\Rlog} = (3 \speed^2 - 2)\omega g_{ab} u^a \dot{u}^b + g_{ab} u^a \mathfrak{G}^b.
	\end{align}
	Multiplying \eqref{E:dotu0equationidentity2} by $\frac{1}{u^0}$ and using \eqref{E:dotu0equationidentity1},
	we arrive at \eqref{E:DOTU0EQUATION}.
\end{proof}

In the next lemma, we investigate the structure of the inhomogeneous terms in the equations of variation verified by a solution's derivatives $(\partial_{\vec{\alpha}} \Rlog, \partial_{\vec{\alpha}} u^1, \partial_{\vec{\alpha}}u^2, \partial_{\vec{\alpha}}u^3)^T.$ We again split the inhomogeneous terms into a decay-inducing piece 
$\mathbf{b}_{\vec{\alpha}}$ and a small error term $\mathbf{b}_{\triangle \vec{\alpha}}.$

\begin{lemma} \label{L:EOVINHOMOGENEOUS}
Let $\mathbf{W} \eqdef (\Rlog, u^1,u^2,u^3)^T$ be a solution to the relativistic Euler equations \eqref{E:MATRIXVECTOREULER}, 
i.e., $A^{\beta} \partial_{\beta} \mathbf{W} = \mathbf{b},$ \\ 
$\mathbf{b} \eqdef \omega \Big(- \frac{(1 + \speed^2)}{u^0} g_{ab}u^a u^b, (3 \speed^2 - 2) u^0 u^1, (3 \speed^2 - 2) u^0 u^2, (3 \speed^2 - 2) u^0 u^3 \Big)^{T}.$ Then $(\dot{\Rlog},\dot{u}^1,\dot{u}^2,\dot{u}^3)^T $ $\eqdef (\partial_{\vec{\alpha}} \Rlog, \partial_{\vec{\alpha}} u^1, \partial_{\vec{\alpha}}u^2, \partial_{\vec{\alpha}}u^3)^T$ is a solution to the equations of variation \eqref{E:EOVMATRIXVECTOR} with an inhomogeneous term $\mathbf{I}$ that can be expressed as follows:

\begin{align} \label{E:Ialphadecomp}
	\mathbf{I} \eqdef \mathbf{b}_{\vec{\alpha}} + \mathbf{b}_{\triangle \vec{\alpha}},
\end{align}
where

\begin{subequations}
\begin{align} \label{E:bdef}
	\mathbf{b}_{\vec{\alpha}} & \eqdef 
		\omega (3 \speed^2 - 2)u^0 \Big(0, \partial_{\vec{\alpha}}u^1, \partial_{\vec{\alpha}} u^2, \partial_{\vec{\alpha}} u^3 
			\Big)^T, \\
		\mathbf{b}_{\triangle \vec{\alpha}} & \eqdef \big(\mathfrak{F}_{\vec{\alpha}}, \mathfrak{G}_{\vec{\alpha}}^1, 
		\mathfrak{G}_{\vec{\alpha}}^2, \mathfrak{G}_{\vec{\alpha}}^3 \big)^T
		= \big\lbrace \partial_{\vec{\alpha}} \mathbf{b} - \mathbf{b}_{\vec{\alpha}} \big\rbrace
		+ \Big\lbrace A^0 \partial_{\vec{\alpha}} \big[ (A^0)^{-1}\mathbf{b} \big] 
		- \partial_{\vec{\alpha}} \mathbf{b} \Big\rbrace \label{E:FalphamathfrakGjalphainhomogeneousterms}  \\
		& \ \ + A^0 \Big\lbrace(A^0)^{-1}A^a \partial_a \partial_{\vec{\alpha}} \mathbf{W}
 			- \partial_{\vec{\alpha}} \big[(A^0)^{-1} A^a \partial_a \mathbf{W}\big] \Big\rbrace. \notag
\end{align}

\end{subequations}
Furthermore, $\dot{u}^0 \eqdef \frac{1}{u^0} u_a \dot{u}^a$ is a solution to equation 
\eqref{E:DOTU0EQUATION} with an inhomogeneous term $\mathfrak{G}_{\vec{\alpha}}^0$ defined by

\begin{align}
 	\mathfrak{G}_{\vec{\alpha}}^0 & = g_{ab}\bigg[u^{\nu} \partial_{\nu} \Big(\frac{u^a}{u^0}\Big)\bigg] 
 		\partial_{\vec{\alpha}} u^b 
		+  \omega\Big(\frac{3 \speed^2 - 2 + 2 u^0}{u^0}\Big) g_{ab} u^a \partial_{\vec{\alpha}} u^b
		+ \Big(\frac{1}{u^0}\Big) g_{ab} u^a \mathfrak{G}_{\vec{\alpha}}^b. \label{E:mathfrakG0alphainhomogeneousterm}
\end{align}

\end{lemma}

\begin{proof}
	Equations \eqref{E:Ialphadecomp} - \eqref{E:FalphamathfrakGjalphainhomogeneousterms} 
	are a straightforward decomposition of the inhomogeneous term $\mathbf{I} = A^0 \partial_{\vec{\alpha}} \big\lbrace 
	(A^0)^{-1} A^{\mu} \partial_{\mu} \mathbf{W} \big\rbrace + A^0 [(A^0)^{-1}A^{\mu} \partial_{\mu}, \partial_{\vec{\alpha}}] 
	\mathbf{W}$ \\
	$= A^0 \partial_{\vec{\alpha}} \big\lbrace (A^0)^{-1} \mathbf{b}  \big\rbrace
	+ A^0 [(A^0)^{-1}A^a \partial_a, \partial_{\vec{\alpha}}] \mathbf{W},$ where $[\cdot,\cdot]$
	denotes the commutator. The relation \eqref{E:mathfrakG0alphainhomogeneousterm} follows directly from \eqref{E:dotu0G0def}.
	
\end{proof}

\subsubsection{The fluid energy currents}

Our energy, which will control the up-to-top-order spatial derivatives of the solution, will be defined with the help of \emph{energy current} vectorfields $\dot{J}^{\mu}.$ 

\begin{definition} [\textbf{Currents}] \label{D:ENERGYCURRENT}
In the cases $0 < \speed^2 \leq 1/3,$ to each variation $\dot{\mathbf{W}}=(\dot{\Rlog},\dot{u}^1,\dot{u}^2,\dot{u}^3)^T,$ we associate the following energy current, where $\dot{u}^0 \eqdef \frac{1}{u^0}u_a \dot{u}^a:$ 

\begin{align} \label{E:ENERGYCURRENT}
	\dot{J}^{\mu} & \eqdef \frac{\speed^2 u^{\mu}}{(1 + \speed^2)}\dot{\Rlog}^2 + 2 \speed^2 \dot{u}^{\mu} \dot{\Rlog} 
		+ (1 + \speed^2) u^{\mu} g_{\alpha \beta} \dot{u}^{\alpha} \dot{u}^{\beta}.
\end{align}

In the case $\speed^2 = 0,$ we associate both a velocity energy current and a density energy current:

\begin{subequations}
\begin{align} \label{E:ENERGYCURRENTDUSTU}
	\dot{J}_{velocity}^{\mu}[(\dot{u}^1,\dot{u}^2,\dot{u}^3),(\dot{u}^1,\dot{u}^2,\dot{u}^3)] 
		& \eqdef e^{2 \Omega} u^{\mu} g_{\alpha \beta} \dot{u}^{\alpha} \dot{u}^{\beta}, \\
	\dot{J}_{density}^{\mu}[\dot{\Rlog},\dot{\Rlog}] & \eqdef u^{\mu} \dot{\Rlog}^2. \label{E:ENERGYCURRENTDUSTRHO}
\end{align}
\end{subequations}

\end{definition}
Similar currents have been used in \cite{dC2007}, \cite{jS2008b}, \cite{jS2008a}, and \cite{jS2011}. See the discussion in Section \ref{SSS:CURRENTS}.

\begin{remark}
	We use the notation $\dot{J}^{\mu}[\cdot, \cdot]$ when we want to emphasize that
	$\dot{J}^{\mu}$ depends quadratically on the variations.
\end{remark}

\subsubsection{The coordinate divergence of the currents}

In this section, we provide expressions for the coordinate divergence of the currents. These expressions will play a key role in our analysis of the evolution of the energies $\fluidenergy{N},$ $\mathcal{E}_{N;velocity},$ and $\mathcal{E}_{N-1;density}.$

\begin{lemma} [\textbf{Divergence of the Currents}]  \label{L:DIVDOTJ}
Let $\dot{\mathbf{W}}=(\dot{\Rlog},\dot{u}^1,\dot{u}^2,\dot{u}^3)^T$ be a solution to the equations of variation 
\eqref{E:EOV1} - \eqref{E:EOV2}. Then in the cases $0 < \speed^2 \leq 1/3,$ the coordinate divergence of $\dot{J}^{\mu}$ can be expressed as follows:

\begin{align} \label{E:DIVDOTJ}
	\partial_{\mu} \big(\dot{J}^{\mu}[\dot{\mathbf{W}}, \dot{\mathbf{W}}] \big) 
		& = \frac{\speed^2 (\partial_{\mu} u^{\mu})}{(1 + \speed^2)}\dot{\Rlog}^2 
			+ (1 + \speed^2)(\partial_{\mu} u^{\mu}) 
	\Big(-(\dot{u}^{0})^2 + g_{ab} \dot{u}^{a} \dot{u}^{b} \Big) \\
	& \ \ + 2 \speed^2 g_{ab} \bigg(\partial_t \Big[\frac{u^a}{u^0} \Big] \bigg)\dot{u}^b \dot{\Rlog} 
		+ 4 \speed^2 \omega  \frac{g_{ab} u^a \dot{u}^b}{u^0} \dot{\Rlog} \notag \\
	& \ \ + \underbrace{2(1+\speed^2)(3 \speed^2 - 1) \omega g_{ab}\dot{u}^a 	
		\dot{u}^b}_{\mbox{$\leq 0$ if $\speed^2 \leq 1/3$}} \notag \\
	& \ \ + \frac{2 \speed^2 \mathfrak{F}}{(1 + \speed^2)} \dot{\Rlog}
		- 2(1+\speed^2) \mathfrak{G}^0 \dot{u}^0
		+ 2(1+\speed^2) g_{ab} \mathfrak{G}^a \dot{u}^b. \notag 
\end{align}

In the case $\speed^2 = 0,$ the coordinate divergences of $\dot{J}_{velocity}^{\mu}$ and $\dot{J}_{density}^{\mu}$
can be respectively expressed as 

\begin{subequations}
\begin{align}
	\partial_{\mu} \big(\dot{J}_{velocity}^{\mu}[(\dot{u}^1,\dot{u}^2,\dot{u}^3),(\dot{u}^1,\dot{u}^2,\dot{u}^3)]\big)
	& = e^{2 \Omega} (\partial_{\mu} u^{\mu}) \Big(-(\dot{u}^{0})^2 + g_{ab} \dot{u}^{a} \dot{u}^{b} \Big)
	\label{E:DIVDOTJ0SPEEDVELOCITY} \\
	& \ \ + 2 \omega e^{2 \Omega} (u^0 - 1) g_{ab} \dot{u}^a \dot{u}^b
		- 2 \omega e^{2 \Omega} u^0 (\dot{u}^{0})^2 \notag \\
	& \ \ - 2 e^{2 \Omega} \mathfrak{G}^0 \dot{u}^0 + 2 e^{2 \Omega} g_{ab} \mathfrak{G}^a \dot{u}^b, \notag \\
	\partial_{\mu} \big(\dot{J}_{density}^{\mu}[\dot{\Rlog}, \dot{\Rlog}] \big) & = (\partial_{\mu} u^{\mu}) \dot{\Rlog}^2 
		+ \frac{4}{u^0} \omega g_{ab}u^a \dot{u}^b \dot{L}  \label{E:DIVDOTJ0SPEEDDENSITY} \\
	& \ \ - \frac{2}{(u^0)^2} g_{ab} u^a \mathfrak{G}^b \dot{L}
		+ \frac{2}{(u^0)^2} g_{ab} u^a (u^k \partial_k \dot{u}^b) \dot{L} \notag \\
	& \ \ - 2 (\partial_a \dot{u}^a) \dot{L} 
		+ 2 \mathfrak{F} \dot{L}. \notag
\end{align}	
\end{subequations}

\end{lemma}

\begin{remark}
We stress that the right-hand sides of \eqref{E:DIVDOTJ} and \eqref{E:DIVDOTJ0SPEEDVELOCITY} - \eqref{E:DIVDOTJ0SPEEDDENSITY}
do not depend on the derivatives of the variations. This property is essential for our analysis of the energies. We have organized the terms on the right-hand sides of \eqref{E:DIVDOTJ} and \eqref{E:DIVDOTJ0SPEEDVELOCITY} - \eqref{E:DIVDOTJ0SPEEDDENSITY} in order to help us deduce the estimates \eqref{E:DIVJDOTL1DUSTVELOCITY} - \eqref{E:DIVJDOTL1DUSTDENSITY} and \eqref{E:DIVJDOTL1}.
\end{remark}

\begin{proof}
	To deduce \eqref{E:DIVDOTJ}, one can take the divergence of \eqref{E:ENERGYCURRENT} and then
	use the equations \eqref{E:EOV1} - \eqref{E:EOV2} and \eqref{E:DOTU0EQUATION} to replace 
	the derivatives of $\dot{\mathbf{W}}$ with inhomogeneous terms. We leave the tedious details up to the reader. The proofs of 
	\eqref{E:DIVDOTJ0SPEEDVELOCITY} - \eqref{E:DIVDOTJ0SPEEDDENSITY} are similar but simpler.
\end{proof}

\subsubsection{The definition of the fluid energies}

We are now ready to define our fluid energies, which will be used to control all spatial derivatives of
the solution, including those of top order.

\begin{definition}  [\textbf{Energies}] \label{D:ENERGY}
	Let $\mathbf{W} \eqdef (\Rlog,u^1,u^2,u^3)^T$ be the array of fluid variables, and let
	$\dot{J}^{\mu}[\cdot,\cdot],$ $\dot{J}_{velocity}^{\mu}[\cdot,\cdot],$ $\dot{J}_{density}^{\mu}[\cdot,\cdot]$
	be the currents defined in \eqref{E:ENERGYCURRENT}, \eqref{E:ENERGYCURRENTDUSTU} - \eqref{E:ENERGYCURRENTDUSTRHO}.
	In the cases $0 < \speed^2 \leq 1/3,$ we define the fluid energy $\fluidenergy{N}(t) \geq 0$ (see inequality 
	\eqref{E:JDOT0INEQUALITY}) by
	
	\begin{align} \label{E:FLUIDENERGYDEF}
		\fluidenergy{N}^2(t) & \eqdef 
	  	\sum_{|\vec{\alpha}| \leq N} \int_{\mathbb{R}^3} \dot{J}^0[\partial_{\vec{\alpha}} 
				\mathbf{W},\partial_{\vec{\alpha}} \mathbf{W}] \, d^3 x.
		\end{align}

In the case $\speed^2 = 0,$ we define the energies $\mathcal{E}_{N;velocity}(t) \geq 0$
and $\mathcal{E}_{N-1;density}(t) \geq 0$ by

\begin{subequations}
\begin{align} \label{E:FLUIDENERGYDEFDUSTU}
	\mathcal{E}_{N;velocity}^2(t) & \eqdef \sum_{|\vec{\alpha}| \leq N} \int_{\mathbb{R}^3} 	
		\dot{J}_{velocity}^0[\partial_{\vec{\alpha}} 
		(u^1,u^2,u^3),\partial_{\vec{\alpha}} (u^1,u^2,u^3)] \, d^3 x, \\		
	\mathcal{E}_{N-1;density}^2(t) & \eqdef \sum_{|\vec{\alpha}| \leq N - 1} \int_{\mathbb{R}^3} 	
		\dot{J}_{density}^0[\partial_{\vec{\alpha}} \Rlog ,\partial_{\vec{\alpha}}\Rlog] \, d^3 x.
		\label{E:FLUIDENERGYDEFDUSTDENSITY}
\end{align}
\end{subequations}

\end{definition}

In the next corollary, we provide a preliminary estimate for the time derivatives of some of the fluid-controlling
quantities. These estimates are the starting point for our proof of Proposition \ref{P:INTEGRALENERGYINEQUALITIES}, which
provides the integral inequalities that form the crux of our future stability theorem.

\begin{corollary} [\textbf{Preliminary Norm and Energy Inequalities}] \label{C:fluidenergytimederivative}
	Let $\mathbf{W} \eqdef (\Rlog, u^1,u^2,u^3)^T$ be a classical solution to the relativistic Euler equations 
	\eqref{E:MATRIXVECTOREULER}. Let $\mathcal{U}_{N-1}(t),$ $\fluidenergy{N}(t),$ $\mathcal{E}_{N;velocity},$
	and $\mathcal{E}_{N - 1;density}$ be the norm and energies defined in Definitions \ref{D:NORMS} and \ref{D:ENERGY}. 
	Then in the cases $0 < \speed^2 \leq 1/3,$ the following identity holds:
	
	\begin{align}
	\frac{d}{dt}\big(\fluidenergy{N}^2 \big) & =  
	    \sum_{|\vec{\alpha}| \leq N} \int_{\mathbb{R}^3} \partial_{\mu} \big(\dot{J}^{\mu}[\partial_{\vec{\alpha}} 
			\mathbf{W},\partial_{\vec{\alpha}} \mathbf{W}] \big) \, d^3 x. \label{E:ENTIMEDERIVATIVE}
	\end{align}
	
	In the cases $0 < \speed^2 < 1/3,$ the following differential inequality holds:
	
	\begin{align} \label{E:UNORMTIMEDERIVATIVE}
		\frac{d}{dt}\big(\mathcal{U}_{N-1}^2 \big) & \leq 
			2\Big\lbrace 3\speed^2 - 1 + \frac{\DecayFunction'(\Omega)}{\DecayFunction(\Omega)} \Big\rbrace\omega \mathcal{U}_{N-1}^2
			+ 2 \mathcal{U}_{N-1} \sum_{a=1}^3 e^{\Omega} \DecayFunction(\Omega) \| \triangle'^a \|_{H^{N-1}}.
	\end{align}
	
	In the case $\speed^2 = 0,$ the following identities hold:
	
	\begin{subequations}
	\begin{align} 
		\frac{d}{dt} \big(\mathcal{E}_{N;velocity}^2\big) 
		& = \sum_{|\vec{\alpha}| \leq N} \int_{\mathbb{R}^3} \partial_{\mu} \big(\dot{J}^{\mu}[\partial_{\vec{\alpha}} 
			(u^1,u^2,u^3),\partial_{\vec{\alpha}} (u^1,u^2,u^3)] \big) \, d^3 x, 
			\label{E:FLUIDENERGYTIMEDERIVATIVE0SPEEDVELOCITY} \\
		\frac{d}{dt} \big(\mathcal{E}_{N - 1;density}^2 \big) 
		& = \sum_{|\vec{\alpha}| \leq N - 1} \int_{\mathbb{R}^3} 	
			\partial_{\mu} \big( \dot{J}_{density}^{\mu}[\partial_{\vec{\alpha}} \Rlog ,\partial_{\vec{\alpha}} \Rlog ] \big) 
			\, d^3 x. \label{E:FLUIDENERGYTIMEDERIVATIVE0SPEEDDENSITY}
	\end{align}
	\end{subequations}
	
\end{corollary}

\begin{proof}
	To prove \eqref{E:UNORMTIMEDERIVATIVE}, we use the definition \eqref{E:UNMINUSONEDEF} of $\mathcal{U}_{N-1}$
	and equation \eqref{E:PARTIALTUJ}  (differentiated with $\partial_{\vec{\alpha}}$) to conclude that
	
	\begin{align} \label{E:PARTIALUNMINUSONESQUAREDEXPRESSION}
		\frac{d}{dt}\big(\mathcal{U}_{N-1}^2 \big) & = 
			2\Big\lbrace 3\speed^2 - 1 + \frac{\DecayFunction'(\Omega)}{\DecayFunction(\Omega)} \Big\rbrace\omega \mathcal{U}_{N-1}^2
			+ 2 e^{2 \Omega} \DecayFunction^2(\Omega) \sum_{|\vec{\alpha}| \leq N-1} \sum_{a=1}^3 \int_{\mathbb{R}^3}
			(\partial_{\vec{\alpha}} u^a) \partial_{\vec{\alpha}} \triangle'^a \, d^3 x. 
	\end{align}
	Inequality \eqref{E:UNORMTIMEDERIVATIVE} now follows from \eqref{E:PARTIALUNMINUSONESQUAREDEXPRESSION},
	the definition of $\mathcal{U}_{N-1},$ and the Cauchy-Schwarz inequality for integrals. 
	
	\eqref{E:FLUIDENERGYTIMEDERIVATIVE0SPEEDDENSITY} and \eqref{E:FLUIDENERGYTIMEDERIVATIVE0SPEEDVELOCITY} - 
	\eqref{E:FLUIDENERGYTIMEDERIVATIVE0SPEEDDENSITY} follow easily from the definitions \eqref{E:FLUIDENERGYDEF}, 
	\eqref{E:FLUIDENERGYDEFDUSTU} - \eqref{E:FLUIDENERGYDEFDUSTDENSITY}, and the divergence theorem.

\end{proof}

\section{Sobolev Estimates} \label{S:SOBOLEV}

In this section, we derive Sobolev estimates for the inhomogeneous terms appearing in equations \eqref{E:FINALEULERP} - 
\eqref{E:FINALEULERUJ}, \eqref{E:U0EQUATION}, and \eqref{E:PARTIALTRLOG} - \eqref{E:PARTIALTUJ}, for the terms in Lemma \ref{L:EOVINHOMOGENEOUS}, and for the right-hand sides of \eqref{E:UNORMTIMEDERIVATIVE} - \eqref{E:PARTIALUNMINUSONESQUAREDEXPRESSION}. All of the bounds are derived under a smallness assumption on the norm $\fluidnorm{N}.$ These bounds form the backbone of Section \ref{S:INTEGRALINEQUALITY}, where they are used to help derive integral inequalities for $\mathcal{U}_{N-1}$ and $\fluidenergy{N},$ $\mathcal{E}_{N;velocity},$ and $\mathcal{E}_{N-1;density}.$ We collect together all of the estimates in the next proposition.

\begin{remark}
	Note that in following proposition, we establish different estimates for the special cases $\speed^2 = 0$ and $\speed^2 =1/3$ 
	compared to the other cases.
\end{remark}

\begin{proposition} [\textbf{Sobolev Estimates}] \label{P:NONLINEARITIES}
	Assume that $0 \leq \speed^2 \leq 1/3.$ Let
	$(\Rlog,u^1,u^2,u^3)$ be a classical solution to the relativistic Euler equations \eqref{E:FINALEULERP} - 
	\eqref{E:FINALEULERUJ} on the spacetime slab $[1,T) \times \mathbb{R}^3,$ and let $\fluidnorm{N}$ be the solution 
	norm defined in Definition \ref{D:NORMS}. Assume that $\fluidnorm{N}(t) \leq \epsilon$ on $[1,T).$ Then if $\epsilon$ is 
	sufficiently small, the following estimate for $u^0$ holds on $[1,T):$
	
	\begin{align} 
		\| u^0 - 1 \|_{H^N} & \lesssim \left\lbrace \begin{array}{ll} 
	   e^{- 2 \Omega} \fluidnorm{N;velocity}, & \speed^2 = 0, \\
	    \DecayFunction^{-1}(\Omega) \fluidnorm{N}, & 0 < \speed^2 < 1/3, \\
	    \fluidnorm{N}, & \speed^2 = 1/3. \\
	    \end{array} \right. \label{E:U0UPPERHN}
	\end{align}
	
	For the time derivatives of the fluid quantities, we have the following estimates on $[1,T):$
	
	\begin{subequations}
	\begin{align}
		 & \| \partial_t \Rlog \|_{H^{N-2}} \lesssim (\omega + 1)e^{- 2 \Omega} \mathcal{S}_{N;velocity}, \qquad \speed^2 = 0, 
		 \label{E:PARTIALTRLOGHNMINUSTWO} \\
	   & \| \partial_t \Rlog \|_{H^{N-1}} \lesssim \left\lbrace \begin{array}{ll}
	   	\omega \DecayFunction^{-1}(\Omega) \fluidnorm{N} 
	    	+ e^{- \Omega} \fluidnorm{N}, & 0 < \speed^2 < 1/3, \\
	    e^{- \Omega} \fluidnorm{N}, & \speed^2 = 1/3,
	    \end{array} \right. \label{E:PARTIALTRLOGHNMINUSONE} \\
		& \| \partial_t u^0 \|_{H^{N-1}} \lesssim 
			\left\lbrace \begin{array}{ll} 
	    (\omega + 1) e^{- 2 \Omega} \mathcal{S}_{N;velocity}, & \speed^2 = 0, \\
	     \omega \DecayFunction^{-1}(\Omega) \fluidnorm{N} 
	     + e^{- \Omega} \DecayFunction^{-1}(\Omega) \fluidnorm{N}, & 0 < \speed^2 < 1/3, \\
	    	e^{- \Omega} \fluidnorm{N}, & \speed^2 = 1/3,
	    \end{array} \right. \label{E:PARTIALTU0UPPERHNMINUSONE} \\
		& \| \partial_t u^j \|_{H^{N-1}} \lesssim \left\lbrace \begin{array}{ll} 
	    (\omega + 1) e^{- 2 \Omega} \mathcal{S}_{N;velocity}, & \speed^2 = 0, \\
	    \omega e^{- \Omega} \DecayFunction^{-1}(\Omega) \mathcal{U}_{N-1} + e^{-2 \Omega} \fluidnorm{N}, & 0 < \speed^2 < 1/3, \\ 
	    e^{- \Omega} \fluidnorm{N}, & \speed^2 = 1/3.
	    \end{array} \right. \label{E:PARTIALTUJUPPERHNMINUSONE} 
	\end{align}
	\end{subequations}

For the inhomogeneous terms $\omega \frac{(1 + \speed^2)\Rlog}{u^0} g_{ab}u^a u^b$ and
$(3 \speed^2 - 2) \omega u^0 u^{j}$ from Proposition \ref{P:DECOMPOSITION}, 
we have the following estimates on $[1,T)$ in the cases $0 \leq \speed^2 < 1/3$
(the case $\speed^2 = 1/3$ is not directly needed for the remaining estimates):
	
	\begin{subequations}	
	\begin{align}
		\Big \| \omega \frac{(1 + \speed^2)}{u^0} g_{ab}u^a u^b \Big\|_{H^N} & \lesssim \omega \DecayFunction^{-1}(\Omega) \fluidnorm{N}, 
			\label{E:PINHOMOGENEOUSHN} \\
		\| (3 \speed^2 - 2) \omega u^0 u^{j} \|_{H^N} & \lesssim \omega e^{- \Omega} \fluidnorm{N}, \label{E:UJINHOMOGENEOUSHN} \\
		\| \partial_{\vec{\alpha}} (\omega u^0 u^{j}) - \omega u^0 \partial_{\vec{\alpha}} u^{j} \|_{L^2}
			& \lesssim \omega e^{- \Omega} \DecayFunction^{-1}(\Omega) \fluidnorm{N}, && (|\vec{\alpha}| \leq N). 
			\label{E:PARTIALVECALPHAUJINHOMOGENEOUSNONPRINCIPALL2} 
	\end{align}
	\end{subequations}
	
For the inhomogeneous terms $\triangle',$ $\triangle'^0,$ and $\triangle'^j$ from Corollary \ref{C:ISOLATEPARTIALTPUJ},
we have the following estimates on $[1,T):$ 

\begin{subequations}
\begin{align}
	 & \| \triangle' \|_{H^{N-2}} \lesssim (\omega + 1)e^{- 2 \Omega} \mathcal{S}_{N;velocity}, \qquad \speed^2 = 0, 
	 \label{E:TRIANGLEPRIMEHNMINUSTWO} \\
	  & \| \triangle' \|_{H^{N-1}} \lesssim \left\lbrace \begin{array}{ll}
	   	\omega \DecayFunction^{-1}(\Omega) \fluidnorm{N} 
	    	+ e^{- \Omega} \fluidnorm{N}, & 0 < \speed^2 < 1/3, \\
	    e^{- \Omega} \fluidnorm{N}, & \speed^2 = 1/3,
	    \end{array} \right.  \label{E:TRIANGLEPRIMEHNMINUSONE} \\
	& \| \triangle'^0 \|_{H^{N-1}} \lesssim 
			\left\lbrace \begin{array}{ll} 
	    (\omega + 1) e^{- 2 \Omega} \mathcal{S}_{N;velocity}, & \speed^2 = 0, \\
	     \omega \DecayFunction^{-1}(\Omega) \fluidnorm{N} 
	     + e^{- \Omega} \DecayFunction^{-1}(\Omega) \fluidnorm{N}, & 0 < \speed^2 < 1/3, \\
	    	e^{- \Omega} \fluidnorm{N}, & \speed^2 = 1/3,
	    \end{array} \right. \label{E:TRIANGLEPRIME0HNMINUSONE} \\
	& \| \triangle'^j \|_{H^{N-1}} \left\lbrace \begin{array}{ll} 
	    (\omega + 1) e^{- 2 \Omega} \mathcal{S}_{N;velocity}, & \speed^2 = 0, \\
	    \omega e^{- \Omega} \DecayFunction^{-2}(\Omega) \mathcal{U}_{N-1} + e^{-2 \Omega} \fluidnorm{N}, & 0 < \speed^2 < 1/3, \\ 
	    e^{-2 \Omega} \fluidnorm{N}, & \speed^2 = 1/3.
	    \end{array} \right. \label{E:TRIANGLEPRIMEJHNMINUSONE} 
\end{align}
\end{subequations}
	
	For the $L^2$ norm of the variation $\dot{u}^0 = \frac{1}{u^0} u_a \dot{u}^a$ defined in \eqref{E:DOT0INTERMSOFDOTJ}, we have 
	the following estimate on $[1,T):$
	
	\begin{align} \label{E:DOTU0L2INTERMSOFDOTUAL2}
		\| \dot{u}^0 \|_{L^2} & \lesssim \left\lbrace \begin{array}{ll} 
	    \mathcal{S}_{N;velocity} \sum_{a=1}^3 \| \dot{u}^a \|_{L^2}, & \speed^2 = 0, \\
	    e^{\Omega} \DecayFunction^{-1}(\Omega)\fluidnorm{N} \sum_{a=1}^3 \| \dot{u}^a \|_{L^2}, & 0 < \speed^2 < 1/3, \\
	    e^{\Omega} \fluidnorm{N} \sum_{a=1}^3 \| \dot{u}^a \|_{L^2}, & \speed^2 = 1/3.
	    \end{array} \right.
	\end{align}

For the $L^2$ norms of the inhomogeneous terms $\mathfrak{F}_{\vec{\alpha}},$ 
$\mathfrak{G}_{\vec{\alpha}}^{\mu}$ defined in \eqref{E:FalphamathfrakGjalphainhomogeneousterms}, \eqref{E:mathfrakG0alphainhomogeneousterm}, we have the following estimates on $[1,T)$ in the case $\speed^2 = 0:$

\begin{subequations}
\begin{align} 
	\| \mathfrak{F}_{\vec{\alpha}} \|_{L^2} & \lesssim (\omega + 1)e^{-2 \Omega} \fluidnorm{N}, 
		&& (0 \leq |\vec{\alpha}| \leq N - 1), 		
		\label{E:MATHFRAKFALPHAMATHFRAKGALPHAL2DUST} \\
	\left\| \begin{pmatrix}
					\mathfrak{G}_{\vec{\alpha}}^1 \\
					\mathfrak{G}_{\vec{\alpha}}^2 \\
					\mathfrak{G}_{\vec{\alpha}}^3																			
				\end{pmatrix} \right\|_{L^2}
				& \lesssim \mathcal{S}_{N;velocity} \begin{pmatrix}
                       	e^{- 2\Omega}  \\
                        e^{- 2\Omega}  \\
                        e^{- 2\Omega}
     		\end{pmatrix},
				&& (0 \leq |\vec{\alpha}| \leq N), \label{E:MATHFRAKGALPHAMATHFRAKGALPHAL2DUST} \\
	\| \mathfrak{G}_{\vec{\alpha}}^0 \|_{L^2} & \lesssim (\omega + 1) e^{-2 \Omega} \mathcal{S}_{N;velocity}, 
		&& (0 \leq |\vec{\alpha}| \leq N).  \label{E:MATHFRAKG0ALPHAL2DUST}
\end{align}\end{subequations}

Furthermore, we have the following estimates on $[1,T)$ in the cases $0 < \speed^2 < 1/3$
(the case $\speed^2 = 1/3$ is not directly needed for the remaining estimates):

\begin{subequations}
\begin{align} \label{E:MATHFRAKFALPHAMATHFRAKGALPHAL2}
	\left\| \begin{pmatrix}
					\mathfrak{F}_{\vec{\alpha}} \\
					\mathfrak{G}_{\vec{\alpha}}^1 \\
					\mathfrak{G}_{\vec{\alpha}}^2 \\
					\mathfrak{G}_{\vec{\alpha}}^3																			
				\end{pmatrix} \right\|_{L^2}
				& \lesssim \omega \fluidnorm{N}
				\begin{pmatrix}
                        \DecayFunction^{-1}(\Omega) \\
                       	e^{- \Omega} \DecayFunction^{-1}(\Omega)   \\
                        e^{- \Omega} \DecayFunction^{-1}(\Omega)   \\
                        e^{- \Omega} \DecayFunction^{-1}(\Omega)
     		\end{pmatrix}
     		+ \fluidnorm{N} \begin{pmatrix}
                        e^{- \Omega} \\
                       	e^{- 2\Omega} \\
                        e^{- 2\Omega} \\
                        e^{- 2\Omega} 
     		\end{pmatrix}, && (0 \leq |\vec{\alpha}| \leq N), \\
	\| \mathfrak{G}_{\vec{\alpha}}^0 \|_{L^2} & \lesssim \omega \DecayFunction^{-1}(\Omega) \fluidnorm{N}
		+ e^{- \Omega} \fluidnorm{N}, && (0 \leq |\vec{\alpha}| \leq N).  \label{E:MATHFRAKG0ALPHAL2}
\end{align}
\end{subequations}

In the case $\speed^2 = 0,$ for the currents $\dot{J}_{velocity}^{\mu}[\cdot,\cdot]$ and $\dot{J}_{density}^{\mu}[\cdot,\cdot]$
defined in \eqref{E:ENERGYCURRENTDUSTU} - \eqref{E:ENERGYCURRENTDUSTRHO}, we have the following estimates on $[1,T):$

\begin{align}
	\sum_{|\vec{\alpha}| \leq N} 
		\int_{\mathbb{R}^3} \partial_{\mu} \big(\dot{J}_{velocity}^{\mu}[\partial_{\vec{\alpha}} (u^1,u^2,u^3), 
		\partial_{\vec{\alpha}} (u^1,u^2,u^3)]\big) \, d^3 x & \lesssim (\omega + 1) e^{-2\Omega} \mathcal{S}_{N;velocity}^2, 
		\label{E:DIVJDOTL1DUSTVELOCITY} \\
	\sum_{|\vec{\alpha}| \leq N-1} 
		\int_{\mathbb{R}^3} \partial_{\mu} \big(\dot{J}_{density}^{\mu}[\partial_{\vec{\alpha}} \Rlog, 
		\partial_{\vec{\alpha}} \Rlog]\big) \, d^3 x & \lesssim (\omega + 1) e^{-2\Omega} \fluidnorm{N}^2. 
		\label{E:DIVJDOTL1DUSTDENSITY}
\end{align}

In the cases $0 < \speed^2 \leq 1/3,$ for the currents $\dot{J}^{\mu}[\cdot,\cdot]$ defined in \eqref{E:ENERGYCURRENT}, we have the following estimates on $[1,T):$

\begin{align}
	\sum_{|\vec{\alpha}| \leq N} 
		\int_{\mathbb{R}^3} \partial_{\mu} \big(\dot{J}^{\mu}[\partial_{\vec{\alpha}} \mathbf{W}, \partial_{\vec{\alpha}} 
			\mathbf{W}]\big) \, d^3 x 
		& \lesssim \left\lbrace \begin{array}{ll} 
			\omega \DecayFunction^{-1}(\Omega) \fluidnorm{N}^2 + e^{- \Omega} \fluidnorm{N}^2, & 0 < \speed^2 < 1/3, \\
			e^{- \Omega} \fluidnorm{N}^2, & \speed^2 = 1/3.
	\end{array} \right.\label{E:DIVJDOTL1}
\end{align}

\end{proposition}

\begin{remark}
	We have not attempted to derive optimal estimates in Proposition \ref{P:NONLINEARITIES}; we only derived
	estimates that are sufficient to prove our main future stability theorem.
\end{remark}

\begin{proof} 

\noindent \emph{Proofs of \eqref{E:U0UPPERHN} - \eqref{E:TRIANGLEPRIMEJHNMINUSONE}}:
We first consider the cases $0 < \speed^2 < 1/3.$ To prove \eqref{E:U0UPPERHN} we use equation \eqref{E:U0UPPERISOLATED},
Definition \ref{D:NORMS}, Proposition \ref{P:SobolevTaylor} (with $F(v) = \sqrt{1 + v},$ $v = g_{ab} u^a u^b,$ and $\bar{v} = 0$), Proposition \ref{P:F1FKLINFINITYHN}, and Sobolev embedding to conclude that

\begin{align} \label{E:u0HNfirstinequality}
	\| u^0 - 1 \|_{H^N} & \lesssim \| g_{ab} u^a u^b \|_{H^N} 
		\lesssim e^{2 \Omega} \delta_{ab} \| u^a \|_{L^{\infty}} \| u^b \|_{H^N} 
		\lesssim \DecayFunction^{-1}(\Omega) \mathcal{U}_{N-1} \lesssim \DecayFunction^{-1}(\Omega) \fluidnorm{N}, \notag
\end{align}
where $\delta_{ab}$ is the Kronecker delta. The estimates \eqref{E:PINHOMOGENEOUSHN} - \eqref{E:PARTIALVECALPHAUJINHOMOGENEOUSNONPRINCIPALL2} and \eqref{E:TRIANGLEPRIMEHNMINUSTWO} - \eqref{E:TRIANGLEPRIMEJHNMINUSONE} follow similarly (use Proposition \ref{P:SOBOLEVMISSINGDERIVATIVEPROPOSITION} to deduce \eqref{E:PARTIALVECALPHAUJINHOMOGENEOUSNONPRINCIPALL2}). The estimates \eqref{E:PARTIALTRLOGHNMINUSTWO} - \eqref{E:PARTIALTUJUPPERHNMINUSONE} then follow trivially with the help of equations \eqref{E:PARTIALTRLOG} - \eqref{E:PARTIALTUJ}. The cases $\speed^2 = 0, 1/3$ follow similarly. 
\\

\noindent \emph{Proof of \eqref{E:DOTU0L2INTERMSOFDOTUAL2}}:  
To prove \eqref{E:DOTU0L2INTERMSOFDOTUAL2} in the cases $0 < \speed^2 < 1/3,$ we use equation \eqref{E:DOT0INTERMSOFDOTJ},
\eqref{E:U0UPPERHN}, and Sobolev embedding to deduce

\begin{align}
	\| \dot{u}^0 \|_{L^2} & \lesssim \Big\| \frac{1}{u^0} \Big\|_{L^{\infty}} \| u_a \|_{L^{\infty}} \| \dot{u}^a \|_{L^2}
		\lesssim e^{\Omega} \DecayFunction^{-1}(\Omega) \fluidnorm{N} \| \dot{u}^a \|_{L^2}.
\end{align}
The cases $\speed^2 = 0, 1/3$ follow similarly. 
\\

\noindent \emph{Proofs of \eqref{E:MATHFRAKFALPHAMATHFRAKGALPHAL2DUST} - \eqref{E:MATHFRAKG0ALPHAL2}}: We first prove
\eqref{E:MATHFRAKFALPHAMATHFRAKGALPHAL2} (in the cases $0 < \speed^2 < 1/3$). We begin by using equation \eqref{E:FalphamathfrakGjalphainhomogeneousterms} to deduce that

\begin{align} \label{E:FalphamathfrakGjalphainhomogeneoustermsagain}
	\left\| \begin{pmatrix}
					\mathfrak{F}_{\vec{\alpha}} \\
					\mathfrak{G}_{\vec{\alpha}}^1 \\
					\mathfrak{G}_{\vec{\alpha}}^2 \\
					\mathfrak{G}_{\vec{\alpha}}^3																			
				\end{pmatrix} \right\|_{L^2}
	& \leq \big \| \partial_{\vec{\alpha}} \mathbf{b} - \mathbf{b}_{\vec{\alpha}} \big \|_{L^2}
		+ \big \|  A^0 \partial_{\vec{\alpha}} \big[ (A^0)^{-1}\mathbf{b} \big] 
		- \partial_{\vec{\alpha}} \mathbf{b} \big\|_{L^2} \\
	& \ \ + \Big\| A^0 \Big\lbrace(A^0)^{-1}A^a \partial_a \partial_{\vec{\alpha}} \mathbf{W}
 		- \partial_{\vec{\alpha}} \big[(A^0)^{-1} A^a \partial_a \mathbf{W}\big] \Big\rbrace \Big\|_{L^2}. \notag
\end{align}
See the remarks at the end of Section \ref{SS:NORMS} concerning our use of notation for the norms of array-valued functions.

We now deduce the following preliminary estimates for $\mathbf{W} \eqdef (\Rlog, u^1,u^2,u^3)^{T},$ \\
$\mathbf{b}=\omega \Big(- \frac{(1 + \speed^2)}{u^0} g_{ab}u^a u^b, (3 \speed^2 - 2) u^0 u^1, 
(3 \speed^2 - 2) u^0 u^2, (3 \speed^2 - 2) u^0 u^3 \Big)^{T},$ 
$\mathbf{b}_{\vec{\alpha}} \eqdef \omega (3 \speed^2 - 2)u^0 \Big(0, \partial_{\vec{\alpha}}u^1, \partial_{\vec{\alpha}} u^2, \partial_{\vec{\alpha}} u^3 \Big)^T,$ $|\vec{\alpha}| \leq N,$ $A^{\mu},$ and $(A^0)^{-1}$
(explicit expressions for the matrices are provided in \eqref{E:A0def} - \eqref{E:A0inverse}):

\begin{align} 
		& \| \mathbf{b} \|_{H^{N-1}}
		\lesssim  \omega \fluidnorm{N} \begin{pmatrix}
                        \DecayFunction^{-1}(\Omega)  \\
                       	e^{-\Omega} \DecayFunction^{-1}(\Omega)   \\
                        e^{-\Omega} \DecayFunction^{-1}(\Omega)   \\
                        e^{-\Omega} \DecayFunction^{-1}(\Omega) 
                    \end{pmatrix}, \qquad  
      \| \partial_{\vec{\alpha}} \mathbf{b} - \mathbf{b}_{\vec{\alpha}} \|_{L^2}
							\lesssim  \omega \fluidnorm{N} \begin{pmatrix}
                        \DecayFunction^{-1}(\Omega)  \\
                       	e^{-\Omega} \DecayFunction^{-1}(\Omega) \\
                        e^{-\Omega} \DecayFunction^{-1}(\Omega) \\
                        e^{-\Omega} \DecayFunction^{-1}(\Omega) 
                    \end{pmatrix}, \label{E:BCOMMUTATORESTIMATE}
\end{align}

\begin{align}
	\| A^{0} \|_{L^{\infty}}
	& \lesssim \begin{pmatrix}
                        1  & e^{\Omega}\DecayFunction^{-1}(\Omega)\fluidnorm{N} & e^{\Omega}\DecayFunction^{-1}(\Omega)\fluidnorm{N} & e^{\Omega}\DecayFunction^{-1}(\Omega)\fluidnorm{N} \\
                       	e^{-\Omega} \DecayFunction^{-1}(\Omega)\fluidnorm{N} & 1 & 0 & 0 \\
                        e^{-\Omega} \DecayFunction^{-1}(\Omega)\fluidnorm{N} & 0 & 1 & 0 \\
                        e^{-\Omega} \DecayFunction^{-1}(\Omega)\fluidnorm{N} & 0 & 0 & 1 \\
                    \end{pmatrix},
\end{align}

\begin{align}
	\big\| (A^{0})^{-1} \big\|_{L^{\infty}}
	& \lesssim \begin{pmatrix}
                        1  & e^{\Omega}\DecayFunction^{-1}(\Omega)\fluidnorm{N} & e^{\Omega}\DecayFunction^{-1}(\Omega)\fluidnorm{N} & e^{\Omega}\DecayFunction^{-1}(\Omega)\fluidnorm{N} \\
                       	e^{-\Omega} \DecayFunction^{-1}(\Omega)\fluidnorm{N} & 1 & \DecayFunction^{-1}(\Omega)\fluidnorm{N} & \DecayFunction^{-1}(\Omega)\fluidnorm{N} \\
                        e^{-\Omega} \DecayFunction^{-1}(\Omega)\fluidnorm{N} & \DecayFunction^{-1}(\Omega)\fluidnorm{N} & 1 & \DecayFunction^{-1}(\Omega)\fluidnorm{N} \\
                        e^{-\Omega} \DecayFunction^{-1}(\Omega)\fluidnorm{N} & \DecayFunction^{-1}(\Omega)\fluidnorm{N} & \DecayFunction^{-1}(\Omega)\fluidnorm{N} & 1 \\
                    \end{pmatrix},
\end{align}

\begin{align}
	\big\| \underpartial (A^{0})^{-1} \big\|_{H^{N-1}}
	& \lesssim \fluidnorm{N} \begin{pmatrix}
                        \DecayFunction^{-1}(\Omega)  & e^{\Omega} & e^{\Omega} & e^{\Omega} \\
                       	e^{- \Omega} & \DecayFunction^{-1}(\Omega) & \DecayFunction^{-1}(\Omega) & \DecayFunction^{-1}(\Omega) \\
                        e^{- \Omega} & \DecayFunction^{-1}(\Omega) & \DecayFunction^{-1}(\Omega) & \DecayFunction^{-1}(\Omega)\\
                        e^{- \Omega} & \DecayFunction^{-1}(\Omega) & \DecayFunction^{-1}(\Omega) & \DecayFunction^{-1}(\Omega) \\
                    \end{pmatrix},
\end{align}

\begin{align} 
            \| \mathbf{W} \|_{H^N}
							\lesssim \fluidnorm{N} \begin{pmatrix}
                        1 \\
                       	e^{- \Omega}   \\
                        e^{- \Omega}   \\
                        e^{- \Omega} 
                    \end{pmatrix},
\end{align}

\begin{align}
	\| A^1 \|_{L^{\infty}} & \lesssim \begin{pmatrix}
                        e^{- \Omega} \DecayFunction^{-1}(\Omega) \fluidnorm{N} & 1  & 0 & 0 \\
                       	e^{- 2\Omega} & e^{- \Omega} \DecayFunction^{-1}(\Omega) \fluidnorm{N} & 0 & 0 \\
                        e^{- 2\Omega} & 0 & e^{- \Omega} \DecayFunction^{-1}(\Omega) \fluidnorm{N}  & 0 \\
                        e^{- 2\Omega} & 0 & 0 & e^{- \Omega} \DecayFunction^{-1}(\Omega) \fluidnorm{N}  \\
                    \end{pmatrix},
\end{align}

\begin{align}
	\| \underpartial A^1 \|_{H^{N-1}} & \lesssim \fluidnorm{N} \begin{pmatrix}
                        e^{- \Omega}  & 0  & 0 & 0 \\
                       	e^{- 2 \Omega} \DecayFunction^{-1}(\Omega) & e^{- \Omega} & 0 & 0 \\
                        e^{- 2 \Omega} \DecayFunction^{-1}(\Omega) & 0 & e^{- \Omega} & 0 \\
                        e^{- 2 \Omega} \DecayFunction^{-1}(\Omega) & 0 & 0 & e^{- \Omega}  \\
                    \end{pmatrix},
\end{align}
and analogously for the matrices $A^2,$ $A^3.$ The above preliminary estimates follow from repeated applications of 
Proposition \ref{P:DIFFERENTIATEDSOBOLEVCOMPOSITION}, Proposition \ref{P:F1FKLINFINITYHN}, the definition \eqref{E:FLUIDNORMDEF} of $\fluidnorm{N},$ and the estimates \eqref{E:U0UPPERHN} - \eqref{E:PARTIALVECALPHAUJINHOMOGENEOUSNONPRINCIPALL2}.

We now estimate the second term on the right-hand side of \eqref{E:FalphamathfrakGjalphainhomogeneoustermsagain}. To
this end, we use Propositions \ref{P:F1FKLINFINITYHN} and \ref{P:SOBOLEVMISSINGDERIVATIVEPROPOSITION}, together with the above estimates and Sobolev embedding to conclude that for $0 \leq |\vec{\alpha}| \leq N$

\begin{align} \label{E:SECONDTERMFALPHAMATHFRAKGJALPHAL2AGAIN}
	\Big\| A^0 & \partial_{\vec{\alpha}} \big[ (A^0)^{-1}\mathbf{b} \big] - \partial_{\vec{\alpha}} \mathbf{b} \Big\|_{L^2} 
		\lesssim  \| A^{0} \|_{L^{\infty}} * \big\| \underpartial (A^0)^{-1} \big\|_{H^{N-1}} 
		* \| \mathbf{b} \|_{H^{N-1}} \\
	& \lesssim \omega \fluidnorm{N} \begin{pmatrix}
                        \DecayFunction^{-1}(\Omega) \\
                       	e^{- \Omega} \DecayFunction^{-1}(\Omega)  \\
                        e^{- \Omega} \DecayFunction^{-1}(\Omega)  \\
                        e^{- \Omega} \DecayFunction^{-1}(\Omega)
                    \end{pmatrix}, \notag
\end{align}
where we write $*$ to indicate that we are performing matrix multiplication on the matrices of norms.

We similarly estimate the third term on the right-hand side of \eqref{E:FalphamathfrakGjalphainhomogeneoustermsagain} (using Proposition \ref{P:derivativesofF1FkL2} to estimate 
$ \| \underpartial [(A^0)^{-1}A^a] \|_{H^{N-1}} \lesssim
\| (A^0)^{-1} \|_{L^{\infty}} * \| \underpartial A^a \|_{H^{N-1}} + \| \underpartial (A^0)^{-1} \|_{H^{N-1}} * \| A^a \|_{L^{\infty}}$), thus arriving at the following bound:

\begin{align}  \label{E:THIRDTERMFALPHAMATHFRAKGJALPHAL2AGAIN}
		& \Big\| A^0 \Big\lbrace(A^0)^{-1}A^a \partial_a \partial_{\vec{\alpha}} \mathbf{W}
 		- \partial_{\vec{\alpha}} \big[(A^0)^{-1} A^a \partial_a \mathbf{W}\big] \Big\rbrace \Big\|_{L^2} \\
 		& \lesssim  \| A^0 \|_{L^{\infty}} * 
 			\Big\lbrace \| (A^0)^{-1} \|_{L^{\infty}} * \| \underpartial A^a \|_{H^{N-1}}
 			+ \| \underpartial (A^0)^{-1} \|_{H^{N-1}} * \| A^a \|_{L^{\infty}}
 			\Big\rbrace * \| \partial_a \mathbf{W} \|_{H^{N-1}} \notag \\
 		& \lesssim \fluidnorm{N} \begin{pmatrix}
                        e^{-\Omega} \\
                       	e^{-2\Omega}   \\
                        e^{-2\Omega}   \\
                        e^{-2\Omega} 
                    \end{pmatrix}. \notag 
\end{align}

Finally, adding the second estimate in \eqref{E:BCOMMUTATORESTIMATE} and \eqref{E:SECONDTERMFALPHAMATHFRAKGJALPHAL2AGAIN} -
\eqref{E:THIRDTERMFALPHAMATHFRAKGJALPHAL2AGAIN}, we deduce \eqref{E:MATHFRAKFALPHAMATHFRAKGALPHAL2}. The proofs of \eqref{E:MATHFRAKFALPHAMATHFRAKGALPHAL2DUST} - \eqref{E:MATHFRAKGALPHAMATHFRAKGALPHAL2DUST} (for $\speed^2 = 0$) are similar, and we omit the details. 

To prove \eqref{E:MATHFRAKG0ALPHAL2}, we use equation \eqref{E:mathfrakG0alphainhomogeneousterm},
\eqref{E:U0UPPERHN}, \eqref{E:PARTIALTU0UPPERHNMINUSONE} - \eqref{E:PARTIALTUJUPPERHNMINUSONE}, and \eqref{E:MATHFRAKFALPHAMATHFRAKGALPHAL2} to conclude that

\begin{align} \label{E:mathfrakG0alphainhomogeneoustermagain}
		\| \mathfrak{G}_{\vec{\alpha}}^0 \|_{L^2} 
		& \leq g_{ab} \Big\| u^{\nu} \partial_{\nu} \Big(\frac{u^a}{u^0}\Big) \Big\|_{L^{\infty}} \| u^b \|_{H^N}
		+ \omega \Big\| \frac{(3 \speed^2 - 2 + 2 u^0)u_a}{u^0} \Big\|_{L^{\infty}} \| u^a \|_{H^N} \\
		& \ \ + \Big\| \frac{1}{u^0} \Big \|_{L^{\infty}} \| u_a \|_{L^{\infty}} \| \mathfrak{G}_{\vec{\alpha}}^a \|_{L^2} 
		\notag \\
		& \lesssim \omega \DecayFunction^{-1}(\Omega) \fluidnorm{N} + e^{- \Omega} \fluidnorm{N}. \notag
\end{align}
This completes the proof of \eqref{E:MATHFRAKG0ALPHAL2}. The proof of \eqref{E:MATHFRAKG0ALPHAL2DUST} (for $\speed^2 = 0$) is similar, and we omit the details. 
\\

\noindent \emph{Proofs of \eqref{E:DIVJDOTL1DUSTDENSITY} - \eqref{E:DIVJDOTL1}}: To prove \eqref{E:DIVJDOTL1}
in the cases $0 < \speed^2 < 1/3,$ we first recall equation \eqref{E:DIVDOTJ}, where $\dot{\mathbf{W}} = (\dot{\Rlog}, \dot{u}^1, \dot{u}^2, \dot{u}^3)^T \eqdef \partial_{\vec{\alpha}}\mathbf{W} = (\partial_{\vec{\alpha}} \Rlog, \partial_{\vec{\alpha}} u^1,\partial_{\vec{\alpha}} u^2,\partial_{\vec{\alpha}} u^3)^T,$ and as in \eqref{E:DOT0INTERMSOFDOTJ}, $\dot{u}^0 \eqdef \frac{1}{u^0} u_a \dot{u}^a:$

\begin{align} \label{E:DIVDOTJAGAIN}
	\partial_{\mu} \big(\dot{J}^{\mu}[\dot{\mathbf{W}}, \dot{\mathbf{W}}] \big) 
		& = \frac{\speed^2 (\partial_{\mu} u^{\mu})}{(1 + \speed^2)}\dot{\Rlog}^2 
			+ (1 + \speed^2)(\partial_{\mu} u^{\mu}) 
	\Big(-(\dot{u}^{0})^2 + g_{ab} \dot{u}^{a} \dot{u}^{b} \Big) \\
	& \ \ + 2 \speed^2 g_{ab} \bigg(\partial_t \Big[\frac{u^a}{u^0} \Big] \bigg)\dot{u}^b \dot{\Rlog} 
		+ 4 \speed^2 \omega  \frac{g_{ab} u^a \dot{u}^b}{u^0} \dot{\Rlog} \notag \\
	& \ \ + \underbrace{2(1+\speed^2)(3 \speed^2 - 1) \omega g_{ab}\dot{u}^a 	
		\dot{u}^b}_{\mbox{$\leq 0$ if $\speed^2 \leq 1/3$}} \notag \\
	& \ \ + \frac{2 \speed^2 \mathfrak{F}}{(1 + \speed^2)} \dot{\Rlog}
		- 2(1+\speed^2) \mathfrak{G}^0 \dot{u}^0
		+ 2(1+\speed^2) g_{ab} \mathfrak{G}^a \dot{u}^b. \notag
\end{align}

Above, the terms $\mathfrak{F} \eqdef \mathfrak{F}_{\vec{\alpha}}$ and $\mathfrak{G} \eqdef \mathfrak{G}_{\vec{\alpha}}^{\mu},$ are defined in \eqref{E:FalphamathfrakGjalphainhomogeneousterms} and \eqref{E:mathfrakG0alphainhomogeneousterm}. We now deduce \eqref{E:DIVJDOTL1} using the following four steps: \textbf{i)} we ignore the non-positive term $2(1+\speed^2)(3 \speed^2 - 1) \omega g_{ab}\dot{u}^a \dot{u}^b$ on the right-hand side of \eqref{E:DIVDOTJAGAIN}; \textbf{ii)} we bound each variation 
$\dot{\Rlog}, \dot{u}^{\mu}$ and $\mathfrak{F}_{\vec{\alpha}},$ $\mathfrak{G}_{\vec{\alpha}}^{\mu}$ in $L^2,$ and use the estimates \eqref{E:DOTU0L2INTERMSOFDOTUAL2}, \eqref{E:MATHFRAKFALPHAMATHFRAKGALPHAL2}, and \eqref{E:MATHFRAKG0ALPHAL2}; \textbf{iii)} we bound all of the remaining terms in $L^{\infty}$ and use Sobolev embedding together with the estimates \eqref{E:U0UPPERHN} and \eqref{E:PARTIALTU0UPPERHNMINUSONE} - \eqref{E:PARTIALTUJUPPERHNMINUSONE}; \textbf{iv)} we make repeated use of the estimate $\| v_1 v_2 v_3 \|_{L^1} \leq \| v_1 \|_{L^{\infty}} \| v_2 \|_{L^2} \| v_3 \|_{L^2},$ where $v_2$ and $v_3$ are terms estimated in step \textbf{ii)}, and $v_1$ is estimated in step \textbf{iii)}.

The proofs of \eqref{E:DIVJDOTL1DUSTVELOCITY} - \eqref{E:DIVJDOTL1DUSTDENSITY} (in the case $\speed^2 = 0$) are similar but simpler (use the expressions \eqref{E:DIVDOTJ0SPEEDVELOCITY} - \eqref{E:DIVDOTJ0SPEEDDENSITY}). The proof of \eqref{E:DIVJDOTL1} in the case $\speed^2 = 1/3$ is provided in the next paragraph. 
\\

\noindent \emph{Details for the special case $\speed^2 = 1/3$}: We now address the case $\speed^2 = 1/3.$ Our goal is to show that in this case, many of the estimates \eqref{E:U0UPPERHN} - \eqref{E:DIVJDOTL1} are valid
without the $\omega \cdots$ terms on the right-hand side. The fact that these estimates hold without these terms is a consequence of special algebraic cancellation. This cancellation, which is connected to the conformal invariance of the relativistic Euler equations when $\speed^2 = 1/3,$ is discussed in Section \ref{S:PURERADIATION} in more detail. We will only demonstrate the cancellation for the estimate \eqref{E:DIVJDOTL1}; this is the only estimate needed to derive the crucial inequality \eqref{E:ENINTEGRAL} below in the case $\speed^2 = 1/3.$ In Section \ref{S:PURERADIATION}, we will show that when $\speed^2 = 1/3,$ $(\rho,u)$ verify the relativistic Euler equations \eqref{E:EULERINTROP} - \eqref{E:EULERINTROU} corresponding to the metric $g$ if and only if the rescaled variables $(\rho',U) \eqdef (e^{4 \Omega}\rho,e^{\Omega}u)$ verify the relativistic Euler equations \eqref{E:EULERINTROP} - \eqref{E:EULERINTROU} corresponding to the Minkowski metric $m = e^{-2 \Omega} g.$ Now if we consider the change-of-time-variable $\frac{d \tau}{d t} = e^{-\Omega(t)},$ $\tau(t=1) = 1,$ then \eqref{E:CONFORMALMETRICFORM} shows that we have $m = -d \tau^2 + \sum_{j=1}^3 (dx^j)^2.$ Therefore, it follows that $\Rlog \eqdef \ln\big(e^{4 \Omega(t \circ \tau)} \rho/\bar{\rho} \big),$ $U^j \eqdef e^{\Omega(t \circ \tau)}u^j,$ verify the 
relativistic Euler equations \eqref{E:FINALEULERP} - \eqref{E:FINALEULERUJ} corresponding to the Minkowski metric $m = -d \tau^2 + \sum_{j=1}^3 (dx^j)^2$ and the coordinates $(\tau,x^1,x^2,x^3).$ In particular, in these new fluid and time variables, the right-hand sides of \eqref{E:FINALEULERP} - \eqref{E:FINALEULERUJ} vanish, and $U^{\tau} = (1 + \delta_{ab}U^aU^b)^{1/2} = u^0.$ Here, $\delta_{ab}$ is the Kronecker delta and the index ``$\tau$'' is meant to emphasize that we are referring to the time component of $U$ relative to the new time coordinate system (while $u^0$ refers to the time component of $u$ relative to the original ``$t$'' coordinate system). Furthermore, relative to these new variables, we can define a current analogous to the current $\dot{J}^{\mu}$
defined in \eqref{E:ENERGYCURRENT}:

\begin{subequations}
\begin{align} \label{E:ENERGYCURRENTRESCALED}
	\dot{\mathscr{J}}^{\tau} & \eqdef \frac{\speed^2 U^{\tau}}{(1 + \speed^2)}\dot{\Rlog}^2 
		+ 2 \speed^2 \dot{U}^{\tau} \dot{\Rlog} 
		+ (1 + \speed^2)U^{\tau} [-(\dot{U}^{\tau})^2 + \delta_{ab} \dot{U}^{a} \dot{U}^{b}], \\
	\dot{\mathscr{J}}^{j} & \eqdef \frac{\speed^2 U^{j}}{(1 + \speed^2)}\dot{\Rlog}^2 + 2 \speed^2 \dot{U}^{j} \dot{\Rlog} 
		+ (1 + \speed^2) U^{j} [-(\dot{U}^{\tau})^2 + \delta_{ab}\dot{U}^{a} \dot{U}^{b}],
\end{align}
\end{subequations}
where $\dot{U}^j = e^{\Omega} \dot{u}^j,$ $\dot{U}^{\tau} = \frac{1}{U^{\tau}} \delta_{ab} U^a \dot{U}^b
= \frac{1}{u^0} g_{ab} u^a \dot{u}^b = \dot{u}^0.$ Note that $\dot{\mathscr{J}}^{\tau} = \dot{J}^0,$  
$\dot{\mathscr{J}}^{j} = e^{\Omega} \dot{J}^j,$ and  

\begin{align} \label{E:CURRENTRELATIONS}
	\partial_{\tau} \dot{\mathscr{J}}^{\tau} + \partial_{a} \dot{\mathscr{J}}^{a}
	= e^{\Omega} \big(\partial_t \dot{J}^0 + \partial_a \dot{J}^a \big) \eqdef e^{\Omega} \partial_{\mu} \dot{J}^{\mu}.
\end{align}

Setting $\boldsymbol{\mathcal{W}} \eqdef (\Rlog,U^1,U^2,U^3),$ using the vanishing of the right-hand sides of \eqref{E:FINALEULERP} - \eqref{E:FINALEULERUJ} in the new coordinate + fluid variables, and noting definition \eqref{E:FLUIDNORMDEF}, our prior proof of \eqref{E:DIVJDOTL1} (under the assumption that the metric is the
standard Minkowski metric) yields 

\begin{align} \label{E:RESCALEDDIVERGENCEBOUND}
	\sum_{|\vec{\alpha}| \leq N} 
		& \int_{\mathbb{R}^3} \partial_{\tau} \big(\dot{\mathscr{J}}^{\tau}[\partial_{\vec{\alpha}} \boldsymbol{\mathcal{W}}, 
			\partial_{\vec{\alpha}} \boldsymbol{\mathcal{W}}]\big)
			+ \partial_{a} \big(\dot{\mathscr{J}}^{a}[\partial_{\vec{\alpha}} \boldsymbol{\mathcal{W}}, \partial_{\vec{\alpha}} 
			\boldsymbol{\mathcal{W}}]\big) \, d^3 x \\
		& \lesssim \| \Rlog \|_{H^N}^2 + \sum_{j=1}^3 \| U^j \|_{H^N}^2. \notag
\end{align}
Using \eqref{E:CURRENTRELATIONS}, \eqref{E:RESCALEDDIVERGENCEBOUND}, and definition \eqref{E:FLUIDNORMDEF}, we deduce 

\begin{align} \label{E:RESCALEDESTIMATE}
	e^{\Omega} \sum_{|\vec{\alpha}| \leq N} & \int_{\mathbb{R}^3} \partial_{\mu}
		\big(\dot{J}^{\mu}[\partial_{\vec{\alpha}} \mathbf{W}, \partial_{\vec{\alpha}} \mathbf{W}]\big) \, d^3 x \\
	& = \sum_{|\vec{\alpha}| \leq N} \int_{\mathbb{R}^3} 
			\partial_{\tau}\big(\dot{\mathscr{J}}^{\tau}[\partial_{\vec{\alpha}} \boldsymbol{\mathcal{W}}, 
			\partial_{\vec{\alpha}} \boldsymbol{\mathcal{W}}]\big) 
		+ \partial_a \big(\dot{\mathscr{J}}^a[\partial_{\vec{\alpha}} \boldsymbol{\mathcal{W}}, 
			\partial_{\vec{\alpha}} \boldsymbol{\mathcal{W}}]\big) \, d^3 x \notag \\
	& \lesssim \| \Rlog \|_{H^N}^2 + \sum_{j=1}^3 \| U^j \|_{H^N}^2 \eqdef \| \Rlog \|_{H^N}^2 
		+ e^{\Omega(t)} \sum_{j=1}^3 \| u^j \|_{H^N}^2 \eqdef \fluidnorm{N}^2, \notag
\end{align}
which is the desired estimate.

\end{proof}

\section{The Energy Norm vs. Sobolev Norms Comparison Proposition} \label{S:NORMVSENERGY}

In this section, we prove a proposition that compares the coercive properties of the norms $\mathcal{U}_{N-1},$ $\fluidnorm{N},$ and $\mathcal{S}_{N;velocity}$ to the coercive properties of the energies $\fluidenergy{N},$ $\mathcal{E}_{N;velocity},$ and $\mathcal{E}_{N-1;density}.$

\begin{proposition}[\textbf{Equivalence of Sobolev Norms and Energy Norms}] \label{P:ENERGYNORMCOMPARISON} 
Let $N \geq 3$ be an integer, and assume that $\fluidnorm{N}(t) \leq \epsilon$ on 
$[1,T).$ Then there exist implicit constants such that if $\epsilon$ is sufficiently small, then the following estimates hold for the norms $\mathcal{U}_{N-1},$ $\fluidnorm{N},$ and $\mathcal{S}_{N;velocity}$ defined in Definition \ref{D:NORMS} and the energies $\fluidenergy{N},$ $\mathcal{E}_{N;velocity},$ and $\mathcal{E}_{N-1;density}$ defined in Definition \ref{D:ENERGY}:
	
	\begin{align} \label{E:NORMENERGYCOMPARISON}
	\begin{array}{ll} 
	 	\mathcal{E}_{N;velocity}^2 + \mathcal{E}_{N-1;density}^2 \approx \fluidnorm{N}^2, & \speed^2 = 0, \\
	 	\mathcal{E}_{N;velocity}^2  \approx \mathcal{S}_{N;velocity}^2, & \speed^2 = 0, \\
	 	\mathcal{U}_{N-1}^2 + \fluidenergy{N}^2 \approx \fluidnorm{N}^2, & 0 < \speed^2 < 1/3, \\
	  \fluidenergy{N}^2 \approx \fluidnorm{N}^2, & \speed^2 = 1/3.
	  \end{array}
	\end{align}
	
\end{proposition}

\begin{proof}
In the cases $0 < \speed^2 < 1/3,$ to prove $\mathcal{U}_{N-1}^2 + \fluidenergy{N}^2 \approx \fluidnorm{N}^2,$ we only need to show 

\begin{align} \label{E:JDOT0TOPESTIMATE}
	\sum_{|\vec{\alpha}| \leq N} & \Big\lbrace \| \partial_{\vec{\alpha}} \Rlog  \|_{L^2}^2 
		+ e^{2 \Omega} \sum_{a=1}^3 \| \partial_{\vec{\alpha}} u^a \|_{L^2}^2 \Big\rbrace 
	\approx \sum_{|\vec{\alpha}| \leq N} 
		\int_{\mathbb{R}^3} \dot{J}^0[\partial_{\vec{\alpha}} \mathbf{W},\partial_{\vec{\alpha}} \mathbf{W}] \, d^3 x.
\end{align}
The desired inequalities would then follow easily from \eqref{E:JDOT0TOPESTIMATE} and Definitions \ref{D:NORMS} and \ref{D:ENERGY}. To prove \eqref{E:JDOT0TOPESTIMATE}, for notational convenience, we set $\mathbf{W} \eqdef (\Rlog, u^1,u^2,u^3)^T,$ $\dot{u}^j \eqdef \partial_{\vec{\alpha}} u^j,$ $\dot{\Rlog} \eqdef \partial_{\vec{\alpha}}\Rlog,$ and as in \eqref{E:DOT0INTERMSOFDOTJ}, we define $\dot{u}^0 \eqdef \frac{1}{u^0}u_a \dot{u}^a.$ We now recall definition \eqref{E:ENERGYCURRENT}:

\begin{align}
	\dot{J}^0 & \eqdef \frac{\speed^2 u^0}{(1 + \speed^2)}\dot{\Rlog}^2 + 2 \speed^2  \dot{u}^0 \dot{\Rlog} 
		+ (1 + \speed^2) u^0 g_{\alpha \beta} \dot{u}^{\alpha} \dot{u}^{\beta}.
\end{align}
Using $0 < \speed^2 \leq \frac{1}{3},$ \eqref{E:U0UPPERHN}, Sobolev embedding, and the simple estimate $|2 \speed^2 \dot{u}^0 \dot{\Rlog}| \leq \frac{\speed^2}{2(1 + \speed^2)}\dot{\Rlog}^2 + 2\speed^2(1 + \speed^2) \Big(\frac{1}{u^0} u_a \dot{u}^a \Big)^2,$ it follows that

\begin{align} \label{E:JDOT0INEQUALITY}
	 \dot{J}^0  & \geq \frac{\speed^2 (u^0 - 1/2)}{(1 + \speed^2)} \dot{\Rlog}^2 
		+ (1 + \speed^2)u^0 g_{ab}\dot{u}^a \dot{u}^b 
		- (1 + \speed^2)(u^0 + 2 \speed^2) \Big(\frac{1}{u^0} u_a \dot{u}^a \Big)^2 \\
	& \geq C_1 (\dot{\Rlog}^2 + e^{2 \Omega} \delta_{ab} \dot{u}^a \dot{u}^b) 
		- C_2 \epsilon^2 e^{2\Omega} \DecayFunction^{-2}(\Omega) \delta_{ab} \dot{u}^a \dot{u}^b 
		\gtrsim \dot{\Rlog}^2 + e^{2 \Omega} \delta_{ab} \dot{u}^a \dot{u}^b. \notag
\end{align}
A reverse inequality can similarly be shown. Integrating these inequalities over $\mathbb{R}^3$ and
summing over all derivatives $\partial_{\vec{\alpha}} \mathbf{W}$ with $|\vec{\alpha}| \leq N,$ we deduce \eqref{E:JDOT0TOPESTIMATE}. The remaining estimates in \eqref{E:NORMENERGYCOMPARISON} (in the cases $\speed^2 = 0, 1/3$)
can be proved similarly.

\end{proof}

\section{Integral Inequalities for the Energy and Norms} \label{S:INTEGRALINEQUALITY}

In this section, we use the Sobolev estimates of Section \ref{S:SOBOLEV} to derive energy and norm integral inequalities. These inequalities form the backbone of our future stability proof.

\begin{proposition}[\textbf{Integral Inequalities}] \label{P:INTEGRALENERGYINEQUALITIES}

Let $(\Rlog,u^1,u^2,u^3)$ be a classical solution to the relativistic Euler equations \eqref{E:FINALEULERP} - \eqref{E:FINALEULERUJ} on the spacetime slab $[1,T) \times \mathbb{R}^3.$ Let $\mathcal{U}_{N-1},$ $\fluidnorm{N},$ and $\mathcal{S}_{N;velocity}$ be the norms in Definition \ref{D:NORMS}, and let $\fluidenergy{N},$ $\mathcal{E}_{N;velocity},$ and $\mathcal{E}_{N-1;density}$ be the energies in Definition \ref{D:ENERGY}. Let $N \geq 3$ be an integer, and assume that $\fluidnorm{N}(t) \leq \epsilon$ on $[1,T).$ If $\speed^2 = 0$ and $\epsilon$ is sufficiently small, then the following integral inequalities are verified for $1 \leq t_1 \leq t < T:$	
	
	\begin{subequations}  
	\begin{align} \label{E:ENINTEGRALDUSTVELOCITY}
		\mathcal{E}_{N;velocity}^2(t) & \leq \mathcal{E}_{N;velocity}^2(t_1)
			+ C \int_{s = t_1}^t \big\lbrace \omega(s) + 1 \big\rbrace e^{-2 \Omega(s)} \mathcal{S}_{N;velocity}^2(s) \, d s, 
	\end{align}
	
	\begin{align} \label{E:ENINTEGRALDUSTDENSITY} 
		\mathcal{E}_{N-1;density}^2(t) & \leq \mathcal{E}_{N-1;density}^2(t_1) 
			+ C \int_{s = t_1}^t \big\lbrace \omega(s) + 1 \big\rbrace e^{-2 \Omega(s)} \mathcal{S}_{N}^2(s) \, d s. 
	\end{align}
	\end{subequations}

If $0 < \speed^2 \leq 1/3$ and $\epsilon$ is sufficiently small, then the following integral inequality is verified for $1 \leq t_1 \leq t < T:$	
	
\begin{align} \label{E:ENINTEGRAL} 
		\fluidenergy{N}^2(t) 
		& \leq \fluidenergy{N}^2(t_1) 
			+ C \int_{s = t_1}^t \underbrace{\omega(s) \DecayFunction^{-1}(\Omega(s)) \fluidnorm{N}^2(s)}_{\mbox{absent if 
			$\speed^2 = 1/3$}}
			+ e^{- \Omega(s)} \fluidnorm{N}^2(s) \, d s. 
\end{align}
	
If $0 < \speed^2 < 1/3$ and $\epsilon$ is sufficiently small, then the following integral inequality is verified for $1 \leq t_1 \leq t < T:$

\begin{align} \label{E:UNMINUSONEINTEGRAL} 
		\mathcal{U}_{N-1}^2(t) 
		& \leq \mathcal{U}_{N-1}^2(t_1)  
		+ 2 \int_{s = t_1}^t \overbrace{\Big\lbrace 3 \speed^2 - 1 + 
			\frac{\DecayFunction'(\Omega(s))}{\DecayFunction(\Omega(s))}}^{
			\mbox{$\leq 0 $ for large $\Omega$}} \Big\rbrace \omega(s) \mathcal{U}_{N-1}^2(s) \, ds   \\
		& \ \ + C \int_{s = t_1}^t \omega(s) \DecayFunction^{-1}(\Omega(s)) \mathcal{U}_{N-1}^2(s)
			+ e^{- \Omega(s)}\DecayFunction(\Omega(s)) \fluidnorm{N}^2(s) \, d s. \notag
	\end{align}

\end{proposition}

\begin{proof}
	To prove \eqref{E:UNMINUSONEINTEGRAL}, we first use \eqref{E:UNORMTIMEDERIVATIVE} and \eqref{E:TRIANGLEPRIMEJHNMINUSONE} to 
	deduce that
	
	\begin{align} \label{E:Unormtimederivativeagain}
		\frac{d}{dt}\big(\mathcal{U}_{N-1}^2 \big) & \leq 2 \overbrace{\Big\lbrace 3 \speed^2 - 1 + 	
			\frac{\DecayFunction'(\Omega)}{\DecayFunction(\Omega)}}^{\mbox{$\leq 0 $ for large $\Omega$}} \Big\rbrace \omega 
			\mathcal{U}_{N-1}^2
			+ 2 e^{ \Omega} \DecayFunction(\Omega) \mathcal{U}_{N-1} \sum_{a=1}^3 \| \triangle'^a \|_{H^{N-1}}  \\
		& \leq 2 \overbrace{\Big\lbrace 3 \speed^2 - 1 + 	
			\frac{\DecayFunction'(\Omega)}{\DecayFunction(\Omega)}}^{\mbox{$\leq 0 $ for large $\Omega$}} \Big\rbrace \omega 
			\mathcal{U}_{N-1}^2
			+ C \omega \DecayFunction^{-1}(\Omega) \mathcal{U}_{N-1}^2
			+ C e^{-\Omega}\DecayFunction(\Omega) \fluidnorm{N}^2. \notag
	\end{align}
	Integrating \eqref{E:Unormtimederivativeagain} from $t_1$ to $t,$ we deduce \eqref{E:UNMINUSONEINTEGRAL}.
	
	Inequalities \eqref{E:ENINTEGRALDUSTVELOCITY} - \eqref{E:ENINTEGRALDUSTDENSITY} 
	and \eqref{E:ENINTEGRAL} follow similarly from \eqref{E:FLUIDENERGYTIMEDERIVATIVE0SPEEDVELOCITY} - 
	\eqref{E:FLUIDENERGYTIMEDERIVATIVE0SPEEDDENSITY}, \eqref{E:ENTIMEDERIVATIVE} and
	\eqref{E:DIVJDOTL1DUSTVELOCITY} - \eqref{E:DIVJDOTL1DUSTDENSITY}, \eqref{E:DIVJDOTL1}.
	
\end{proof}

\section{The Future Stability Theorem} \label{S:GLOBALEXISTENCE}

In this section, we state and prove our future stability theorem. The proof is through a standard bootstrap argument based on the continuation principle (Proposition \ref{P:CONTINUATION}) and the integral inequalities of Proposition \ref{P:INTEGRALENERGYINEQUALITIES}.

\begin{theorem}[\textbf{Future Stability of the Explicit Fluid Solutions}] \label{T:GLOBALEXISTENCE}
Assume that $0 \leq \speed^2 \leq 1/3,$ that the metric scale factor verifies the hypotheses \ref{A:A1} - \ref{A:A3} from Section \ref{S:INTRO}, and that $N \geq 3$ is an integer. Let $\mathring{\Rlog} \eqdef \Rlog |_{t=1} \eqdef \ln \big(\rho/\bar{\rho}\big)|_{t=1}$ $\mathring{u}^j \eqdef u^j|_{t=1}$ $(j=1,2,3)$ be initial data for the relativistic Euler equations \eqref{E:FINALEULERP} - \eqref{E:FINALEULERUJ}, and let $\fluidnorm{N}(t)$ be the norm defined in Definition \ref{D:NORMS}. There exist a small constant $\epsilon_0$ with $0 < \epsilon_0 < 1$ and a large constant $C_*$ such that if $\epsilon \leq \epsilon_0$ and $\fluidnorm{N}(1) = C_*^{-1} \epsilon,$ then the classical solution $\Big(\Rlog = \ln \big(e^{3(1+\speed^2)\Omega} \rho/\bar{\rho} \big), u \Big)$ provided by Theorem \ref{T:LOCAL} exists 
for $(t,x^1,x^2,x^3) \in [1,\infty) \times \mathbb{R}^3$ and furthermore, the following estimate holds for $t \geq 1:$

\begin{align} 
	\fluidnorm{N}(t) & \leq \epsilon. \label{E:FLUIDNORMLESSTHANEPSILON}
\end{align}
In addition, $T_{max} = \infty,$ where $T_{max}$ is the time from the hypotheses of Proposition \ref{P:CONTINUATION}.

\end{theorem}

\begin{proof}
We provide full details in the cases $0 < \speed^2 \leq 1/3;$ the remaining case $\speed^2 = 0$ is similar and simpler, thanks to the fact the evolution of the $u^{j}$ decouples from that of $\Rlog.$ By Theorem \ref{T:LOCAL}, if $C_*$ is sufficiently large and $\epsilon$ is sufficiently small, then there exists a non-trivial spacetime slab $[1,T_{local}) \times \mathbb{R}^3$ upon which a classical solution exists and upon which the following estimate holds:  

\begin{align} \label{E:FLUIDNORMNPROOFBA}
	\fluidnorm{N}(t) & \leq \epsilon. 
\end{align}
We define

\begin{align}
	T \eqdef \sup \big\lbrace T_{local}\geq 1 \ | \ \mbox{The solution exists on} \ [1,T_{local}) \ \mbox{and} \ 
		\eqref{E:FLUIDNORMNPROOFBA} \ \mbox{holds on} \ [1,T_{local}) \big\rbrace.
\end{align}
We will show that if $C_*$ is sufficiently large and $\epsilon$ is sufficiently small, then $T = \infty.$ Throughout the
remainder of the proof, we use the notation

\begin{align} \label{E:DATAEPSILONDEF}
	\mathring{\epsilon} & \eqdef \fluidnorm{N}(1) = C_*^{-1} \epsilon.
\end{align}

We first use the bootstrap assumption \eqref{E:FLUIDNORMNPROOFBA} (under the assumption that $\epsilon$ is sufficiently small), Proposition \ref{P:ENERGYNORMCOMPARISON}, the inequalities \eqref{E:ENINTEGRAL} - \eqref{E:UNMINUSONEINTEGRAL}, and the hypotheses \ref{A:A1} - \ref{A:A3} from Section \ref{S:INTRO} to derive the following inequality, which is valid for all sufficiently large $t_1$ and $t$ verifying $t_1 \leq t < T:$

\begin{align} \label{E:FLUIDNORMT1GRONWALLREADY}
	\fluidnorm{N}^2(t) & \leq C \fluidnorm{N}(t_1) 
		+ C \epsilon^2 \int_{s = t_1}^t \underbrace{\omega \DecayFunction^{-1}(\Omega)}_{\mbox{absent if 
		$\speed^2 = 1/3$}} + e^{- \Omega} \DecayFunction(\Omega) \, d s. 
\end{align}
Now our hypotheses \ref{A:A1} - \ref{A:A3} on $e^{\Omega}$ imply that the integrand on the right-hand 
side of \eqref{E:FLUIDNORMT1GRONWALLREADY} is integrable in $s$ over the interval $s \in [1,\infty).$ Therefore, if $t_1 < T$ and $t_1$ is sufficiently large, it follows that 

\begin{align} \label{E:GLOBALBOUNDALMOSTFINAL}
	\fluidnorm{N}^2(t) & < C \fluidnorm{N}(t_1) + \frac{1}{2}\epsilon^2, && t \in [t_1,T).
\end{align}
It remains to derive a suitable bound for $\fluidnorm{N}(t_1);$ this will be a standard local-existence-type estimate. To this end, for this fixed value of $t_1,$ we again use the bootstrap assumption \eqref{E:FLUIDNORMNPROOFBA}, Proposition \ref{P:ENERGYNORMCOMPARISON}, and \eqref{E:ENINTEGRAL} - \eqref{E:UNMINUSONEINTEGRAL} to derive 

\begin{align} \label{E:WEAKINEQUALITY}
	\fluidnorm{N}^2(t) & \leq C \mathring{\epsilon}^2 + c(t_1) \int_{s = 1}^t \fluidnorm{N}^2(s) \, d s,
	&& t \in [1,t_1).
\end{align}
Applying Gronwall's inequality to \eqref{E:WEAKINEQUALITY}, we deduce that

\begin{align} \label{E:WEAK}
	\fluidnorm{N}^2(t) & \leq C \mathring{\epsilon}^2 e^{c(t_1) t}, && t \in [1,t_1).
\end{align}
Furthermore, we note that by Proposition \ref{P:CONTINUATION}, Sobolev embedding, and the continuity of $\fluidnorm{N}(t),$ it follows from \eqref{E:WEAK} and definition \eqref{E:DATAEPSILONDEF} that if $\epsilon$ is sufficiently small and $C_*$ is sufficiently large, then

\begin{align} \label{E:TWEAK}
	t_1 < T.
\end{align}	
Combining \eqref{E:GLOBALBOUNDALMOSTFINAL} and \eqref{E:WEAK} and using definition \eqref{E:DATAEPSILONDEF}, we deduce that 

\begin{align} \label{E:NEARFINALINEQUALITY}
	\fluidnorm{N}^2(t) & \leq C e^{c(t_1)} \frac{\epsilon^2}{C_*^2} + \frac{1}{2}\epsilon^2, && t \in [1,T).
\end{align}
Therefore, if $\epsilon$ is sufficiently small and $C_*$ is sufficiently large, it follows from \eqref{E:NEARFINALINEQUALITY} that

\begin{align} \label{E:GLOBALBOUNDFINAL}
	\fluidnorm{N}^2(t) & < \epsilon^2, && t \in [1,T).
\end{align}
Note that \eqref{E:GLOBALBOUNDFINAL} is a strict improvement over the bootstrap assumption \eqref{E:FLUIDNORMNPROOFBA}. Thus, using Proposition \ref{P:CONTINUATION}, Sobolev embedding, and the continuity of $\fluidnorm{N}(t),$ it follows that $T = \infty$ and that \eqref{E:GLOBALBOUNDFINAL} holds for $t \in [1,\infty).$ 

The case $\speed^2 = 0$ can be handled similarly using inequalities \eqref{E:ENINTEGRALDUSTVELOCITY} - \eqref{E:ENINTEGRALDUSTDENSITY}.

\end{proof}

\section{The Sharpness of the Hypotheses for the Radiation Equation of State} \label{S:PURERADIATION}

In this section, we show that if $\speed^2 = 1/3,$ hypothesis \ref{A:A1} from Section \ref{S:INTRO} holds, and
$\int_{s = 1}^{\infty} e^{- \Omega(s)} \, d s = \infty,$ then the background solution $\widetilde{\rho} = \bar{\rho} e^{-4 \Omega},$ $\widetilde{u}^{\mu} = \delta_0^{\mu}$ to the relativistic Euler equations \eqref{E:EULERINTROP} - \eqref{E:EULERINTROU} on the spacetime-with-boundary $([1,\infty) \times \mathbb{R}^3,g)$ is \emph{nonlinearly unstable}. The main idea of the proof is to use the special conformally invariant structure of the fluid equations when $\speed^2 = 1/3$ to reduce the problem to the case in which the spacetime is Minkowskian; we can then quote Christodoulou's results \cite{dC2007} to deduce the instability. We begin by noting a standard result: that the change of time variable

\begin{align} \label{E:CONFORMALTIME}
	\frac{d \tau}{dt}= e^{-\Omega(t)}
\end{align}
allows us to write the metric \eqref{E:METRICFORM} in the following form:

\begin{align} \label{E:CONFORMALMETRICFORM}
	g & = e^{2\Omega(t \circ \tau)} \Big(-d\tau^2 +  \sum_{j=1}^3 (dx^j)^2 \Big) = e^{2\Omega(t \circ \tau)}m,
\end{align}
where $m = -d \tau^2 + \sum_{i=j}^3 (dx^j)^2$ is the Minkowski metric on 
$[1,\infty) \times \mathbb{R}^3$ equipped with standard rectangular coordinates. The following elementary but important observation play a fundamental role in the discussion in this section: if the hypothesis \ref{A:A1} is verified but
$\int_{s = 1}^{\infty} e^{- \Omega(s)} \, d s = \infty,$ then after translating $\tau$ by a constant if necessary, it follows that

\begin{align}
	\tau: [1,\infty) \rightarrow [1,\infty),\qquad t \rightarrow \tau(t)
\end{align}	
is an autodiffeomorphism of $[1,\infty).$

We now prove the conformal invariance of the fluid equations when $\speed^2 = 1/3;$ this is a standard result, and 
we provide the short proof only for the convenience of the reader.

\begin{proposition} [\textbf{Conformal Invariance of the Relativistic Euler Equations When} $\speed^2 = 1/3$] \label{P:CONFORMAL}
	Let $g$ be the conformally flat metric defined in \eqref{E:CONFORMALMETRICFORM}. Then $(\rho,u)$ is a solution 
	to the relativistic Euler equations \eqref{E:EULERINTROP} - \eqref{E:EULERINTROU} corresponding to the metric $g$ if and only 
	if $(\rho',U)$ is a solution to the relativistic Euler equations 
	\eqref{E:EULERINTROP} - \eqref{E:EULERINTROU} corresponding to the Minkowski metric $m.$ Here, the rescaled fluid variables 
	$(\rho',U)$ are defined by
	
	\begin{align}
		\rho' & \eqdef e^{4 \Omega} \rho, \qquad  U \eqdef e^{\Omega} u. \label{E:CONFORMALU}
	\end{align}
\end{proposition}

\begin{proof}
Throughout this proof, $\nabla$ denotes the Levi-Civita connection corresponding to the Minkowski metric $m$ 
and $D$ denotes the Levi-Civita connection corresponding to $g = e^{2 \Omega} m.$ As discussed in the beginning of the article, the relativistic Euler equations \eqref{E:EULERINTROP} - \eqref{E:EULERINTROU} are equivalent to the equations $D_{\alpha} T^{\alpha \mu} = 0$ plus the normalization condition $g_{\alpha \beta} u^{\alpha} u^{\beta} = - 1$ (see \eqref{E:DIVT0} and \eqref{E:UNORMALIZED}). In the case $p = 1/3 \speed^2 \rho,$ we note that

\begin{align}
	T^{\mu \nu} & = \frac{4}{3} \rho u^{\mu} u^{\nu} + \frac{1}{3}\rho (g^{-1})^{\mu \nu} = \frac{4}{3}\rho u^{\mu} u^{\nu} 
	+ \frac{1}{3} \rho e^{-2 \Omega}( m^{-1})^{\mu \nu}, 
\end{align}
and we define the following rescaled ``Minkowskian'' energy-momentum tensor:

\begin{align} 
	T_{(m)}^{\mu \nu} & \eqdef e^{6 \Omega} T^{\mu \nu} = \frac{4}{3} \rho' U^{\mu} U^{\nu} + \frac{1}{3} \rho'(m^{-1})^{\mu \nu},
\end{align}
where $\rho'$ and $U^{\mu}$ are defined in \eqref{E:CONFORMALU}. We also note that by
\eqref{E:CONFORMALMETRICFORM},

\begin{align}
	m_{\alpha \beta} U^{\alpha} U^{\beta} = - 1.
\end{align}

The key step is the following identity, whose simple proof we omit:

\begin{align} \label{E:DIVDTINTERMSOFDIVNABLAT}
	D_{\alpha} T^{\alpha \mu} & = e^{-6 \Omega} \nabla_{\alpha} (e^{6 \Omega} T^{\alpha \mu}) = 
		e^{-6 \Omega} \nabla_{\alpha} T_{(m)}^{\alpha \mu}.
\end{align}
We remark that the proof of the identity \eqref{E:DIVDTINTERMSOFDIVNABLAT} heavily leans on the fact that 
$g_{\alpha \beta} T^{\alpha \beta} = 0$ for the equation of state $p = (1/3) \rho.$ It now follows from \eqref{E:DIVDTINTERMSOFDIVNABLAT} that

\begin{align}
	D_{\alpha} T^{\alpha \mu} \iff \nabla_{\alpha} T_{(m)}^{\alpha \mu} = 0.
\end{align}
According to the remarks made at the beginning of the proof, this completes the proof of the proposition.

\end{proof}

In the next theorem, we will recall some important aspects of Christodoulou's shock formation result \cite{dC2007}. To this end, we need to introduce some notation. In order to emphasize the connections between Proposition \ref{P:CONFORMAL}, Theorem \ref{T:CHRISTODOULOUSHOCK}, and Corollary \ref{C:SHOCKSCANFORM}, we will denote the energy density by $\rho',$ the four velocity by $U,$ and the spacetime coordinates by $(\tau,x^1,x^2,x^3).$ Let $(\mathring{\rho}',\mathring{U}^1,\mathring{U}^2,\mathring{U}^3)$ be initial data (at time $\tau = 1$) for the relativistic Euler equations \eqref{E:EULERINTROP} - \eqref{E:EULERINTROU} on the spacetime-with-boundary $\big([1,\infty) \times \mathbb{R}^3, m \big),$ where $m = -d\tau^2 +  \sum_{j=1}^3 (dx^j)^2$ is the standard Minkowski metric on $\mathbb{R}^4.$ Let $\bar{\rho} > 0$ be a constant background density. Let $B_r \subset \mathbb{R}^3 $ denote a solid ball of radius $\frac{2}{3} \leq r < 1$ centered at the origin in the Cauchy hypersurface $\lbrace (\tau,x^1,x^2,x^3) \ | \ \tau = 1 \rbrace \simeq \mathbb{R}^3$ (embedded in Minkowski spacetime) and let $\partial B_r$ denote its boundary. We define the 
order-$1$ Sobolev norm $\mathcal{S}_{B_1 \backslash B_r} \geq 0$ of the data over the annular region in between $B_r$ and the unit ball as follows:
	
	\begin{align} \label{E:ANNULARNORM}
		\mathcal{S}_{B_1 \backslash B_r} \eqdef \|\mathring{\rho}' - \bar{\rho} \|_{H^1(B_1 \backslash B_r)}
		+ \sum_{j = 1}^3 \| \mathring{U}^j \|_{H^1(B_1 \backslash B_r)}. 
	\end{align}
	We define the standard Sobolev norm of the data as follows:
	
	\begin{align}
		\mathcal{D}_M & \eqdef 
	  	\| \mathring{\rho}' - \bar{\rho} \|_{H^M}
	  	+ \sum_{j=1}^3 \| \mathring{U}^j \|_{H^M}. \label{E:STANDARDNORM}	
	\end{align}
	We also define the following surface + annular integrals of the data:
	
	\begin{align}
		Q(r) & \eqdef \int_{\partial B_r} r \Big\lbrace (\mathring{\rho}' - \bar{\rho}) + \frac{4}{\sqrt{3}} \bar{\rho} \mathring{U}^i \hat{N}_i 	
			\Big\rbrace \, d \sigma  
		+ \int_{(B_1 \backslash B_r)} 2(\mathring{\rho}' - \bar{\rho}) + \frac{4}{\sqrt{3}} \bar{\rho} \mathring{U}^i \hat{N}_i \, 
		d^3 x. 
	\end{align}
	Above, $r \eqdef \sqrt{\sum_{j=1}^3 (x^j)^2}$ denotes the standard radial coordinate on $\mathbb{R}^3,$ and
	$\hat{N}$ denotes the outward unit normal to $\partial B_r.$

\begin{theorem} [\cite{dC2007} \textbf{Theorem 14.2}, pg. 925] \label{T:CHRISTODOULOUSHOCK}
	Consider the relativistic Euler equations \eqref{E:EULERINTROP} - \eqref{E:EULERINTROU} on the manifold-with-boundary	
	$[1,\infty) \times \mathbb{R}^3$ equipped with the Minkowski metric and rectangular coordinates
	$(\tau,x^1,x^2,x^3):$ $m = - d \tau^2 + \sum_{j=1}^3 (dx^j)^2.$ Let $\bar{\rho} > 0$ be a constant background density.
	Then the explicit solution $\widetilde{\rho}' = \bar{\rho},
	(\widetilde{U}^{\tau},\widetilde{U}^1,\widetilde{U}^2,\widetilde{U}^3) = (1,0,0,0)$ 
	is unstable. More specifically, there exists an open 
	family of arbitrarily small (non-zero) smooth perturbations of the explicit solution's data which 
	launch solutions that form shocks in finite time. Here, $\widetilde{U}^{\tau}$ denotes the
	time component of $\widetilde{U}$ relative to the coordinate system $(\tau,x^1,x^2,x^3).$
	
	Even more specifically, there exists a large integer $M,$ a large constant $C > 0,$ and a small constant $\epsilon > 0$ such 
	that the following three conditions on the initial data together guarantee finite-time shock formation
	in the corresponding solution:
	
	\begin{subequations}
	\begin{align}
		\mathcal{D}_M & \leq \epsilon, \label{E:SMALLTOPNORM} \\
		Q(r) & \geq C \sqrt{\mathcal{D}_M} \big\lbrace \sqrt{\mathcal{D}_M} + \sqrt{1 - r} \big\rbrace 
			\mathcal{S}_{(B_1 \backslash B_r)}, \label{E:SIGNEDINTEGRALS} \\
		\frac{2}{3} & \leq r < 1. \label{E:RBOUND}
	\end{align}
	\end{subequations}
	
	Furthermore, there exists a uniform constant $C'>0$ such that the above three conditions guarantee that 
	a shock forms in the solution before the time
	
	\begin{align}
		\tau_{max}(r) = \exp \left(\frac{C'(1 - r)}{Q(r)} \right). 
	\end{align}
\end{theorem}

\begin{remark}
	Note that if we are given any data such that $Q(r) > 0,$ then if we rescale its amplitude (more precisely, the amplitude
	of its deviation from the background constant solution) by a sufficiently small constant 
	factor, then \eqref{E:SMALLTOPNORM} and \eqref{E:SIGNEDINTEGRALS} are both verified by the rescaled data.
	This follows from the fact that for a fixed $r,$ $\mathcal{D}_M,$ $\mathcal{S}_{(B_1 \backslash B_r)},$
	and $Q(r)$ shrink linearly with the scaling factor, while the right-hand side of \eqref{E:SIGNEDINTEGRALS} shrinks 
	like the scaling factor to the power $3/2.$
\end{remark}

The following corollary follows easily from \eqref{E:CONFORMALTIME} - \eqref{E:CONFORMALMETRICFORM}, Proposition \ref{P:CONFORMAL}, and Theorem \ref{T:CHRISTODOULOUSHOCK}.

\begin{corollary} [\textbf{Nonlinear Instability of the Explicit Solutions When $\speed^2= 1/3$}] \label{C:SHOCKSCANFORM}
	Assume that $\speed^2 = 1/3.$ Consider the relativistic Euler equations \eqref{E:EULERINTROP} - \eqref{E:EULERINTROU} on the 
	manifold-with-boundary	$[1,\infty) \times \mathbb{R}^3$ equipped with a Lorentzian metric $g = -dt^2 + e^{2\Omega(t)} 
	\sum_{j=1}^3 (dx^j)^2,$ $\Omega(1) = 0$ (as in \eqref{E:METRICFORM}). 
	Assume that hypothesis \ref{A:A1} of Section \ref{S:INTRO} holds, and that $\int_{s = 1}^{\infty} e^{- \Omega(s)} \, d s = 
	\infty.$ Then the explicit solution \eqref{E:BACKGROUNDU} is unstable. More specifically, there exist 
	arbitrarily small (non-zero) perturbations of the explicit solution's initial data which launch perturbed solutions that form 
	shocks in finite time. More precisely, if the open conditions \eqref{E:SMALLTOPNORM} - \eqref{E:RBOUND} are verified by the 
	data, then a shock will form before the conformal $\tau-$coordinate time
	
	\begin{align} \label{E:SHOCKTAU}
		\tau_{max}(r) = \exp \left(\frac{C'(1 - r)}{Q(r)} \right). 
	\end{align}
	In the original time coordinate $t,$ the shock will form before $t_{max}(r),$ which is implicitly determined
	in terms of $\tau_{max}(r)$ via the relation
	
	\begin{align}
		\tau_{max}(r) = \int_1^{t_{max}(r)} e^{- \Omega(s)} \, ds.
	\end{align}	
	
\end{corollary}

\begin{remark} \label{R:UNSTABLEPROOF}
	Unlike the proof of Theorem \ref{T:GLOBALEXISTENCE}, the proof of Corollary \ref{C:SHOCKSCANFORM} is not 
	easily seen to be stable under small perturbations of the metric $g.$ In order to show that the proof is stable,
	one would need to show that Christodoulou's proof of Theorem \ref{T:CHRISTODOULOUSHOCK} is stable under small
	perturbations of the Minkowski metric.
\end{remark}

\section*{Acknowledgments}

This research was supported in part by a Solomon Buchsbaum grant administered by the Massachusetts Institute of Technology, and
in part by an NSF All-Institutes Postdoctoral Fellowship administered by the Mathematical Sciences Research Institute through its core grant DMS-0441170. I would like to thank John Barrow for offering his many insights, which propelled me to closely investigate the cases $\speed^2 = 0$ and $\speed^2 = 1/3.$

\begin{center}
	\textbf{\huge{Appendices}}
\end{center}
\setcounter{section}{0}
   \setcounter{subsection}{0}
   \setcounter{subsubsection}{0}
   \setcounter{paragraph}{0}
   \setcounter{subparagraph}{0}
   \setcounter{figure}{0}
   \setcounter{table}{0}
   \setcounter{theorem}{0}
   \setcounter{definition}{0}
   \setcounter{remark}{0}
   \setcounter{proposition}{0}
   \renewcommand{\thesection}{\Alph{section}}
   \renewcommand{\theequation}{\Alph{section}.\arabic{equation}}
   \renewcommand{\theproposition}{\Alph{section}-\arabic{proposition}}
   \renewcommand{\thecorollary}{\Alph{section}.\arabic{corollary}}
   \renewcommand{\thedefinition}{\Alph{section}.\arabic{definition}}
   \renewcommand{\thetheorem}{\Alph{section}.\arabic{theorem}}
   \renewcommand{\theremark}{\Alph{section}.\arabic{remark}}
   \renewcommand{\thelemma}{\Alph{section}-\arabic{lemma}}

\section{Sobolev-Moser Inequalities} \label{A:SobolevMoser}
		In this Appendix, we provide some Sobolev-Moser estimates. The propositions and corollaries stated below
		are standard results that can be proved using methods similar to those used in \cite[Chapter 6]{lH1997} and in 
		the Appendix of \cite{sKaM1981}. Throughout this Appendix, $L^p = L^p(\mathbb{R}^3)$ and 
		$H^M=H^M(\mathbb{R}^3).$
	
\begin{proposition} \label{P:derivativesofF1FkL2}
	Let $M \geq 0$ be an integer. If $\lbrace v_a \rbrace_{1 \leq a \leq l}$ are functions such that $v_a \in
    L^{\infty}, \|\underpartial^{(M)} v_a \|_{L^2} < \infty$ for $1 \leq a \leq l,$ and
	$\vec{\alpha}_1, \cdots, \vec{\alpha}_l$ are spatial derivative multi-indices with 
	$|\vec{\alpha}_1| + \cdots + |\vec{\alpha}_l| = M,$ then
	
	\begin{align}
		\| (\partial_{\vec{\alpha}_1}v_1) (\partial_{\vec{\alpha}_2}v_2) \cdots (\partial_{\vec{\alpha}_l}v_l)\|_{L^2}
		& \leq C(l,M) \sum_{a=1}^l \Big( \| \underpartial^{(M)} v_a  \|_{L^2} \prod_{b \neq a} \|v_{b} \|_{L^{\infty}} \Big).
	\end{align}
\end{proposition}

\begin{proposition}                                                  \label{P:DIFFERENTIATEDSOBOLEVCOMPOSITION}
    Let $M \geq 1$ be an integer, let $\mathfrak{K}$ be a compact set, and let $F \in C_b^M(\mathfrak{K})$ be a 
    function. Assume that $v$ is a function such that $v(\mathbb{R}^3) \subset \mathfrak{K}$ and $ \underpartial v \in H^{M-1}.$
    Then $\underpartial (F \circ v) \in H^{M-1},$ and
    
    \begin{align} 														\label{E:DifferentiatedModifiedSobolevEstimate}
    	\| \underpartial (F \circ v) \|_{H^{M-1}} 
    		& \leq C(M) \| \underpartial v \|_{H^{M-1}} \sum_{l=1}^M |F^{(l)}|_{\mathfrak{K}} 
    		\| v \|_{L^{\infty}}^{l - 1}.
    \end{align}
\end{proposition}

\begin{proposition}                   \label{P:SobolevTaylor}
     Let $M \geq 1$ be an integer, let $\mathfrak{K}$ be a compact, convex set, and let $F \in C_b^M(\mathfrak{K})$ be a 
     function. Assume that $v$ is a function such that $v(\mathbb{R}^3) \subset \mathfrak{K}$ and $v - \bar{v} \in H^M,$
    where $\bar{v} \in \mathfrak{K}$ is a constant. Then $F \circ v - F \circ \bar{v} \in H^M,$ and
    
    \begin{align} 														\label{E:ModifiedSobolevEstimateConstantArray}
    	\|F \circ v - F \circ \bar{v} \|_{H^M} 
    		\leq C(M) \Big\lbrace |F^{(1)}|_{\mathfrak{K}}\| v - \bar{v} \|_{L^2} 
    		+ \| \underpartial v \|_{H^{M-1}} \sum_{l=1}^M  |F^{(l)}|_{\mathfrak{K}} 
    		\| v \|_{L^{\infty}}^{l - 1} \Big\rbrace.
    \end{align}
\end{proposition}

\begin{proposition} \label{P:F1FKLINFINITYHN}
	Let $M \geq 1, l \geq 2$ be integers. Suppose that $\lbrace v_a \rbrace_{1 \leq a \leq l}$ are functions such that $v_a \in
    L^{\infty}$ for $1 \leq a \leq l,$ that $v_l \in H^M,$ and that
	$\underpartial v_a \in H^{M-1}$ for $1 \leq a \leq l - 1.$
	Then
	
	\begin{align}
		\| v_1 v_2 \cdots v_l \|_{H^M} \leq C(l,M) \Big\lbrace \| v_l \|_{H^M} \prod_{a=1}^{l-1} \| v_a \|_{L^{\infty}}  
		+ \sum_{a=1}^{l-1} \| \underpartial v_a \|_{H^{M-1}} \prod_{b \neq a} \| v_b \|_{L^{\infty}} \Big\rbrace.
	\end{align}	
\end{proposition}

\begin{remark}
	The significance of this proposition is that only one of the functions, namely $v_l,$ is estimated in $L^2.$
\end{remark}

\begin{proposition}                                                                             \label{P:SOBOLEVMISSINGDERIVATIVEPROPOSITION}
    Let $M \geq 1$ be an integer, let $\mathfrak{K}$ be a compact, convex set, and
    let $F \in C_b^M(\mathfrak{K})$ be a function.
    Assume that $v_1$ is a function such that $v_1(\mathbb{R}^3) \subset \mathfrak{K},$ that $\underpartial v_1 \in 
    L^{\infty},$ and that $\underpartial^{(M)} v_1 \in L^2.$ Assume that $v_2 \in L^{\infty},$ that $\underpartial^{(M-1)} v_2 \in L^2,$ 
    and let $\vec{\alpha}$ be a spatial derivative multi-index with with $|\vec{\alpha}| = M.$ Then 
    $\partial_{\vec{\alpha}} \left((F \circ v_1 )v_2\right) - (F \circ v_1)\partial_{\vec{\alpha}} v_2 \in L^2,$ 
    and
        
 			\begin{align}       \label{E:SobolevMissingDerivativeProposition}
      	\|\partial_{\vec{\alpha}} & \left((F \circ v_1 )v_2\right) - (F \circ v_1)\partial_{\vec{\alpha}} v_2\|_{L^2} \notag \\
        	& \leq C(M) \Big\lbrace |F^{(1)}|_{\mathfrak{K}} \|\underpartial v_1 \|_{L^{\infty}} \| \underpartial^{(M-1)} v_2 \|_{L^2} 
					+ \| v_2 \|_{L^{\infty}} \| \underpartial v_1 \|_{H^{M-1}} \sum_{l=1}^M |F^{(l)}|_{\mathfrak{K}} 
    			\| v_1 \|_{L^{\infty}}^{l - 1} \Big\rbrace.
       \end{align}
\end{proposition}

\bibliographystyle{amsalpha}
\bibliography{JBib}

\newcommand{\etalchar}[1]{$^{#1}$}
\def\cprime{$'$}
\providecommand{\bysame}{\leavevmode\hbox to3em{\hrulefill}\thinspace}
\providecommand{\MR}{\relax\ifhmode\unskip\space\fi MR }
\providecommand{\MRhref}[2]{%
  \href{http://www.ams.org/mathscinet-getitem?mr=#1}{#2}
}
\providecommand{\href}[2]{#2}
\begin{thebibliography}{RFC{\etalchar{+}}98}

\bibitem[Chr00]{dC2000}
Demetrios Christodoulou, \emph{The action principle and partial differential
  equations}, Annals of Mathematics Studies, vol. 146, Princeton University
  Press, Princeton, NJ, 2000. \MR{1739321 (2003a:58001)}

\bibitem[Chr07]{dC2007}
\bysame, \emph{The formation of shocks in 3-dimensional fluids}, EMS Monographs
  in Mathematics, European Mathematical Society (EMS), Z\"urich, 2007.
  \MR{2284927 (2008e:76104)}

\bibitem[EBW92]{gEjBsW1992}
G.~Efstathiou, J.~R. Bond, and S.~D.~M. White, \emph{{COBE background radiation
  anisotropies and large-scale structure in the universe}}, Royal Astronomical
  Society, Monthly Notices \textbf{258} (1992), 1P--6P.

\bibitem[Fri02]{hF2002}
Helmut Friedrich, \emph{Conformal {E}instein evolution}, The conformal
  structure of space-time, Lecture Notes in Phys., vol. 604, Springer, Berlin,
  2002, pp.~1--50. \MR{2007040 (2005b:83005)}

\bibitem[H{\"o}r97]{lH1997}
Lars H{\"o}rmander, \emph{Lectures on nonlinear hyperbolic differential
  equations}, Math\'ematiques \& Applications (Berlin) [Mathematics \&
  Applications], vol.~26, Springer-Verlag, Berlin, 1997. \MR{MR1466700
  (98e:35103)}

\bibitem[Hsi97]{lHs1997}
L.~Hsiao, \emph{Quasilinear hyperbolic systems and dissipative mechanisms},
  World Scientific Publishing Co. Inc., River Edge, NJ, 1997. \MR{1640089
  (99i:35096)}

\bibitem[KM81]{sKaM1981}
Sergiu Klainerman and Andrew Majda, \emph{Singular limits of quasilinear
  hyperbolic systems with large parameters and the incompressible limit of
  compressible fluids}, Comm. Pure Appl. Math. \textbf{34} (1981), no.~4,
  481--524. \MR{615627 (84d:35089)}

\bibitem[LV11]{cLjK2011}
C.~{L{\"u}bbe} and J.~A. {Valiente Kroon}, \emph{{A conformal approach for the
  analysis of the non-linear stability of pure radiation cosmologies}}, ArXiv
  e-prints (2011).

\bibitem[Nis78]{tN1978}
Takaaki Nishida, \emph{Nonlinear hyperbolic equations and related topics in
  fluid dynamics}, D\'epartement de Math\'ematique, Universit\'e de Paris-Sud,
  Orsay, 1978, Publications Math{\'e}matiques d'Orsay, No. 78-02. \MR{0481578
  (58 \#1690)}

\bibitem[Noe71]{eN1918}
Emmy Noether, \emph{Invariant variation problems}, Transport Theory Statist.
  Phys. \textbf{1} (1971), no.~3, 186--207, Translated from the German (Nachr.
  Akad. Wiss. G{\"o}ttingen Math.-Phys. Kl. II 1918, 235--257). \MR{0406752 (53
  \#10538)}

\bibitem[PR99]{aRsP1999}
S.~{Perlmutter} and A.~{Riess}, \emph{{Cosmological parameters from supernovae:
  Two groups' results agree}}, COSMO-98 ({D.~O.~Caldwell}, ed.), American
  Institute of Physics Conference Series, vol. 478, July 1999, pp.~129--142.

\bibitem[Ren04a]{aR2004a}
Alan~D. Rendall, \emph{Accelerated cosmological expansion due to a scalar field
  whose potential has a positive lower bound}, Classical Quantum Gravity
  \textbf{21} (2004), no.~9, 2445--2454. \MR{2056321 (2005a:83146)}

\bibitem[Ren04b]{aR2004b}
\bysame, \emph{Asymptotics of solutions of the {E}instein equations with
  positive cosmological constant}, Ann. Henri Poincar\'e \textbf{5} (2004),
  no.~6, 1041--1064. \MR{2105316 (2005h:83017)}

\bibitem[Ren05a]{aR2005a}
\bysame, \emph{Intermediate inflation and the slow-roll approximation},
  Classical Quantum Gravity \textbf{22} (2005), no.~9, 1655--1666.
  \MR{MR2136118 (2005m:83164)}

\bibitem[Ren05b]{aR2005b}
\bysame, \emph{Theorems on existence and global dynamics for the {E}instein
  equations}, Living Reviews in Relativity \textbf{8} (2005), no.~6.

\bibitem[Ren06]{aR2006b}
\bysame, \emph{Mathematical properties of cosmological models with accelerated
  expansion}, Analytical and numerical approaches to mathematical relativity,
  Lecture Notes in Phys., vol. 692, Springer, Berlin, 2006, pp.~141--155.
  \MR{2222551 (2007f:83117)}

\bibitem[RFC{\etalchar{+}}98]{aRaF1998}
Adam~G. Riess, Alexei~V. Filippenko, Peter Challis, Alejandro Clocchiattia,
  Alan Diercks, Peter~M. Garnavich, Ron~L. Gilliland, Craig~J. Hogan, Saurabh
  Jha, Robert~P. Kirshner, B.~Leibundgut, M.~M. Phillips, David Reiss, Brian~P.
  Schmidt, Robert~A. Schommer, Chris~R. Smith, J.~Spyromilio, Christopher
  Stubbs, Nicholas~B. Suntzeff, and John Tonry, \emph{Observational evidence
  from supernovae for an accelerating universe and a cosmological constant}.

\bibitem[Rin08]{hR2008}
Hans Ringstr{\"o}m, \emph{Future stability of the {E}instein-non-linear scalar
  field system}, Invent. Math. \textbf{173} (2008), no.~1, 123--208.
  \MR{2403395 (2009c:53097)}

\bibitem[Rin09]{hR2009}
\bysame, \emph{Power law inflation}, Comm. Math. Phys. \textbf{290} (2009),
  no.~1, 155--218. \MR{MR2520511 (2010g:58024)}

\bibitem[RS09]{iRjS2009}
I.~{Rodnianski} and J.~{Speck}, \emph{{The stability of the irrotational
  {Euler-Einstein} system with a positive cosmological constant}}, arXiv
  preprint: http://arxiv.org/abs/0911.5501 (2009), 1--70.

\bibitem[Sid85]{tS1985}
Thomas~C. Sideris, \emph{Formation of singularities in three-dimensional
  compressible fluids}, Comm. Math. Phys. \textbf{101} (1985), no.~4, 475--485.
  \MR{815196 (87d:35127)}

\bibitem[Spe09a]{jS2008b}
Jared Speck, \emph{The non-relativistic limit of the {E}uler-{N}ordstr\"om
  system with cosmological constant}, Rev. Math. Phys. \textbf{21} (2009),
  no.~7, 821--876. \MR{MR2553428}

\bibitem[Spe09b]{jS2008a}
\bysame, \emph{Well-posedness for the {E}uler-{N}ordstr\"om system with
  cosmological constant}, J. Hyperbolic Differ. Equ. \textbf{6} (2009), no.~2,
  313--358. \MR{MR2543324}

\bibitem[Spe11]{jS2011}
\bysame, \emph{{The nonlinear future-stability of the FLRW family of solutions
  to the Euler-Einstein system with a positive cosmological constant}}, arXiv
  preprint: http://arxiv.org/abs/1102.1501 (2011).

\bibitem[SS11]{jSrS2011}
Jared Speck and Robert Strain, \emph{Hilbert expansion from the {Boltzmann}
  equation to relativistic fluids}, Communications in Mathematical Physics
  (2011), 1--52.

\bibitem[Wal84]{rW1984}
Robert~M. Wald, \emph{General relativity}, University of Chicago Press,
  Chicago, IL, 1984. \MR{MR757180 (86a:83001)}

\bibitem[Wei08]{sW2008}
Steven Weinberg, \emph{Cosmology}, Oxford University Press, Oxford, 2008.
  \MR{2410479 (2009g:83182)}

\bibitem[WLR99]{kKsWoLmR1999}
Kelvin K.~S. Wu, Ofer Lahav, and Martin~J. Rees, \emph{The large-scale
  smoothness of the universe}, Nature \textbf{397} (1999), 225.

\bibitem[YBPS05]{jYsBbPtS2005}
Jaswant Yadav, Somnath Bharadwaj, Biswajit Pandey, and T.~R. Seshadri,
  \emph{Testing homogeneity on large scales in the sloan digital sky survey
  data release one}, Monthly Notices of the Royal Astronomical Society
  \textbf{364} (2005), no.~2, 601--606.

\end{thebibliography}

\ \\

\end{document}